\documentclass[12pt]{amsart}
\usepackage{srcltx,bm}

\usepackage{varioref}
\usepackage{hyperref}       
\hypersetup{       
   pdftex,                  
   colorlinks=true,        
   pdfstartview=FitH,      
   allcolors=[rgb]{0,0.3,0.6},        
   bookmarks=true,          
	 bookmarksdepth=section, 
   bookmarksopen=false,     
   bookmarksnumbered=true, 
	 backref, pagebackref   
}

\allowdisplaybreaks

\DeclareMathOperator{\End}{End}
\DeclareMathOperator{\Hom}{Hom}
\DeclareMathOperator{\ev}{ev}
\DeclareMathOperator{\coev}{coev}

\DeclareMathOperator{\Gr}{Gr}
\DeclareMathOperator{\Irr}{Irr}

\newcommand\be            {\begin{equation}}
\newcommand\ee            {\end{equation}}

\theoremstyle{plain}

\newtheorem{theorem}{Theorem}
\newtheorem{lemma}[theorem]{Lemma}
\newtheorem{proposition}[theorem]{Proposition}
\newtheorem{corollary}[theorem]{Corollary}

\theoremstyle{definition}

\newtheorem{remark}[theorem]{Remark}
\newtheorem{definition}[theorem]{Definition}

\numberwithin{equation}{section}
\numberwithin{theorem}{section}
	
\newcommand*{\refh}[2]{\hyperref[#2]{#1~\ref{#2}}} 

\newcounter{ourcount}
\advance\textheight\topskip
\usepackage{amsmath,amsfonts,amsthm,amssymb,amscd}
\usepackage{mathtools} 
\allowdisplaybreaks

 \usepackage{graphicx}
 \usepackage[curve,matrix,arrow,color]{xy}

 \usepackage[usenames,dvipsnames]{color}
 \xyoption{line}

\usepackage[mathscr]{eucal}
\usepackage[OT2,OT1]{fontenc}

\newcommand{\id}{\mathrm{id}}

\newcommand{\ok}{{\ensuremath{\Bbbk}}}

\newcommand{\SLiiZ}{SL(2,\oZ)}

\newcommand{\modS}{\mathscr{S}}
\newcommand{\modT}{\mathscr{T}}

\newcommand{\rA}{\mathcal{A}}
\newcommand{\natiso}{\varphi}
\newcommand{\nattr}{\tilde{\iota}}

\newcommand{\Nat}{\mathrm{Nat}}
\newcommand{\Din}{\mathrm{Din}}

\newcommand{\cat}{\mathcal{C}}
\newcommand{\catV}{\mathcal{V}}
\newcommand{\Cc}{\cat}
\newcommand{\catop}{\cat^{\operatorname{op}}}
\newcommand{\catB}{\mathcal{B}}
\newcommand{\catD}{\mathcal{D}}
\newcommand{\Fun}{\mathrm{Fun}}

\newcommand{\vect}{\mathrm{{\bf vect}}}
\newcommand{\Set}{\mathrm{{\bf Set}}}
\newcommand\Zc            {\mathcal{Z}}

\newcommand{\rep}{\mathrm{\bf Rep}\,}
\newcommand{\corep}{\mathrm{\bf coRep}\,}

\newcommand{\fun}{\mathcal{F}}    

\newcommand{\idt}{\mathscr{N}}
\newcommand{\natf}{\mathscr{N}}
\newcommand{\dinf}{\mathscr{D}}

\newcommand{\ot}{\otimes}
\newcommand{\tensor}{\otimes}

\newcommand{\as}{\Phi}            

\newcommand{\assoc}{\alpha} 

\newcommand{\brC}{c}               

\newcommand{\lunit}{\lambda}
\newcommand{\runit}{\varrho}
\newcommand{\tid}{\one}      

\newcommand{\uiso}{u}

\newcommand\ldX{{}^*\!X}

\newcommand{\Salpha}{\boldsymbol{\alpha}}
\newcommand{\Sbeta}{\boldsymbol{\beta}}

\newcommand{\bchi}{\pmb{\chi}}
\newcommand{\bphi}{\pmb{\phi}}

\newcommand{\im}{\mathrm{im}}

\newcommand{\one}{\boldsymbol{1}}

\newcommand{\oC}{\mathbb{C}}

\newcommand{\eps}{\varepsilon}

\newcommand{\Ho}{A}

\newcommand{\coend}{\mathcal{L}}
\newcommand{\coa}{\rho_{\Ho^\ast}^\coend}
\newcommand{\adj}{\rho^\mathrm{adj}_\Ho}

\newcommand{\twist}{\theta}

\newcommand{\muc}{\mu_{\coend}}
\newcommand{\intL}{\Lambda_{\coend}}

\newcommand\eqpicnn[4]{\begin{eqnarray*}
                   \begin{picture}(#2,#3){}\end{picture}\nonumber\\
                   \raisebox{-#3pt}{ \begin{picture}(#2,#3) #4 \end{picture} }
                   \label{#1} \\~\nonumber \end{eqnarray*} }

\newcommand{\Rad}{\rho}

\newcommand{\intQ}{{\mathbf{c}}}
\newcommand{\coint}{\Lambda^{\rm co}}

\newcommand{\ribbon}{{\boldsymbol{v}}}
\newcommand{\sqs}{{\boldsymbol{u}}}
\newcommand{\Dt}{{\boldsymbol{f}}}

\newcommand{\oZ}{\mathbb{Z}}

\newcommand{\gcg}{\ , \ } 
\newcommand{\gp}{\ .}
\newcommand{\gc}{\ ,}
\newcommand{\qcq}{\quad , \quad } 
\newcommand{\qp}{\quad .}
\newcommand{\qc}{\quad ,}
\newcommand{\act}{\, . \,}
\newcommand{\ipic}[3][-0.5]{\raisebox{#1\height}{\scalebox{#3}{\includegraphics{pics/#2.pdf}}}}

\newcommand{\tr}{\mathrm{Tr}}
	
\newcommand{\bal}{\varkappa}
	
\begin{document}
\begin{flushright}
ZMP-HH/17-10\\
Hamburger Beitr\"age zur Mathematik 650
\end{flushright}

\vspace{2em}

\title[$\SLiiZ$-action for ribbon quasi-Hopf algebras]{$\bm{\SLiiZ}$-action for ribbon quasi-Hopf algebras}

\author{V.~Farsad, A.M.~Gainutdinov, I.~Runkel}

\address{VF: Fachbereich Mathematik, Universit\"at Hamburg, Bundesstra\ss e 55,
20146 Hamburg, Germany}
\email{vanda.farsad@uni-hamburg.de}

\address{AMG: Laboratoire de Math\'ematiques et Physique Th\'eorique CNRS,
Universit\'e de Tours,
Parc de Grammont, 37200 Tours, 
France and
 Fachbereich Mathematik, Universit\"at Hamburg, Bundesstra\ss e 55, 20146 Hamburg, Germany}
\email{azat.gainutdinov@lmpt.univ-tours.fr}

\address{IR: Fachbereich Mathematik, Universit\"at Hamburg, Bundesstra\ss e 55,
20146 Hamburg, Germany}
\email{ingo.runkel@uni-hamburg.de}

\begin{abstract}
We study the universal Hopf algebra $\coend$ of Majid and Lyubashenko in the case that the underlying ribbon category is the category of representations of a finite dimensional ribbon quasi-Hopf algebra $A$.
	We show that $\coend = A^*$ with coadjoint action and compute the Hopf algebra structure morphisms of $\coend$ in terms of the defining data of $A$. We give explicitly the condition on $A$ which makes $\rep A$ factorisable and compute Lyubashenko's projective $SL(2,\mathbb{Z})$-action on the centre of $A$ in this case.

The point of this exercise is to provide the groundwork for the applications to ribbon categories arising in logarithmic conformal field theories -- in particular symplectic fermions and $W_p$-models -- and to test a conjectural non-semisimple Verlinde formula.
\end{abstract}


\maketitle

\setcounter{tocdepth}{1}
\tableofcontents
    
\thispagestyle{empty}
\newpage

\section{Introduction} \label{sec:intro}

A modular tensor category is a finitely semisimple linear and abelian category, which is in addition a ribbon category (a braided tensor category with ribbon twist), which has a simple tensor unit, and whose braiding satisfies a certain non-degeneracy condition \cite{tur}. Modular tensor categories are important algebraic objects because they precisely encode the data necessary to define a 3-2-1 extended topological field theory \cite{retu,tur,Bartlett:2015baa}. 

One source of modular tensor categories are finite-dimensional ribbon Hopf algebras $H$ which are semisimple as algebras and which are 
{\em factorisable}~\cite{RS,Ta}.
Factorisability means that the monodromy matrix $M=R_{21}R \in H \otimes H$ obtained from the universal $R$-matrix of $H$ is non-degenerate as a copairing. 

A three-dimensional topological field theory gives rise to representations of mapping class groups of surfaces, possibly with marked points (in the extended case). 
	It turns out that one can drop the semisimplicity condition on the category and still obtain such mapping class group representations \cite{Lyubashenko:1994tm,Kerler:2001}, even if there no longer is an underlying 3-2-1 topological field theory in the sense of \cite{Bartlett:2015baa}.

The algebraic datum is now a 
finite abelian ribbon category with simple tensor unit, whose braiding satisfies a (more complicated) non-degeneracy condition \cite{Lyubashenko:1995,Lyubashenko:1994tm}.
We refer to such categories as factorisable finite ribbon tensor categories.
Again, finite-dimensional factorisable ribbon Hopf algebras provide examples, now without the semisimplicity requirement. 

In this paper we apply the general formalism of \cite{Lyubashenko:1995,Lyubashenko:1994tm,Kerler:2001} to
	finite-dimensional ribbon
\textit{quasi}-Hopf 
algebras $A$. We express the relevant non-degeneracy condition on the braiding in terms of the defining data of $A$ 
(see Section \ref{sec:non-deg-hopf-pair}) 
and compute the action of $SL(2,\oZ)$ -- the mapping class group of the torus -- on the centre $Z(A)$ of $A$
(Theorem \ref{thm:SL2Z-on-centre}). As is maybe not surprising for readers who have looked at quasi-Hopf algebras before, this leads to fairly long expressions in terms of the coassociator, universal $R$-matrix, etc.

Our motivation for carrying out this exercise is two fold: firstly -- as we explain next -- it is easy to detect when a finite tensor category comes from a quasi-Hopf algebra; secondly, it puts in place the explicit expressions we need for the symplectic fermion calculation in \cite{FGRprep} (Remark \ref{rem:SF-motivation}).

Let $\cat$ be a 
finite tensor category over a field $\ok$. If there exists a fiber functor $F : \cat \to \vect_\ok$, by reconstruction one can find a Hopf algebra $H$ such that $\cat \cong \rep H$ as linear
 tensor categories~\cite{Ulbrich}, see also~\cite[Sec.\,9.4]{Majid-book}. 
If we only require $F$ to be {\em multiplicative}, i.e.\ that there are isomorphisms $F(U \otimes V) \cong F(U) \otimes F(V)$ natural in $U,V$ but not subject to coherence conditions, then reconstruction results in a
{\em quasi}-Hopf 
algebra~\cite[Sec.\,9.4]{Majid-book}.
While it may be difficult to determine whether there is a fiber functor $\cat \to \vect_\ok$, there is a very simple criterion for the existence of multiplicative functors (see \cite[Prop.\,6.1.14]{EGNO-book}):

\begin{theorem}\label{thm:quasi-integralFP}
Let $\cat$ be a finite tensor category 
	over an algebraically closed field
for which the Perron-Frobenius dimensions of its simple objects are integers.
Then $\cat$ is equivalent as a linear tensor category to $\rep A$ for a
finite-dimensional quasi-Hopf algebra $A$.
\end{theorem}

The Perron-Frobenius dimensions of a simple object $X \in \cat$ is the 
	positive real number given by the
maximal non-negative eigenvalue of the linear map $[X \otimes -]$ 
on the $\mathbb{C}$-linearised 
Grothendieck ring $\mathbb{C} \otimes_\oZ \Gr(\cat)$,
 see e.g.\ \cite{EGNO-book}.
	For $\rep A$, the Perron-Frobenius dimension of an object is simply the dimension of the underlying vector space.

\smallskip

We now describe in more detail the construction of \cite{Lyubashenko:1995,Lyubashenko:1994tm,Kerler:2001}, see also 
	\cite[Sec.\,4]{Fuchs:2010mw}
for a review. In this paper we will only be interested in the action of the mapping class group of the torus, i.e.\ of $SL(2,\oZ)$.

Let $\cat$ be a factorisable finite ribbon tensor category over a field $\ok$.
Let $\coend \in \cat$ be the coend for the functor 
	$\cat^{\mathrm{op}} \times \cat \to \cat$:
 $(U,V) \mapsto U^* \otimes V$. 
As we review in Section \ref{sec:univ-hopf-via-coends}, 
using the universal property of the coend, one can endow $\coend$ with the structure of a Hopf algebra in the braided category $\cat$
	(Definition \ref{def:Hopf-C}),
together with a Hopf pairing $\omega_\coend : \coend \otimes \coend \to \one$.
The category
$\cat$ is called {\em factorisable} if $\omega_\coend$ is non-degenerate. 
Using once more the universal property 
one defines endomorphisms $\modS$, $\modT$ of~$\coend$
	which induce a projective action of $SL(2,\oZ)$ on $\cat(\one,\coend)$ (see Section \ref{sec:SL2Z-action}).
One finds that $\cat(\one,\coend) \cong \End(\id_{\cat})$, and so one obtains a projective action of $SL(2,\oZ)$ on $\End(\id_{\cat})$.

Let $A$ be a finite-dimensional ribbon quasi-Hopf algebra (see Section \ref{sec:ribbon-qHopf} for conventions and details). 
We show that, as for Hopf algebras \cite{Lyubashenko:1994tm,Kerler:1996}, the coend $\coend$ in $\rep A$ is given by the coadjoint representation on $A^*$ (Proposition \ref{prop:coend-RepH}). 
We 
compute the structure morphisms of the 
Hopf algebra $\coend$, as well as the 	pairing
$\omega_\coend$, 
in terms of the data of $A$ (Theorem \ref{thm:explicit-coend-qHopf}). Note that $\End(\id_{\rep A}) \cong Z(A)$, the centre of $A$. Using our explicit expressions, we
	give
the action of the $S$- and $T$-generators of $SL(2,\oZ)$ on $Z(A)$ in terms of the defining data 
	of $A$ and an integral, 
see Proposition \ref{prop:coint-coend-via-int-qHopf} and Theorem \ref{thm:SL2Z-on-centre}.
This generalises results for Hopf algebras in \cite{Lyubashenko:1994ma}  to quasi-Hopf algebras.

\smallskip

The Hopf algebra structure on $\coend$ can be understood as a special case of
braided version~\cite{Majid:1993} of the reconstruction theorem~\cite{Ulbrich}. 
There, one characterises a Hopf algebra in a braided monoidal category $\mathcal{V}$ via a monoidal functor $F$ from a monoidal category $\cat$ to $\mathcal{V}$ (see Section \ref{sec:univ-hopf-braided}). If $\cat$ is itself braided, there is a canonical choice for 
$\mathcal{V}$ and $F$, namely one can consider $\id_{\cat} : \cat \to \cat$. We will refer to the Hopf algebra $\rA$ in $\Cc$
characterised by this functor as {\em universal Hopf algebra}
 (being defined by a universal property, $\rA$ may or may not exist). 
We explain in detail an observation in \cite{Lyubashenko:1995} that $\rA \cong \coend$ as Hopf algebras in $\Cc$ with Hopf pairing (Proposition \ref{prop:univ-vs-coend}).

This reconstruction point of view is also used in \cite{[BT]}, where the notion of factorisability of a quasi-triangular quasi-Hopf algebra has first been defined. The definition in \cite{[BT]} is a priori different from factorisability of $\rep A$ (i.e.\ non-degeneracy of $\omega_\coend$), but we show in Corollary \ref{prop:fact-qHopf=fact-coend} 
 that the two definitions are equivalent.
To do so, we use  that factorisability can be expressed as the invertibility of a map -- composed of braidings and duality morphisms -- from the coend
 $\coend$ for $(U,V) \mapsto U^* \otimes V$ to the end for $(U,V) \mapsto U \otimes V^*$
  (Proposition~\ref{prop:C-DD}). 
  When applied to $\rep A$, this invertibility condition reduces precisely
to the definition in \cite{[BT]} (see Section \ref{sec:equiv-fact-cond}).
 
\smallskip

In our discussion of the universal Hopf algebra via reconstruction and via coends, we are careful to include the coherence isomorphism of the underlying monoidal category (as opposed to \cite{Majid:1993,Lyubashenko:1995} which work in the strict case). This is not meant as an extra torture for the reader but is necessary for the computation of the structure morphisms on the coend in Section \ref{sec:coend-repA}.

\begin{remark}\label{rem:SF-motivation}
Computing the $SL(2,\oZ)$-action on $\End(\id_{\rep A})$ explicitly is 
part of a larger project to provide a family of examples for a conjectured non-semisimple variant of the Verlinde formula \cite[Conj.\,5.10]{Gainutdinov:2016qhz}. The original semisimple version of the Verlinde formula was found in \cite{Verlinde:1988sn} and proved in the context of vertex operator algebras in \cite{Huang:2004bn}. A related investigation of the non-semisimple Verlinde formula is carried out in \cite{Creutzig:2016fms}.
\\
The conjecture in \cite{Gainutdinov:2016qhz} stipulates an isomorphism of projective representations between two $SL(2,\oZ)$-actions
associated to a $C_2$-cofinite, simple, self-dual and non-negatively graded vertex operator algebra $V$. The first action is obtained by modular transformations on the space of so-called pseudo-trace functions of $V$ \cite{Miyamoto:2002ar,Arike:2011ab}. For the second action one uses that $\rep V$ is conjecturally a factorisable finite ribbon tensor category and thus carries a projective $SL(2,\oZ)$-action on $\End(\id_{\rep V})$ as described above (see e.g.\ \cite[Sec.\,5]{Gainutdinov:2016qhz} for details).
\\
The series of examples we investigate are the so-called symplectic fermions, which are
 pa\-ra\-met\-ris\-ed by $N \in \oZ_{>0}$, the ``number of pairs of symplectic fermions''. In \cite{Davydov:2012xg,Runkel:2012cf} a conjecture was presented for the explicit ribbon category structure on $\rep V$, for $V$ the (even part of the) vertex operator algebra of $N$ pairs of symplectic fermions, see \cite[Conj.\,7.4]{Davydov:2016euo} for a precise formulation. 
Using the results of the present work, in the follow-up paper \cite{FGRprep} a description of this category as representations of a quasi-Hopf algebra $Q_N$ is given
and the $SL(2,\oZ)$-action on $\End(\id_{\rep {Q_N}})$ is computed 
(see also \cite{Gainutdinov:2015lja} for the case $N=1$).
The $SL(2,\oZ)$-action on the space of pseudo-trace functions
follows from \cite{Gainutdinov:2016qhz}. 
The outcome of this project is that the $SL(2,\oZ)$-actions do indeed agree projectively \cite{FGRprep} 
-- this provides the first example of such a comparison for non-semisimple theories in the literature.
\\
In a separate ongoing project \cite{Creutzig-prep}, the formalism presented here is applied to a series of ribbon quasi-Hopf algebras based on so-called restricted quantum groups for $sl(2)$ that have  already appeared in relation to the logarithmic $W_p$-triplet conformal field theories~\cite{[FGST]}.
\\
The explicit description via quasi-Hopf algebras of the factorisable finite tensor categories associated to symplectic fermions and $W_p$-triplet models  in \cite{FGRprep,Creutzig-prep} will also be useful when investigating 
bulk correlation functions for these models using the formalism developed in \cite{Fuchs:2010mw,Fuchs:2013lda,Fuchs:2016wjr}.
\end{remark}

The Verlinde formula has a purely categorical counterpart. Namely, let $\cat$ be a factorisable finite tensor category over an algebraically closed field of characteristic zero. Denote by $\modS_{\cat}$ the action of the $S$-generator of $SL(2,\oZ)$ on $\End(\id_{\cat})$ from \cite{Lyubashenko:1995}. 
Using the theory of internal characters of 
	\cite{Fuchs:2010mw,Shimizu:2015}
one obtains an injective map $\phi : \Gr(\cat) \to \End(\id_{\cat})$, $[M] \mapsto \phi_M$ (see Section \ref{sec:intchar-natendo}). For $U,V,W \in \cat$ simple denote by $N_{UV}^{~W}$ the structure constants in $\Gr(\cat)$, i.e.\ $[U][V] = \sum_{W} N_{UV}^{~W} [W]$. 
Here, the sum runs over representatives $W$ of isomorphism classes of simple objects in $\cat$. The categorical Verlinde formula states \cite[Thm.\,3.9]{Gainutdinov:2016qhz} (see \cite[Thm.\,4.5.2]{tur} for the semisimple case, i.e.\ the case of modular fusion categories)
\be\label{eq:verlinde-general}
	\modS_{\cat}^{-1}\big( 
	\modS_{\cat}(\phi_U) \circ \modS_{\cat}(\phi_V) \big)
	~=~ \sum_W N_{UV}^{~W} \, \phi_W \ .
\ee
In Section \ref{sec:intchar-qHopf} we evaluate this result explicitly in the case $\cat = \rep A$ for $A$ a finite dimensional factorisable quasi-Hopf algebra. 
	We give elements $\bchi_M \in Z(A)$, $M \in \rep A$,
corresponding to $S_{\rep A}(\phi_M)$ 
 via $Z(A) \cong \End(\id_{\rep A})$. One finds that the structure constants of $\Gr(\rep A)$ can be computed from \eqref{eq:verlinde-general} in terms of
\begin{itemize}
\item the defining data of the ribbon quasi-Hopf algebra $A$,
\item the characters $\tr_M(-)$ of all irreducible $A$-modules.
\end{itemize}
Explicitly, the $N_{UV}^{~W}$ are uniquely determined by the following linear relations in $Z(A)$:
\be
    \bchi_U \, \bchi_V
    ~=~ \sum_{W} N_{UV}^{~W} \, \bchi_W \ .
\ee
	We stress that there is no need to compute the centre of $A$, which is a difficult problem in general.
A related (but different) result on the categorical Verlinde formula in the case of factorisable Hopf algebras is given in \cite{Cohen:2008}.

\medskip

This paper is organised as follows. 

In Section \ref{sec:univ-hopf-braided} we give our conventions for braided tensor categories and Hopf algebras in them, and we review the reconstruction theory for Hopf algebras of \cite{Majid:1993} in the special case of the identity functor, leading to the universal Hopf algebra. 
In Section \ref{sec:univ-hopf-via-coends} an equivalent description of the universal Hopf algebra in terms of coends is given. This is the formalism used in \cite{Lyubashenko:1995} and in the rest of this paper.
The categorical setting we will work in -- factorisable finite tensor categories -- is described in Section \ref{sec:univ-Hopf-finite}. 
The $SL(2,\mathbb{Z})$-action of \cite{Lyubashenko:1995} and the theory of internal characters of \cite{Fuchs:2010mw,Shimizu:2015} are recalled in Section \ref{sec:SL2Z}. 
In Section \ref{sec:ribbon-qHopf} we state our conventions for quasi-triangular quasi-Hopf algebras $A$.
Section \ref{sec:coend-repA} contains our first main result, the explicit computation of the Hopf algebra structure maps of the universal Hopf algebra $\coend$ in $\rep A$. We state the factorisability condition on $\rep A$ in terms of the defining data of $A$ and show that it is equivalent to the definition in \cite{[BT]}.
Finally, Section \ref{sec:SL2Z-quasiHopf} contains our second main result, namely explicit expressions for the $S$- and $T$-action on the centre $Z(A)$ of a ribbon quasi-Hopf algebra $A$.

\bigskip

\medskip

\noindent
{\bf Acknowledgements:} We thank 
	Daniel Bulacu,
	Simon Lentner, 
	Yorck Sommerh\"auser
	and
	Blas Torrecillas
for discussions and helpful comments on a draft of this paper. The work of AMG was supported by DESY and CNRS.

\medskip

\noindent
\textbf{Conventions:}  In what follows,
by ``category'' we will mean ``essentially small category'', i.e.\ we will always assume that the isomorphism classes of objects form a set.
For the set of morphisms $U\to V$ in a category $\cat$ we use the notation $\cat(U,V)$ instead of $\Hom_{\cat}(U,V)$.
We fix a field $\ok$ and denote by $\vect_\ok$ the category of finite-dimensional $\ok$-vector spaces.
For $A$ a $\ok$-algebra (later mostly a quasi-Hopf algebra), the $\ok$-linear category of finite-dimensional left $A$-modules will be denoted by $\rep A$.

\section{The universal Hopf algebra in a braided monoidal category with duals} 
\label{sec:univ-hopf-braided}

In this section, we introduce our conventions on monoidal categories, duals and braidings, as well as for Hopf algebras in such categories. Then we review the generalized Tannaka--Krein-like reconstruction of a Hopf algebra for  a braided category with left duals, following~\cite{Majid:1993}.

\subsection{Conventions for monoidal categories} \label{subsec:conventions-mono-cats}

Let $\cat$ be a monoidal category. Our conventions for associator and unit isomorphisms are (following \cite[Sec.\,5]{ChPr})
\be
\assoc_{U,V,W}:\; U\otimes(V\otimes W) \xrightarrow{~ \sim ~} (U\otimes V)\otimes W ~~,\quad
\lunit_U: \one\otimes U \xrightarrow{ \sim } U 
~~,\quad
\runit_U:  U\otimes\one \xrightarrow{ \sim } U
\ .
\ee
For a braided monoidal category we denote the braiding isomorphisms by
\be
\brC_{U,V}:\; U\otimes V\xrightarrow{~ \sim ~} V\otimes U \ .
\ee
A monoidal category $\cat$ is said to have {\em left duals} if for each $U\in \cat$ there have been chosen an object $U^{\ast} \in \cat$ and morphisms
\be
\ev_U:\, U^\ast\otimes U\to \one
\quad , \quad \coev_U:\, \one \to U\otimes U^\ast \ ,
\ee
which satisfy the two zig-zag identities. Similarly, we say $\cat$ has {\em right duals} if for each $U\in \cat$ there have been chosen an object ${^{*}U} \in \cat$ and morphisms
\be\label{eq:left-duals-def}
\widetilde{\ev}_U:\, U\otimes {^{*}U}\to \one
\quad , \quad
\widetilde{\coev}_U:\, \one \to {^{*}U}\otimes U \ .
\ee 
subject to the zig-zag identities.

For a monoidal category $\cat$ with left duals one obtains a contravariant functor $(-)^*: \cat \to \cat$. For a morphism $f : U \to V$, the image 
$f^* : V^* \to U^*$ 
under $(-)^*$ is constructed from $\ev_U$, $\coev_U$ and the coherence isomorphisms of $\cat$. An analogous remark applies to right duals and ${}^*(-): \cat \to \cat$.

A {\em rigid monoidal category} is a monoidal category which has both right and left duals. 

A braided monoidal category with left duals is {\em ribbon} if it is equipped with a natural isomorphism $\theta$ of the identity functor (the {\em ribbon twist}), which satisfies, for all $U,V \in \cat$,
\be\label{eq:theta-braiding-prop}
\theta_{U\tensor V}  = (\theta_U\tensor \theta_V) \circ \brC_{V,U}\circ \brC_{U,V} \qquad \text{and} \qquad
\theta_{U^*} = (\theta_U)^*\ .
\ee
In a ribbon category, the left dual $U^*$ of an object $U$ is automatically also a right dual, with the right duality morphisms constructed from the left ones, together with the braiding and the ribbon twist, see e.g.\ \cite[Sec.\,XIV.3]{Kassel-book}.
In particular, a ribbon category is rigid
	and in fact pivotal, see \cite[Sec.\,8.10]{EGNO-book}.

Below we will use string diagram notation for morphisms. Our diagrams are read from bottom to top, and the diagrams we will use for the braiding, ribbon twist, and the
	left/right
duality maps are
\begin{align}
\brC_{U,V} &~=\hspace*{.7em} \ipic{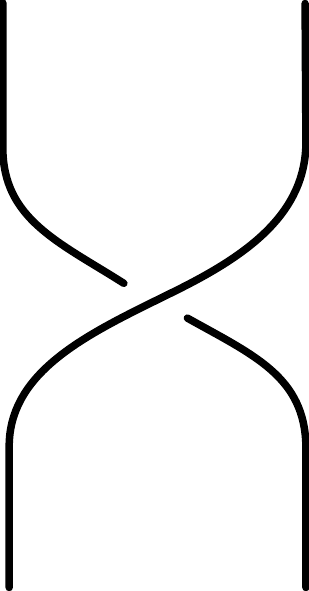}{.25} 
\put(-42,-44){\scriptsize $U$} \put(-5,-44){\scriptsize $V$} 
\put(-42,38){\scriptsize $V$} \put(-5,38){\scriptsize $U$} \qc
&
\ev_U &~=\hspace*{.7em} \ipic{counit-rhs}{.25}  
\put(-36,-36){\scriptsize $U^\ast$} \put(-6,-36){\scriptsize $U$} \qc&
\coev_{U} &~=\hspace*{.7em} \ipic{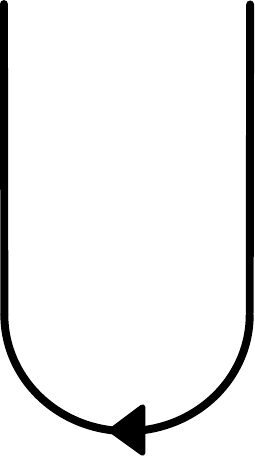}{.25}  
\put(-3,31){\scriptsize $U^\ast$} \put(-35,31){\scriptsize $U$} \qc&
\\ 
\theta_{U} &~=\hspace*{1.7em} \ipic{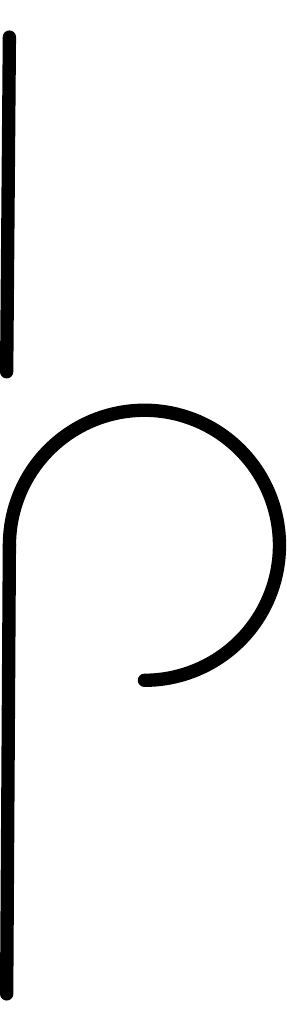}{.15}  
\put(-24,-44){\scriptsize $U$} \put(-24,37){\scriptsize $U$} 
\qc
&
\widetilde\ev_U &~=\hspace*{.7em} \ipic{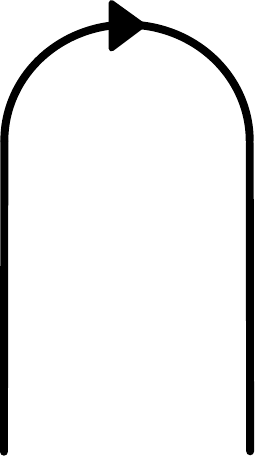}{.25} 
\put(-34,-37){\scriptsize $U$} \put(-7,-37){\scriptsize $^{*}U$} 
\qc
& 
\widetilde\coev_U &~=\hspace*{.7em} \ipic{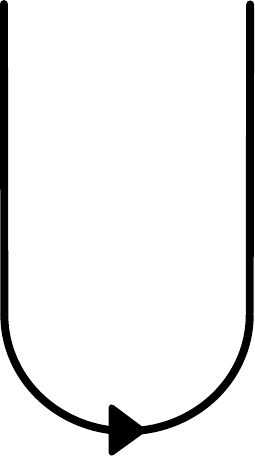}{.25} 
\put(-35,31){\scriptsize $^{*}U$} \put(-5,31){\scriptsize $U$} 
\qp
\end{align}

When giving morphism involving associator and unit isomorphism, we often write them as sequences of arrows, where for better readability we omit the tensor product symbol between objects and often only write ``$\sim$'' for a composition of coherence isomorphisms of the monoidal category. For example, one of the two zig-zag identities for left duals reads
\be
	\big[\,
	U \xrightarrow{\sim}
	\one U  
	\xrightarrow{\coev_U \otimes \id_U}
	(UU^*)U
	\xrightarrow{\sim} 
	U(U^*U)
	\xrightarrow{\id \otimes \ev_U}
	U\one
	\xrightarrow{\sim}
	U
	\,\big]
	~=~ \id_U \ .
\ee

In a monoidal category $\cat$ with left duals, there is a canonical natural isomorphism $\gamma_{V,U}: U^\ast\tensor V^\ast \to (V\tensor U)^\ast$. 
To define it, we first introduce the morphism (here we write out all coherence isomorphisms explicitly as we will need them in Section \ref{sec:coend-repA})
\begin{align}\label{eq:gamma-tildeUV}
\begin{split}
\tilde\gamma_{V,U} = \Big[ (U^\ast  V^\ast)  ( V U ) &\xrightarrow{\assoc^{-1}_{U^\ast,V^\ast,  V U}} U^\ast  (V^\ast  (V  U))
\xrightarrow{\id\tensor\assoc_{V^\ast ,V, U}} 
U^\ast  ((V^\ast  V)  U)\\
&\quad  \xrightarrow{\id\tensor\ev_{V}\tensor\id} U^\ast  (\tid  U) 
\xrightarrow{\id\tensor\lambda_U} U^\ast  U \xrightarrow{\ev_{U}}\tid \Big] \ .
\end{split}
\end{align}
Using this, $\gamma_{V,U}$ is given by
\begin{align}\label{eq:gammaUV}
\begin{split}
\gamma_{V,U} = \Big[ 
&U^\ast  V^\ast \xrightarrow{\;\runit^{-1}_{U^\ast  V^\ast}\;} (U^\ast  V^\ast)  \tid \xrightarrow{\id\tensor\coev_{V  U}}
(U^\ast  V^\ast)  ((V  U)  (V  U)^\ast) \\
&\xrightarrow{\assoc_{U^\ast  V^\ast,V  U,(V  U)^\ast}} ((U^\ast  V^\ast)  (V  U))  (V  U)^\ast
\xrightarrow{\tilde\gamma_{V, U}\tensor\id} \tid   (V  U)^\ast \xrightarrow{\;\lambda_{(V  U)^\ast}\;} (V  U)^\ast
\Big] \ . 
\end{split}
\end{align}
In string diagram notation, these definitions look much simpler: 
\be\label{eq:gammaVU}
   \tilde\gamma_{V,U} ~=~	 \ipic{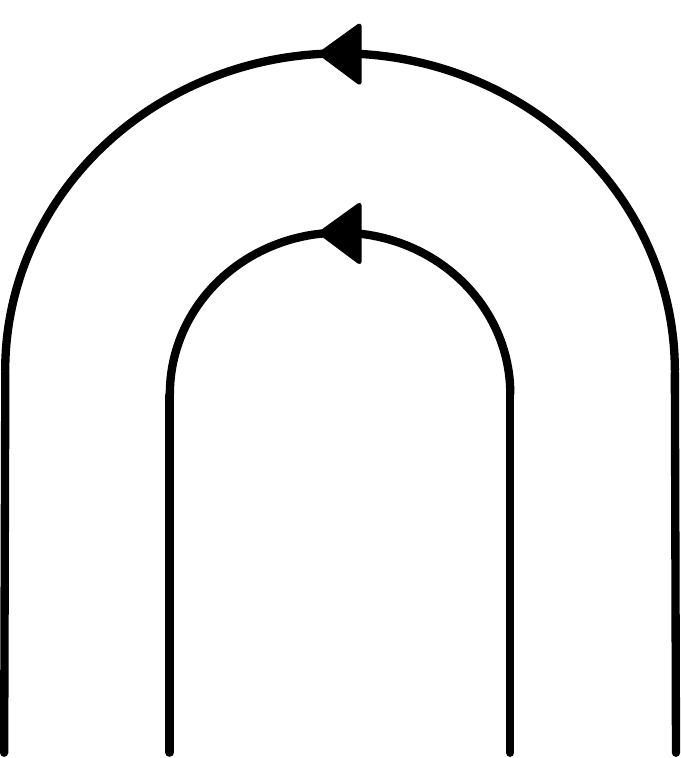}{0.2} 
	 \put(-69,-45){\scriptsize $U^\ast$} \put(-53,-45){\scriptsize $V^\ast$} 
	 \put(-21,-45){\scriptsize $V$} \put(-5,-45){\scriptsize $U$} 
	 \quad , \quad\gamma_{V,U}~=~\ipic{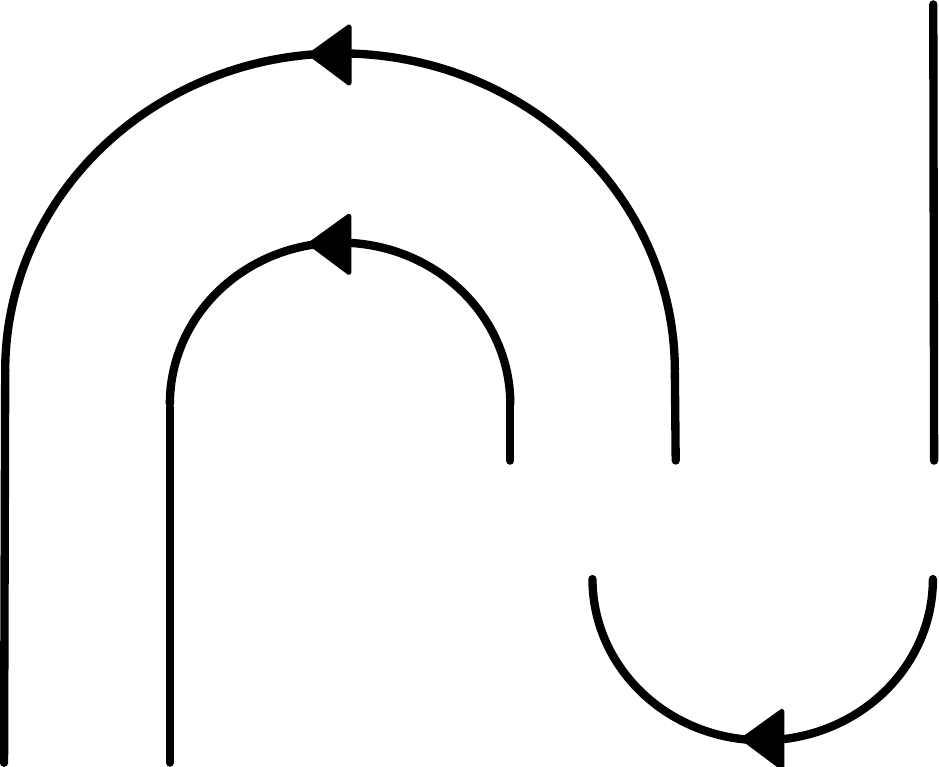}{0.2} \quad 
 	 \put(-95,-45){\scriptsize $U^\ast$} \put(-79,-45){\scriptsize $V^\ast$} 
	 \put(-44,-16){\scriptsize $V$} \put(-28,-16){\scriptsize $U$}  \put(-10,-16){\scriptsize $(VU)^\ast$} 
   \quad = \quad\ipic{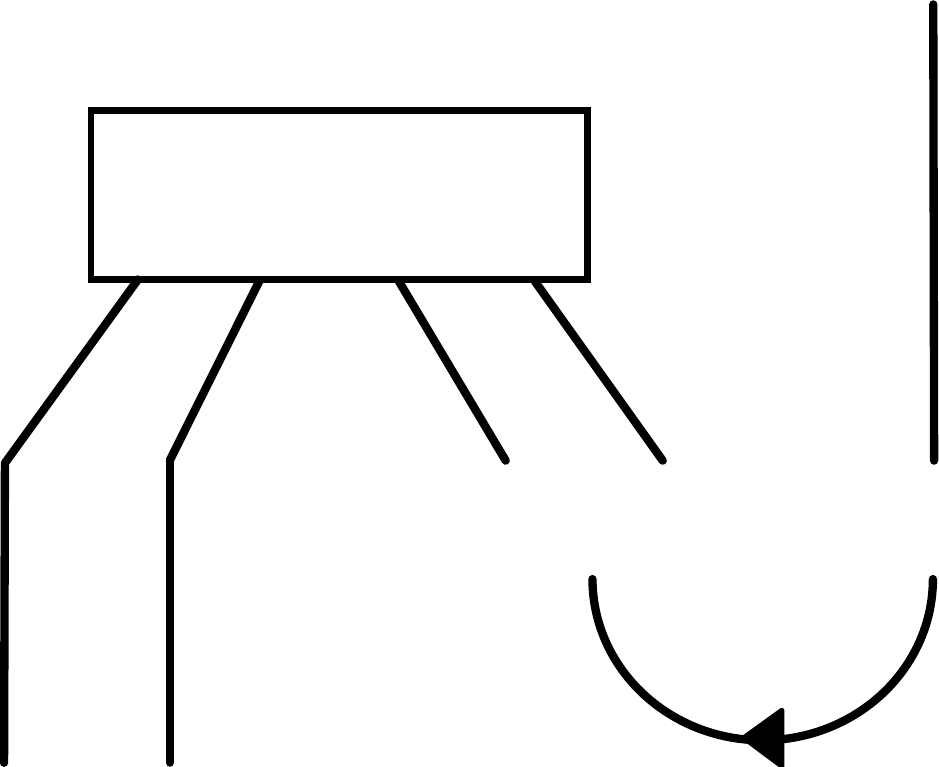}{0.2}
	 \put(-95,-45){\scriptsize $U^\ast$} \put(-79,-45){\scriptsize $V^\ast$} 
	 \put(-44,-16){\scriptsize $V$} \put(-28,-16){\scriptsize $U$}  \put(-10,-16){\scriptsize $(VU)^\ast$}
	 \put(-70,16){ $\tilde\gamma_{V,U}$}
	 \quad .
\ee
Finally we note that if $\cat$ has left duals, there is a canonical isomorphism $\one^* \to \one$ given by
\be\label{eq:iso11*}
	\one^* \xrightarrow{\runit_{\one^*}^{-1}} \one^*\one \xrightarrow{\ev_{\one}} \one \ .
\ee
When writing $\one^* \xrightarrow{\sim} \one$ below, we refer to this isomorphism.

\subsection{Conventions for Hopf algebras in braided categories} \label{subsec:conventions-Hopf-bmc}

The definition of a Hopf algebra over a field has a natural generalisation to braided monoidal categories, see e.g.~\cite{Majid:1994}.

\begin{definition}\label{def:Hopf-C}
Let $\cat$ be a braided monoidal category.
 \textit{A Hopf algebra $H$ in $\cat$}  is  an object~$H$ together with morphisms
\begin{align}
\text{(product)} \quad
&\mu_H: H\tensor H \to H \ ,
&
\text{(coproduct)}\quad
&\Delta_H: H \to H\tensor H \ ,
\\
\text{(unit)}\quad
&\eta_H:\tid\to H \ ,
&
\text{(counit)}\quad
&\eps_H: H\to \tid \ ,
\nonumber\\
\text{(antipode)} \quad
& S_H: H\to H \ .
\nonumber
\end{align} 
These data are subject to the conditions
\begin{itemize}
	\setlength{\leftskip}{-1.5em}
\item associativity and unitality:
\begin{align}
&
\big[\, H(HH) \xrightarrow{\id \otimes \mu_H} HH \xrightarrow{\mu_H} H \,\big] 
~=~
\big[\, H(HH) \xrightarrow{\sim} (HH)H \xrightarrow{\mu_H \otimes \id} HH \xrightarrow{\mu_H} H \,\big]
\ ,
\\
&\big[\, H \xrightarrow{\sim} \one H \xrightarrow{\eta_H \otimes \id} HH \xrightarrow{\mu_H} H \,\big]
~=~
\id_H
~=~
\big[\, H \xrightarrow{\sim}  H \one \xrightarrow{\id \otimes \eta_H} HH \xrightarrow{\mu_H} H \,\big] \ .
\nonumber
\end{align}
\item coassociativity and counitality: same as above but with all arrows reversed, $\mu_H$ replaced by $\Delta_H$ and $\eta_H$ by $\eps_H$.
\item $\Delta_H,\eps_H$ are algebra homomorphisms:
\begin{align}\label{eq:Delta-alg-hom-condition}
 &\Delta_H \circ \eta_H = (\eta_H \otimes \eta_H) \circ \lambda_{\one}^{-1}
 \ ,
 \\ \nonumber
 &\big[\, HH \xrightarrow{\mu_H} H \xrightarrow{\Delta_H} HH
 \,\big]
 ~=~
 \big[\,
 HH \xrightarrow{\Delta_H \ot \Delta_H} (HH)(HH)
 \xrightarrow{\sim} H((HH)H)
 \\ \nonumber & \hspace{12em}
 \xrightarrow{\id \ot c_{H,H} \ot \id} H((HH)H)
 \xrightarrow{\sim}(HH)(HH)
 \xrightarrow{\mu_H \ot \mu_H}
 HH
 \,\big] \ ,
\end{align}
and
\be
	\eps_H \circ \mu_H = \lambda_{\one} \circ (\eps_H \otimes \eps_H)
	\quad , \quad
	\eps_H \circ \eta_H = \id_{\one} \ .
\ee
\item antipode condition: 
\begin{align}
\big[\, H \xrightarrow{\eps_H} \one \xrightarrow{\eta_H} H \,\big]
&~=~
\big[\, H \xrightarrow{\Delta_H} HH \xrightarrow{S_H \otimes \id} HH \xrightarrow{\mu_H} H \,\big]
\\
&~=~
\big[\, H \xrightarrow{\Delta_H} HH \xrightarrow{\id \otimes S_H} HH \xrightarrow{\mu_H} H \,\big]
\ .
\nonumber
\end{align}
\end{itemize}
\end{definition}

As a consequence of the Hopf-algebra axioms, we get that $S_H$ is an algebra and a coalgebra 
	anti-homomorphism,
in particular we have $S_H\circ \mu_H = \mu_H \circ \brC_{H,H} \circ (S_H\tensor S_H)$, see~\cite[Lem.\,2.3]{Majid:1994}.

When using string diagram notation to depict morphisms involving Hopf algebras, we use the following notation for its structure morphisms:
\be
\mu_H = \ipic{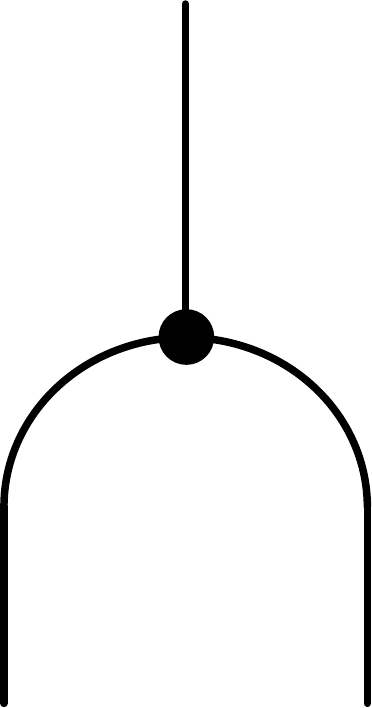}{.25} \quad , \quad \Delta_H = \ipic{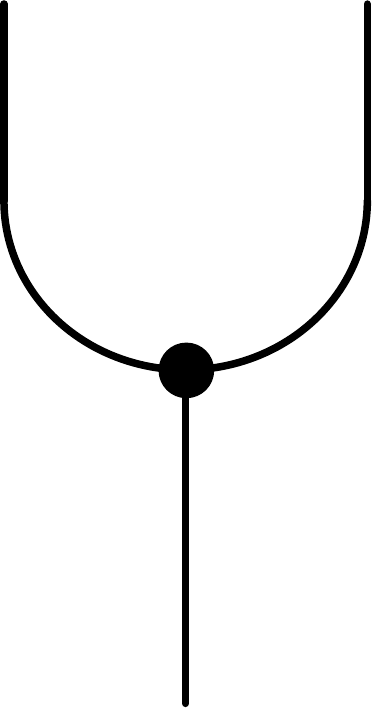}{.25} \quad , \quad
 \eta_H = \ipic{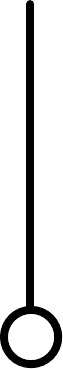}{.25} \quad , \quad
  \eps_H = \ipic{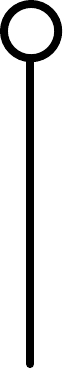}{.25} 
\quad , \quad S_H = \ipic{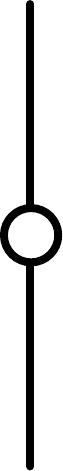}{.25} \quad .
\put(-347,45){\scriptsize $H$}  \put(-370,-51){\scriptsize $H$} \put(-326,-51){\scriptsize $H$}  
\put(-241,-51){\scriptsize $H$}  \put(-262,45){\scriptsize $H$} \put(-218,45){\scriptsize $H$} 
\put(-156,26){\scriptsize $H$} \put(-91,-31){\scriptsize $H$} 
\put(-24,-37){\scriptsize $H$} \put(-23,31){\scriptsize $H$}
\ee
For example, the second condition in \eqref{eq:Delta-alg-hom-condition}, i.e.\ that $\Delta_H$ is compatible with $\mu_H$, reads

\be\label{coprod-alg-map}
   \ipic{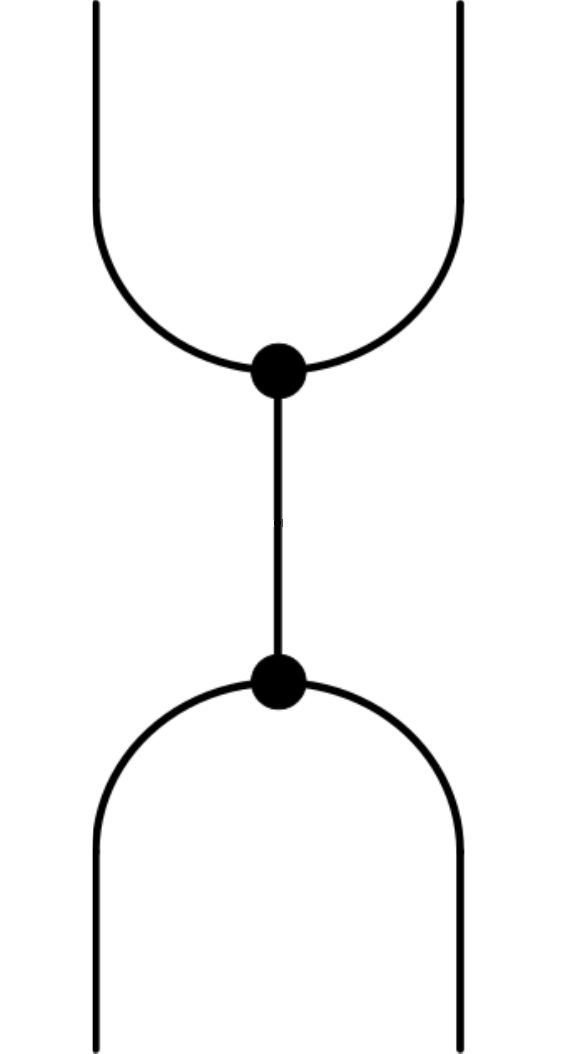}{0.23}
	 \put(-17,-67){\scriptsize $H$} \put(-59,-67){\scriptsize $H$} 
	 \put(-16,61){\scriptsize $H$} \put(-59,61){\scriptsize $H$} 
	 \put(-49,8){\scriptsize $\Delta_H$} \put(-29,-13){\scriptsize $\mu_H$}
	 \ = \quad \ipic{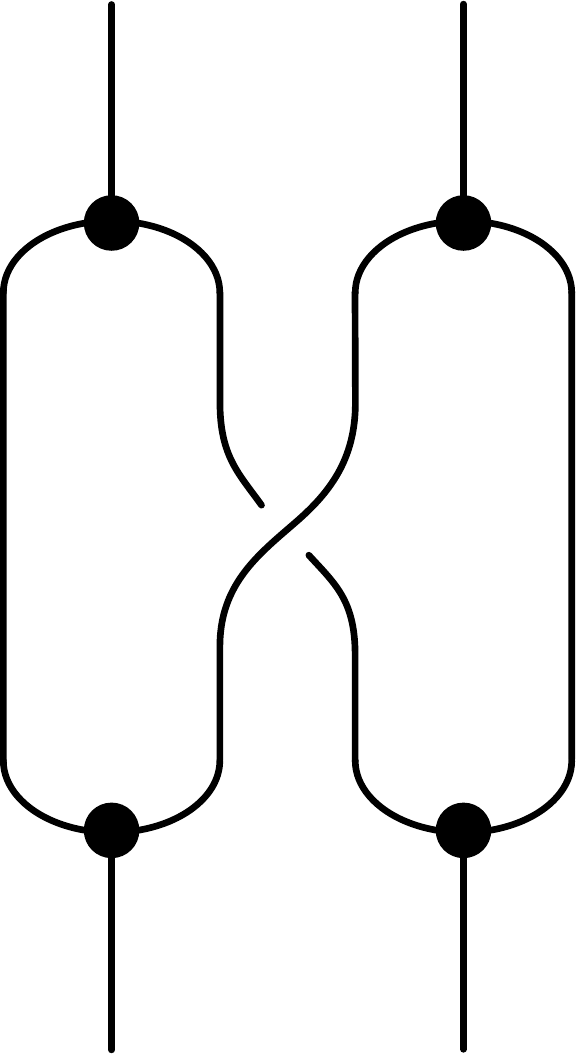}{.23}
	 \put(-17,-67){\scriptsize $H$} \put(-56,-67){\scriptsize $H$} 
	 \put(-16,61){\scriptsize $H$} \put(-56,61){\scriptsize $H$} 
\ee

The appearance of the braiding as opposed to the inverse braiding in the above condition is a choice, related to the choice made in defining the tensor product of algebras: For (associative, unital) algebras $(A,\mu_A,\eta_A)$ and $(B,\mu_B,\eta_B)$ we define the algebra $A\otimes B$ to have structure morphisms
\begin{align}
	\mu_{A \otimes B}
	&~=~
	\big[\,
	(AB)(AB) \xrightarrow{\sim} (A((BA)B)
	\xrightarrow{\id \otimes c_{B,A} \otimes \id}
	(A((AB)B)
	\\ \nonumber & \hspace{4em}
	\xrightarrow{\sim}(AA)(BB)
	\xrightarrow{\mu_A \otimes \mu_B}
	AB
	\,\big] \ ,
	\\ \nonumber
	\eta_{A \otimes B}
	&~=~
	\big[\,\one \xrightarrow{\sim} \one\one \xrightarrow{\eta_A \otimes \eta_B} AB \,\big] \ .
\end{align}

By definition, a \textit{Hopf pairing} for a Hopf algebra $H$ in $\cat$ is a morphism $\omega_H: H\tensor H \to \tid$  which makes
 the multiplication $\mu_H$ and the coproduct $\Delta_H$, as well as the unit $\eta_H$ and the counit $\eps_H$, each others adjoints. 
In terms of string diagrams, this means

\be\label{eq:hopf-pair-prop} 
\begin{split}
&   \ipic{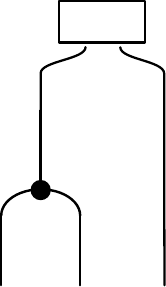}{0.6}
	 \put(-51,-51){\scriptsize$H$} \put(-29,-51){\scriptsize$H$} \put(-5,-51){\scriptsize$H$}
	 \put(-25,33){$\omega_H$}
	 \quad = \quad \ipic{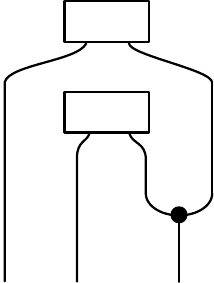}{0.6}
	 \put(-65,-50){\scriptsize$H$} \put(-43,-50){\scriptsize$H$} \put(-15,-50){\scriptsize$H$}
	 \put(-37,33){$\omega_H$}  \put(-37,6){$\omega_H$} 
	 \qquad , \qquad \ipic{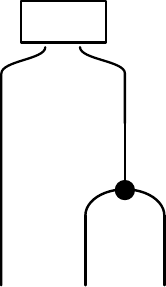}{0.6}
	 \put(-51,-51){\scriptsize$H$} \put(-29,-51){\scriptsize$H$} \put(-5,-51){\scriptsize$H$}
	 \put(-35,33){$\omega_H$}
	 \quad = \quad \ipic{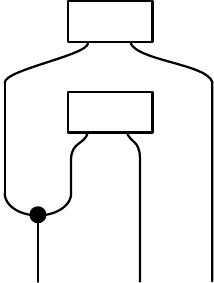}{0.6}
	 \put(-55,-50){\scriptsize$H$} \put(-25,-50){\scriptsize$H$} \put(-7,-50){\scriptsize$H$}
	 \put(-37,33){$\omega_H$}  \put(-37,6){$\omega_H$} 
\\[1em]
&   \ipic{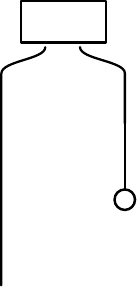}{.6}
	 \put(-43,-50){\scriptsize$H$}  \put(-28,33){$\omega_H$}
	 \quad = \quad \ipic{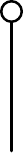}{.6}
	 \put(-8,-30){\scriptsize$H$} 
	 \qquad , \qquad
	 \ipic{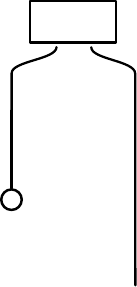}{.6}
	 \put(-6,-50){\scriptsize$H$}  \put(-26,33){$\omega_H$}
	 \quad = \quad
	 \ipic{hopf-pair-prop-3-4-rhs}{.6}
	 \put(-8,-30){\scriptsize$H$}
\end{split}	 
\ee
Translating back into formulas, for example the first of the above identities becomes
 \begin{align}
& \Big[(HH)H\xrightarrow{\mu_H\otimes\id} HH \xrightarrow{\omega_H} \tid\Big] 
\\ \nonumber
&  \qquad = \Big[(HH)H  \xrightarrow{\id\otimes\id\otimes\Delta_H} (HH)(HH) \xrightarrow {\; \sim \;} H((HH)H)
\\ \nonumber
& \hspace{5em} \xrightarrow {\id\otimes\omega_H\otimes\id} H(\one H)
  \xrightarrow {\; \sim \;} HH \xrightarrow{\omega_H} \tid
  \Big] \ .
\end{align}

If the braided monoidal category $\cat$ is equipped with left duals, for each Hopf algebra $H$ in $\cat$ we obtain the {\em (left) dual Hopf algebra} $H^*$. Its structure maps are
 \begin{align}
 \mu_{H^*} &= \Big[ H^* H^* \xrightarrow{\gamma_{H,H}} (H H)^* \xrightarrow{(\Delta_H)^*} H^* \Big]\ , \\ \nonumber
 \Delta_{H^*} &=  \Big[H^* \xrightarrow{(\mu_H)^*} (H  H)^* \xrightarrow{\gamma^{-1}_{H,H}} H^*  H^*\Big] \ ,
 \\ \nonumber
 \eta_{H^*} &= \Big[ \one \xrightarrow{\sim} \one^* \xrightarrow{(\eps_H)^*} H^* \big] \ ,
 \\ \nonumber
 \eps_{H^*} &= \Big[ H^* \xrightarrow{(\eta_H)^*} \one^*\xrightarrow{\sim} \one \big] \ ,
 \end{align}
where we used the isomorphism $\gamma_{H,H}$ from \eqref{eq:gammaUV} and the isomorphism \eqref{eq:iso11*}. The antipode is given by $S_{H^*} = (S_H)^*$.

Given a Hopf pairing $\omega_H$ on $H$, we can define the map
\be\label{eq:Omega}
	\mathcal{D}_H ~:=~
	\big[\,
	H \xrightarrow{\sim}
	H\one \xrightarrow{\id\ot \coev_H}
	H(HH^*) \xrightarrow{\sim}
	(HH)H^* \xrightarrow{\omega_H\ot\id}
	\one H^* \xrightarrow{\sim} H^*
	\,\big]	
\ee
The definitions above are set up such that $\mathcal{D}_H$ is a homomorphism of Hopf algebras.

\begin{definition}\label{def:non-deg-Hopf-pair}
Let $H$ be a Hopf algebra in a braided monoidal category with left duals.
A Hopf pairing $\omega_H$ for $H$ is called \textit{non-degenerate} if the morphism $\mathcal{D}_H$ in~\eqref{eq:Omega} is an isomorphism. 
\end{definition}

In Sections \ref{sec:univ-Hopf-finite}--\ref{sec:SL2Z-quasiHopf} we will also need  integrals and cointegrals for a Hopf algebra $H$ in $\cat$.
In general, these are morphisms 
$\Lambda_H:{\rm Int} H \to H$ and $\coint_{H} : H \to {\rm Int} H$ for an invertible object ${\rm Int} H$ (see e.g.\ \cite[Sec.\,4.2.3]{Kerler:2001}). 
In this paper we will only need to consider
${\rm Int} H = \tid$, and in this case, the conditions satisfied by $\Lambda_H:\one \to H$ and $\coint_{H} : H \to \tid$ are:
\begin{align} \label{eq:integral-def-general}
&\text{(left integral)} &
\mu_H\circ(\id_H\otimes\Lambda_H)\circ \runit^{-1}_H &~=~ \Lambda_H\circ \eps_H \ ,
& \qquad \\ \nonumber
&\text{(right integral)}&
\mu_H\circ(\Lambda_H\otimes \id_H) \circ \lambda_H^{-1} &~=~ \Lambda_H\circ \eps_H \ ,
 \\ \nonumber
&\text{(left cointegral)}&
\runit_H \circ (\id_H\tensor\coint_H)\circ \Delta_H &~=~ \eta_H \circ \coint_H \ ,
 \\ \nonumber
&\text{(right cointegral)}&
\lambda_H \circ (\coint_H\tensor\id_H)\circ \Delta_H &~=~ \eta_H \circ \coint_H \ .
\end{align}

One can use the integrals and cointegrals to test non-degeneracy of a Hopf pairing~$\omega_H$:

\begin{lemma}[\cite{Kerler:1996}]\label{lem:coint-from-int}
Assume that  $H\in\cat$ is a Hopf algebra with a
right  cointegral $\coint_H: H\to \tid$ and a Hopf pairing $\omega_H$.
The Hopf pairing $\omega_H$ is non-degenerate iff there exists a morphism $\Lambda_H: \tid \to H$ such that the  cointegral $\coint_H$ factors through $\omega_H$:
$$
\coint_H = \Big[H\xrightarrow{\sim} \tid  H \xrightarrow{\Lambda_H\ot\id} H  H \xrightarrow{\omega_H} \tid \Big]\ .
$$
If such a $\Lambda_H$ exists, it is automatically a left  integral for $H$. A similar statement can be made for left cointegrals.
\end{lemma}

\subsection{Reconstruction of a Hopf algebra} 
\label{sec:Hopf-alg-C}

Given a Hopf algebra $\rA$ over a field $\ok$, let $\corep_{\ok}\rA$ denote the category of finite-dimensional 	right
corepresentations of a $\rA$ over $\ok$.
We have the following well-known reconstruction theorem~\cite{Ulbrich} (see also~\cite[Thm.\,5.1.11]{ChPr}):

\begin{theorem}
Let $\cat$ be 	a
$\ok$-linear abelian monoidal category with left duals and let $\fun : \cat \to \vect_\ok$ be a fiber functor
(i.e.\ a $\ok$-linear exact faithful monoidal functor into the category of finite-dimensional $\ok$-vector spaces). Then there exists a Hopf algebra $\rA$ over $\ok$ such that $\fun$ factors monoidally as $\cat \to \corep_{\ok}\rA \xrightarrow{\mathrm{forget}} \vect_\ok$, where  $\cat \to \corep_{\ok}\rA$ is a $\ok$-linear equivalence of monoidal categories.
\end{theorem}

This theorem can be formulated in a more general context \cite[Thm.\,2.2]{Majid:1993}, which we now review. Let $\cat$ be 
	a
monoidal category with left duals, let $\catV$ be a
braided monoidal category with left duals, and let $\fun : \cat \to \catV$ be a monoidal functor (so $\catV$ is a replacement for $\vect_\ok$ while $\fun$ is a replacement for the fiber functor). 
In this situation there is a universal property which characterises 
 a Hopf algebra $\rA$ internal to~$\catV$ such that $\fun$ factors monoidally as $\cat\to \corep_{\catV}\rA \xrightarrow{\text{forget}} \catV$. Here, $\corep_{\catV}\rA$ is the category of $\rA$-comodules internal to $\catV$. 
Existence of $\rA$ is guaranteed under certain completeness conditions \cite{Majid:1993}, and we describe such a situation in Section \ref{sec:univ-Hopf-exists} below.
 
If $\cat$ is in addition braided, one has the natural choice $\catV=\cat$, $\fun=\id$, and we obtain a universal property characterising a Hopf algebra in $\cat$ which only depends on $\cat$. We will refer to this Hopf algebra as the {\em universal Hopf algebra for $\cat$}.

\subsubsection{The universal Hopf algebra $\rA$}\label{sec:rA}
Let $\cat$ be 	a
braided monoidal category with left duals.
We now review the construction in \cite{Majid:1993} in the special case $\catV=\cat$ and $\fun = \id$. As an object in $\cat$, $\rA$ is defined to  represent the functor 
$\idt\colon \cat\to \Set$ 
which on objects is given by
\begin{equation}\label{eq:idt}
\idt: \; V\mapsto \Nat(\id,\id\tensor V)\ .
\end{equation}
Here, $\id\tensor V: \cat\to\cat$ is the functor that sends an object to its tensor product with $V$. 
	By the overall assumption that all our categories are essentially small, the functor $\natf$ does indeed land in $\Set$.

If a representing object exists, 
by definition it is equipped with a family of natural isomorphisms
\begin{equation}\label{eq:natiso}
\natiso_V: \quad \cat(\rA,V) \; \xrightarrow{\quad\sim\quad}\; \idt(V)=\Nat(\id,\id\tensor V)\ ,
\end{equation}
and the pair $(\rA,\natiso_V)$ is uniquely defined up to a unique isomorphism.

In particular, for $V=\rA$ we have the natural transformation
\be\label{eq:nattr}
\nattr:=\natiso_\rA(\id)\ .
\ee
In terms  of the morphisms $\nattr_X: X\mapsto X\tensor \rA$, which are natural in $X$, we define the Hopf algebra structure on $\rA$ as follows.
\begin{enumerate}
\setlength{\leftskip}{-1.5em}
\item The \textit{comultiplication} $\Delta_\rA: \rA\to \rA\tensor \rA$ is defined by
\begin{equation}\label{eq:rA-cop}
 \natiso_{\rA\tensor\rA}(\Delta_\rA)_X ~=~ \big[ \, X \xrightarrow{\nattr_X} X \rA  \xrightarrow{\nattr_X\tensor\id_\rA} (X  \rA)   \rA \xrightarrow{\sim} X  (\rA  \rA) \,   \big] \ .
\end{equation}

\item The \textit{counit} $\eps_\rA: \rA\to \tid$ is defined via the right unit isomorphism of $\cat$,
\be\label{eq:rA-counit}
	\natiso_{\tid}(\eps_\rA)_X ~=~ \big[\, X \xrightarrow{\sim} X\one \, \big] \ .
\ee
\end{enumerate}

In order to define the multiplication, we need to consider the functor $\idt^2 : \cat \to \Set$ acting on objects as
\be
\idt^2: \; V\mapsto \Nat(-\tensor-,(-\tensor -)\tensor V) \ ,
\ee
	where the natural transformations are between two functors $\cat \times \cat \to \cat$. $\idt^2$
is a kind of square version of $\idt$. It is shown in \cite[Lem.\,2.3]{Majid:1993} that this functor is representable by $\rA\tensor\rA$. The corresponding natural isomorphisms
\begin{equation}
\natiso^2_V: \quad \cat(\rA\tensor\rA,V) \; \xrightarrow{~\sim~}\; \idt^2(V) 
\end{equation}
are given by
\begin{multline}\label{eq:natiso2}
\natiso^2_V(f)_{X,Y} ~=~ \big[\, X Y \xrightarrow{\nattr_{X}\tensor \nattr_{Y}} (X \rA) (Y  \rA) \xrightarrow{\, \sim \,} 
X ((\rA   Y)   \rA)\\
\xrightarrow{\id\tensor \brC_{A,Y}\tensor \id} X  ((Y  \rA)   \rA) \xrightarrow{\, \sim \,}  X  (Y  (\rA   \rA))
  \xrightarrow{\id\tensor f} X  (Y V) \xrightarrow{\, \sim \,} (XY)V \, \big]   \ .
\end{multline}
\begin{enumerate} \setlength{\leftskip}{-1.5em} \setcounter{enumi}{2}
\item
\textit{The multiplication} $\mu_\rA\in\cat(\rA\tensor\rA,\rA)$ is defined by the equality
\begin{equation}\label{eq:rA-mult}
\natiso^2_{\rA}(\mu_{\rA})_{X,Y}  ~=~ \big[\,
XY \xrightarrow{\nattr_{X\tensor Y}} (XY)\rA \,\big] \ .
\end{equation}

\item \textit{The unit} is given by $\eta_\rA: = \lambda_\rA\circ\nattr_{\tid}: \tid\to \rA$, where $\lambda_\rA: \tid\tensor \rA\to \rA$ is the left-unit isomorphism of $\cat$.

\item \textit{The antipode} $S_\rA: \rA\to \rA$ is defined by the equality
\begin{multline}\label{eq:rA-S}
\natiso_\rA(S)_X ~=~   \big[\, X 
 \xrightarrow{\sim} 
 \tid X  \xrightarrow{\coev_X\tensor \id} (X X^*) X 
\xrightarrow{\id\tensor\nattr_{X^*}\tensor\id} (X (X^* \rA))X  \\
\xrightarrow{\id\tensor\brC_{\rA,X^*}^{-1}\tensor \id} 
(X ( \rA X^* ))X   \xrightarrow{\, \sim \,} X(\rA(X^* X))  \xrightarrow{\id\tensor\ev_X} X(\rA \tid) 
\xrightarrow{\sim}  
X \rA \,\big] \ .
\end{multline}
\end{enumerate}

The following theorem is shown in \cite[Sec.\,2]{Majid:1993} (in the more general case of monoidal functors $\fun:\cat\to\catV$).

\begin{theorem}[Part 1]\label{thm:univ-Hopf}
Let $\rA$ be an object representing  $\idt$ in~\eqref{eq:idt}. Then (1)-(5) endow $\rA$ with the structure of a Hopf algebra in $\cat$. 
\end{theorem}

The above construction also provides us with a functor
 $\mathcal{R}: \cat \to \corep_{\cat}\rA$. Namely, for each $X \in \cat$ consider the morphism $\tilde\iota_X : X \to X \otimes \rA$. It is immediate from the definition of $\rA$ that this defines a right comodule structure on $X$. Naturality of $\tilde\iota$ implies that morphisms in $\cat$ become comodule morphisms in $\corep_{\cat}\rA$. 
Furthermore, it is straightforward to check that $\mathcal{R}$ is strictly monoidal (i.e.\ via the identity morphisms $\mathcal{R}(X) \tensor \mathcal{R}(Y) = \mathcal{R}(X \tensor Y)$, $\mathcal{R}(\one) = \one$). 
Altogether we see that the identity functor on $\cat$ factors monoidally as
\be\label{eq:id-factor-corep}
	\id_\cat ~=~ \big[\, \cat \xrightarrow{\;\mathcal{R}\;} \corep_{\cat}\rA \xrightarrow{\text{forget}} \cat \,\big] \ .
\ee
(And this is indeed an equality, not just an equivalence.)

\medskip

\noindent
{\bf Theorem \ref{thm:univ-Hopf}} (Part 2){\bf.} {\em
 The Hopf algebra $\rA$ is universal in the sense that if $\rA'$ is another Hopf algebra in $\cat$ such that the identity functor factors 
	monoidally  
as  $\cat\to\corep_{\!\cat\,} \rA' \to \cat$
   as in \eqref{eq:id-factor-corep}, then there exists a unique Hopf algebra map $\rA\to \rA'$ and thus the corresponding functor $\corep_{\!\cat\,}\rA\to \corep_{\!\cat\,}\rA'$ such that the functor $\cat \to \corep_{\!\cat\,} \rA'$ 
equals the composition $\cat\to\corep_{\!\cat\,}\rA\to\corep_{\!\cat\,} \rA'$.}

\medskip

\subsubsection{Hopf pairing for $\rA$} \label{subsec:Hopf-pair-A}

The ``inverse monodromy'' natural isomorphisms $c_{X,Y}^{-1} \circ c_{Y,X}^{-1} : X \tensor Y \to X \tensor Y$ of $\cat$ can be used to define a
pairing on the universal Hopf algebra $\rA$. Namely, define the morphism $\omega_\rA : \rA \tensor \rA \to \one$ via
\be\label{eq:rA-omega}
 \natiso^2_{\tid}(\omega_{\rA})_{X,Y}  ~=~
 \big[\, 
 XY \xrightarrow{\brC_{Y,X}^{-1}} YX \xrightarrow{\brC_{X,Y}^{-1}} XY
 \xrightarrow{\sim} (XY)\one
 \,\big]\ ,
 \ee
where the family $\natiso^2_{\tid}(f)_{X,Y}$ was defined in~\eqref{eq:natiso2}. 
The following proposition is a corollary to Theorem \ref{thm:coend-Hopf+pairing} and Proposition \ref{prop:univ-vs-coend} below, where we give an alternative description of the universal Hopf algebra $\rA$ and of the pairing $\omega_\rA$ in terms of coends.

\begin{proposition}\label{prop:univ-hopf-pairing}
Suppose that the universal Hopf algebra $\rA \in \cat$ exists. Then
the pairing $\omega_\rA : \rA \tensor \rA \to \one$ is a Hopf pairing for $\rA$.
\end{proposition}

We remark that if one were to use the monodromy isomorphism $c_{Y,X} \circ c_{X,Y}$ instead of the inverse monodromy in \eqref{eq:rA-omega} one would not obtain a Hopf pairing in the sense of \eqref{eq:hopf-pair-prop}.
 
In particular, if the universal Hopf algebra exists it is automatically equipped with a Hopf algebra map 
\be
\mathcal{D}_\rA:\; \rA\to \rA^* \ ,
\ee
as defined in~\eqref{eq:Omega}.
	We note that in case $\cat$ is the category of representations of a Hopf algebra over a field, $\mathcal{D}_\rA$ specialises to the Drinfeld map composed with an isomorphism to the double dual (see Remark \ref{rem:fact-qHopf-def}\,(2) below).

Later in this paper we will be particularly interested in situations where $\mathcal{D}_\rA$ is invertible, or, equivalently, where $\omega_\rA$ is non-degenerate (recall Definition \ref{def:non-deg-Hopf-pair}).

\section{The universal Hopf-algebra via coends} 
\label{sec:univ-hopf-via-coends}

In this section, we first recall standard facts about coends in general. We then review a construction due to \cite{Lyubashenko:1995} which gives a Hopf algebra structure together with a Hopf pairing on a certain coend for a braided monoidal category with left duals. 
We elaborate on the observation in \cite{Lyubashenko:1995}
that this coend provides an alternative description of the universal Hopf algebra from Section \ref{sec:Hopf-alg-C}.\footnote{
  In \cite[p.\,289]{Lyubashenko:1995}, the relation between the universal Hopf algebra and the coend is explained by: ``In [9] a Hopf algebra $\mathrm{Aut}\,\cat$
was introduced. It can be represented as a coend $\mathrm{Aut}\cat \cong F = \int^{X \in \cat} X \otimes X^\vee$.'' We hope it to be helpful to provide some additional details in Section \ref{sec:univHopf-coend-iso}.} 

Using the coend description, in \cite{Lyubashenko:1995} a projective $SL(2,\oZ)$ is defined on the space of invariants of $\rA$, which we recall in Section \ref{sec:SL2Z-action}.

\subsection{The universal property of coends} \label{subsec:uni-prop-coends}

Let $\cat$ and $\catD$ be any categories, denote by $\catop$ the opposite category with reversed morphisms, and let $F\colon \cat^{\operatorname{op}} \times \cat \to \catD$ be a functor. We recall the definition of a dinatural transformation from the functor $F$ to an object $B\in\catD$ in Appendix~\ref{app:dinat-tr} and recall here the notion of a coend~\cite{MacLane-book}.

\begin{definition} \label{def:coend} 
A \emph{coend} $(C,\iota)$ of a functor $F\colon \cat^{\operatorname{op}} \times \cat \to \catD$ is an object $C\in\catD$ endowed with a dinatural transformation 
$\iota\colon F \stackrel{..}{\longrightarrow} C$,
see Definition~\ref{def:dinat-const} (1),
satisfying the following universal property: 
for any dinatural transformation $\phi\colon F \stackrel{..}{\longrightarrow} B$ there is a unique $g\in{\catD}(C,B)$
such that the following diagram commutes 
for all $U \in \cat$:
\be\label{eq:din-diag}
\xymatrix{
& F(U,U) \ar[dl]_{\iota_U} \ar[dr]^{\phi_U} \\
C \ar@{-->}[rr]^{\exists!\,g} && B
}
\ee
or, in other words, any dinatural transformation $\phi\colon F \stackrel{..}{\longrightarrow} B$
factors through the coend of $F$ in a unique way: $\phi_U = g\circ\iota_U$ for all $U \in \cat$.
\end{definition}

A coend is unique up to unique isomorphism, so that we may refer to `the coend'.
A common notation for the coend is $C=\int^{U\in\cat} F(U,U)$.
 For brevity, we will often just write  $C$ for the coend instead of the pair $(C,\iota)$.
 
 We will also need multiple or iterated coends $\int^{U\in\catB} \int^{V\in\cat} F(U,U,V,V)$ of a functor
\be
 F\colon \catB^{\operatorname{op}} \times \catB  \times \cat^{\operatorname{op}} \times \cat \to \catD\ .
\ee 
 These are defined by considering first the functor $F$ as the functor $\tilde F\colon     \cat^{\operatorname{op}} \times \cat \to \Fun (\catB^{\operatorname{op}} \times \catB, \catD)$ to the category of functors from $\catB^{\operatorname{op}} \times \catB$ to $\catD$, and assuming that the coend of $\tilde F$ exists as an object in this category of functors.
We can then consider the coend of this coend-object $\int^{V\in\cat} F(-,-, V,V)$, which is by definition the iterated coend of $F$ from above. 
Alternatively, we could first take the coend (or ``integration'') over objects in $\catB$ as the object $\int^{U\in\catB} F(U,U, -,-)$ in the category of functors from  $\cat^{\operatorname{op}} \times \cat$ to $\catD$ and take then the corresponding coend (or ``integration'') over objects in $\cat$. This gives another iterated coend 
$ \int^{V\in\cat}  \int^{U\in\catB}F(U,U,V,V)$. 
Finally, one can consider the ``double coend'', that is, the coend for the functor 
\be
(\catB \times \cat)^{\operatorname{op}} \times (\catB  \times \cat) \xrightarrow{~\sim~}
\catB^{\operatorname{op}} \times \catB  \times \cat^{\operatorname{op}} \times \cat
\xrightarrow{~F~}
\catD \ ,
\ee
which we write as $\int^{(U,V)\in\catB\times\cat}F(U,U,V,V)$.
The iterated coends and the double coend can be compared by 
	a ``Fubini theorem'' (see e.g.~\cite[IX.8]{MacLane-book}).

\begin{proposition}\label{prop:coend-Fubini}
Let $\catB,\cat$ be categories and let $F\colon \catB^{\operatorname{op}} \times \catB  \times \cat^{\operatorname{op}} \times \cat \to \catD$ be a functor. Consider the three coends
$$
 \int^{U\in\catB} \int^{V\in\cat} F(U,U,V,V) \quad,\quad  \int^{(U,V)\in\catB\times\cat}F(U,U,V,V) \quad,\quad  \int^{V\in\cat}  \int^{U\in\catB}F(U,U,V,V)\ .
$$
If any one of them exists, then so do the other two, and all three are canonically isomorphic. 
\end{proposition}

We can similarly define higher iterated and multiple coends, up to a unique isomorphism.

\begin{remark} \label{rem:coend}
\mbox{}
\begin{enumerate}\setlength{\leftskip}{-1.5em}
	\item  We can define the category $DIN(F)$ of dinatural transformations for $F$: objects are pairs $(B,\iota)$  and morphisms between two dinatural transformations $(B,\iota)$ and $(B',\phi)$ are defined as
		${DIN(F)}\bigl((B,\iota),(B',\phi)\bigr)\coloneqq \{ f\in{\catD}(B,B') : \phi=f\circ \iota \}$.
		Coends  are then the initial objects in $DIN(F)$.
	(And {\em ends} are terminal objects, but we only use these in Section \ref{sec:fact-end-coend} below.)

	\item We will later use the following important property of a coend $C$: to define a morphism $C\to B$ (e.g. for $B=C\otimes C$ below) is enough to fix a morphism from $F(U,U)$ to~$B$ for all $U\in\cat$ and such that it is dinatural. This is due to the universal property of $C$: 
 there is a one-to-one correspondence between the set $\Din(F,B)$ of dinatural transformations $\iota\colon F \stackrel{..}{\longrightarrow} B$ and the set $\catD(C,B)$.
	Similarly, 
	by Proposition \ref{prop:coend-Fubini} 
the 	iterated coend 
$C=\int^{U\in\catB} \int^{V\in\cat} F(U,U,V,V)$ of a functor $F\colon \catB^{\operatorname{op}} \times \catB  \times \cat^{\operatorname{op}} \times \cat \to \catD$ has the following universal property:
	any  transformation of $F$ dinatural in both the arguments (for $\catB$ and $\cat$) to an object in $\catD$ factors uniquely through the iterated coend.
	
\end{enumerate}
\end{remark}

\subsection{The coend $\coend$} \label{subsec:coends}

Let $\cat$ be a braided monoidal category with left duals.
We will now focus on the coend of the functor
\be \label{eq:F}
 F~ :=~ \big[\, \catop \times \cat \xrightarrow{(-)^\ast\times \id_{\cat}} \cat \times \cat \xrightarrow{\quad\tensor\quad} \cat  \,\big] \ ,
\ee
i.e.\ the functor which acts on objects and morphisms as $(U,V)\mapsto U^\ast \tensor V$ and $(f,g)\mapsto f^\ast\tensor g$.
We denote this coend as
\begin{equation}\label{eq:coend}
\coend:= 
 \int^{U\in\cat} U^\ast\tensor U  \ ,
\end{equation}
and the corresponding family of dinatural transformations as 
\begin{equation}\label{eq:iota}
	\iota_X : X^* \otimes X \longrightarrow \coend \quad , \quad  X \in \cat \ .
\end{equation}
We will abbreviate the morphism $\iota_X$ in string diagrams as
\be\label{eq:iota-pic}
   \iota_X\ \coloneqq
   \ipic{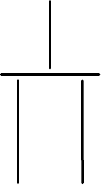}{0.9}
	 \put(-26,43){\scriptsize$\coend$}
	 \put(-40,-48){\scriptsize$X^\ast$} \put(-12,-48)  {\scriptsize$X$}
\ee
Dinaturality means here that for all $X,Y\in\cat$ and $f\in\cat(X,Y)$ we have
\be\label{eq:dinat-iota-f}
\iota_Y \circ (\id_{Y^\ast} \ot f) = \iota_X \circ (f^\ast \ot \id_X) \ .
\ee
In string diagram notation, this equality reads
\be\label{eq:P-iota-pic} 
   \ipic{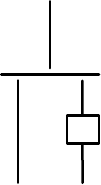}{0.9}
	 \put(-24,42){\scriptsize$\coend$} 
	 \put(-37,12){\scriptsize$\iota_{Y}$}
	 \put(-10,-18){\scriptsize$f$}
	 \put(-40,-48){\scriptsize$Y^\ast$} 
	 \put(-12,-48){\scriptsize$X$}
	 \quad = \quad
	 \ipic{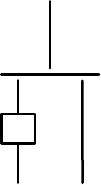}{.9}
	 \put(-24,42){\scriptsize$\coend$} 
	 \put(-37,12){\scriptsize$\iota_{X}$}
	 \put(-40,-18){\scriptsize$f^*$}
	 \put(-40,-48){\scriptsize$Y^\ast$} 
	 \put(-12,-48){\scriptsize$X$} \gp
 \ee
\begin{remark}\label{rem:double-coend}
Consider the functor  
 $F : \catop \times \cat  \times \catop \times \cat \xrightarrow{\quad}\cat$ 
 which is defined on objects and morphisms as
\begin{equation}
F:\quad  
(U,V, X,Y)\mapsto U^\ast \tensor V \tensor X^\ast\tensor Y \ ,\qquad (f,g, h,k)\mapsto f^\ast\tensor g\tensor h^\ast\tensor k \ .
\end{equation}  
It follows from the relation to iterated coends in Proposition \ref{prop:coend-Fubini} that the double coend is given by
\be
	\int^{(U,V) \in\cat\times \cat} 
	F(U,U,V,V) ~=~ \coend\tensor \coend \ ,
\ee
with dinatural family $\iota_{U,V}=\iota_U\tensor \iota_V$. As explained in  Remark~\ref{rem:coend} (2), by the universal property of $\coend\otimes\coend$, any dinatural family $\phi_{U,V}$ from $U^\ast\tensor U\tensor V^\ast \tensor V$ (dinatural in $U$ and $V$) to an object $B\in\cat$ uniquely factors through  a map $g:\coend\tensor \coend\to B$ as $\phi_{U,V} = g\circ (\iota_U\tensor \iota_V)$. 
\end{remark}

\subsection{Hopf algebra structure and Hopf pairing on $\coend$} \label{subsec:Hopf-alg-structure-L}

Let again $\cat$ be a braided monoidal category with left duals and assume that the coend $\coend$ defined in \eqref{eq:coend} exists. 
Following \cite{Lyubashenko:1995} we now use the universal properties of the coends $\coend$ and $\coend \otimes \coend$ 
(as in Remark~\ref{rem:coend}\,(2) and Remark~\ref{rem:double-coend}) to  define the structure morphisms of a Hopf algebra on $\coend$ and 
endow it with a Hopf pairing. For example,
instead of giving the map $\Delta_\coend: \coend \to \coend\tensor \coend$ explicitly, we give the corresponding dinatural transformation $U^\ast\tensor U \to \coend\tensor\coend$, etc.

In giving the dinatural transformations
 defining the product, coproduct, etc.\ on $\coend$, we will write out all coherence isomorphisms of $\cat$ explicitly. This involves choices and we will use the precise form given below to derive explicit expressions for all structure morphisms in the case of representations of a 
	quasi-triangular
quasi-Hopf algebra in Section \ref{sec:coend-repA}. Other combinations of associator and unit isomorphisms  lead to the same structure morphisms, but the explicit formulas in terms of the data of the quasi-Hopf algebra would look differently.

We will need Drinfeld's canonical isomorphism $\uiso_V : V\to V^{**}$ between $V$ and its double dual $V^{**}$, as well as a variant of it which uses the inverse braiding,
 and which we call $\tilde\uiso_V$:
\be\label{def:Drinfeld-maps}
\uiso_V ~=~ \ipic{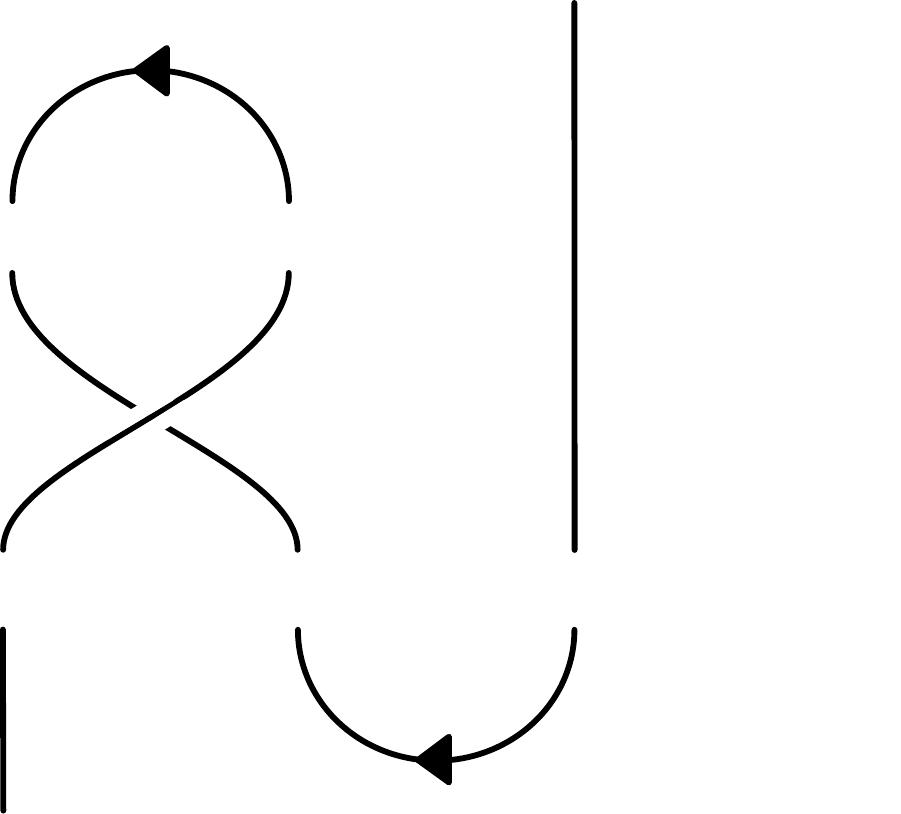}{0.27} 	
\put(-119,-27){\scriptsize$V$} \put(-81,-27){\scriptsize$V^*$} \put(-45,-27){\scriptsize$V^{**}$} \put(-119,19){\scriptsize$V^*$} \put(-81,19){\scriptsize$V$}
 ,  \qquad
\tilde\uiso_V ~=~ \ipic{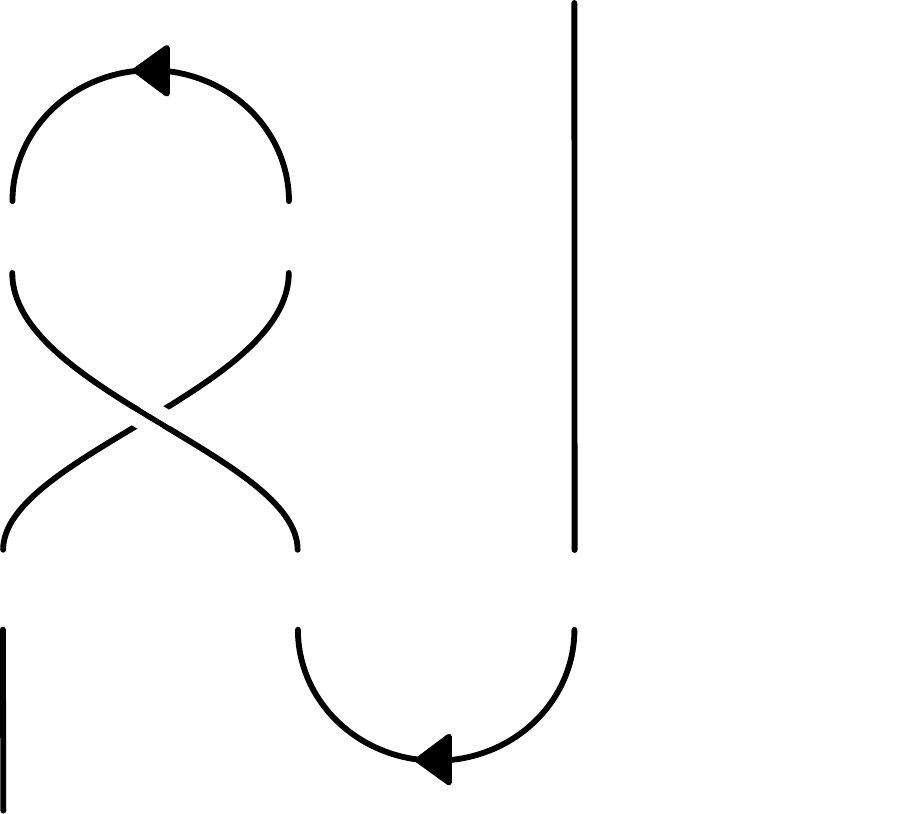}{0.27} 	
\put(-119,-27){\scriptsize$V$} \put(-81,-27){\scriptsize$V^*$} \put(-46,-27){\scriptsize$V^{**}$}  \put(-119,19){\scriptsize$V^*$} \put(-84,19){\scriptsize$V$} 
\put(-30,0){.}
\ee
As a composition of structure morphisms, this reads
\begin{align}\label{eq:uiso}
\uiso_V ~=~ \Big[\, V &\xrightarrow{\runit^{-1}_V} V\tid  \xrightarrow{\id\tensor \coev_{V^*}} V(V^* V^{**})  \xrightarrow{\assoc_{V, V^*,V^{**}}} (V V^*) V^{**}\\
            &  \xrightarrow{\brC_{V,V^*}\tensor \id} (V^* V) V^{**}   \xrightarrow{\ev_V\tensor\id} \tid V^{**}  \xrightarrow{\lambda_{V^{**}} } V^{**}\, \Big] \ , \nonumber \\
\label{eq:uiso-tw}
\tilde\uiso_V ~=~ \Big[\, V &\xrightarrow{\runit^{-1}_V} V\tid  \xrightarrow{\id\tensor \coev_{V^*}} V(V^* V^{**})  \xrightarrow{\assoc_{V, V^*,V^{**}}} (V V^*) V^{**}\\
            &  
\xrightarrow{\brC^{-1}_{V^*,V}\tensor \id} 
(V^* V) V^{**}   \xrightarrow{\ev_V\tensor\id} \tid V^{**}  \xrightarrow{\lambda_{V^{**}} } V^{**}\, \Big] \ . \nonumber
\end{align}
	
\begin{remark} \label{rem:twist}
Suppose for a moment that $\cat$ is in
addition ribbon. Then it is in particular pivotal, and we write $\delta_V : V \to V^{**}$ for the pivotal structure. In this situation, $\uiso_V$ and $\tilde\uiso_V$ are related by $\uiso_V = \delta_V \circ \theta_V^{-1}$ and $\tilde\uiso_V = \delta_V \circ \theta_V$ (cf.\ Section \ref{sec:qHopf-pivot}).
\end{remark}	
	
\begin{figure}[t]
\hspace*{-10em}\begin{minipage}{30em}
\eqpicnn{coend-hopf} {390} {115} {\setlength\unitlength{.85pt}
   \put(42,150) {\scalebox{.2}{\includegraphics[scale=1]{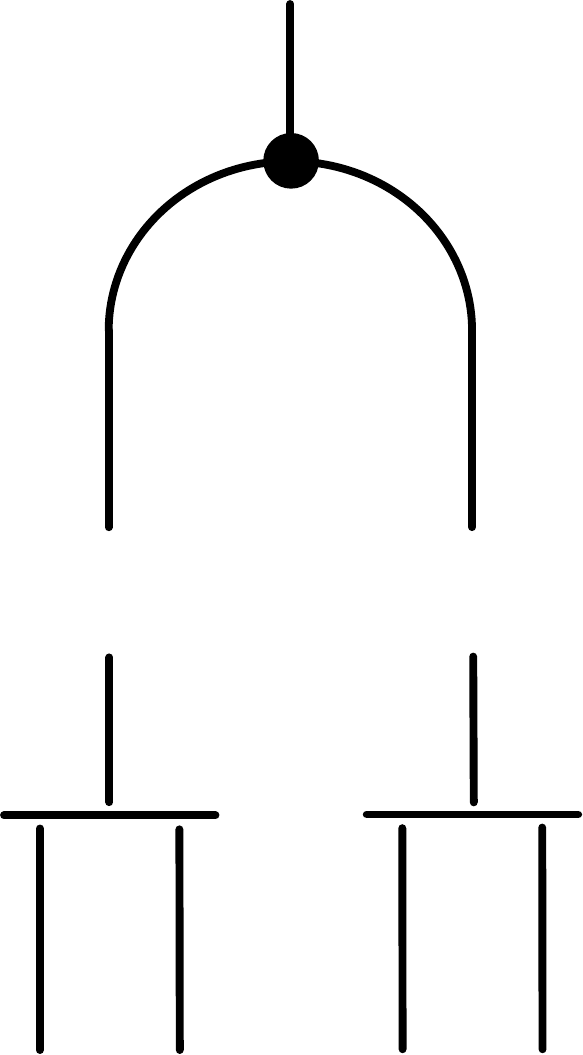}}}
	 \put(71,273){\scriptsize{$\coend$}}
	 \put(80,255){\scriptsize{$\muc$}} 
	 \put(50,198){\scriptsize{$\coend$}} \put(91,198){\scriptsize{$\coend$}}  
   \put(30,170){\scriptsize$\iota_U$} \put(110,170){\scriptsize$\iota_V$}
	 \put(40,140){\scriptsize$U^\ast$} \put(58,140){\scriptsize$U$} \put(81,140){\scriptsize$V^\ast$} \put(99,140){\scriptsize$V$}
   \put(120,194) {$~:=$}
	
   \put(160,150) {\scalebox{.2}{\includegraphics{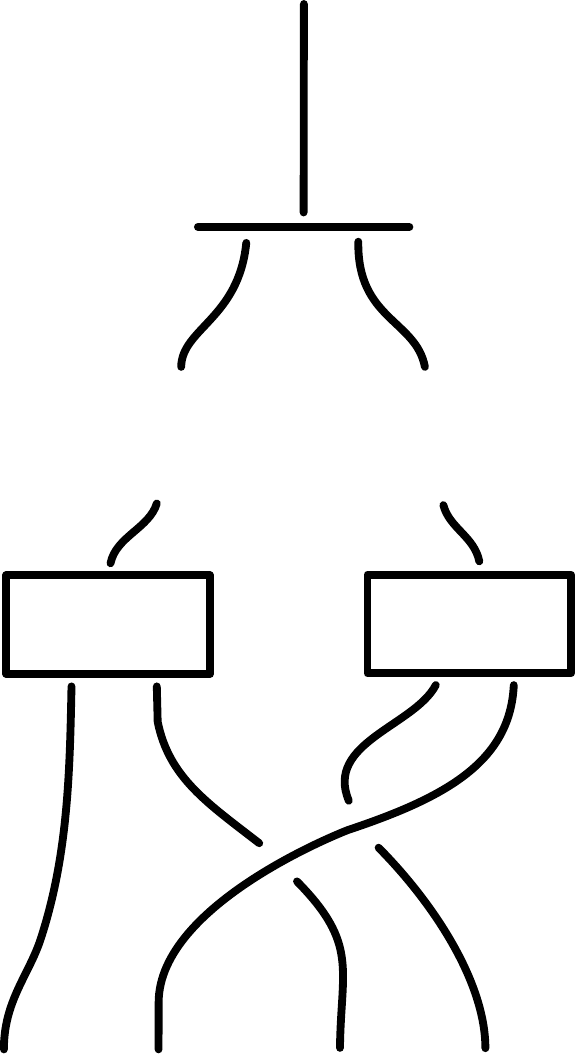}}} 
	 \put(191,273){\scriptsize{$\coend$}} 
   \put(164,250){\scriptsize$\iota_{V\tensor U}$}
	 \put(155,217){\tiny$(V\tensor U)^\ast$} \put(203,216){\tiny$V\tensor U$}
	 \put(161.8,197){\scriptsize$\gamma_{V,U}$} \put(209,195){\scriptsize$\id$}
	 \put(155,140){\scriptsize$U^\ast$} \put(173,140){\scriptsize$U$} \put(193,140){\scriptsize$V^\ast$} \put(211,140){\scriptsize$V$}
	 \put(250,194) {$,$}
	
	 \put(337,150){\scalebox{.2}{\includegraphics{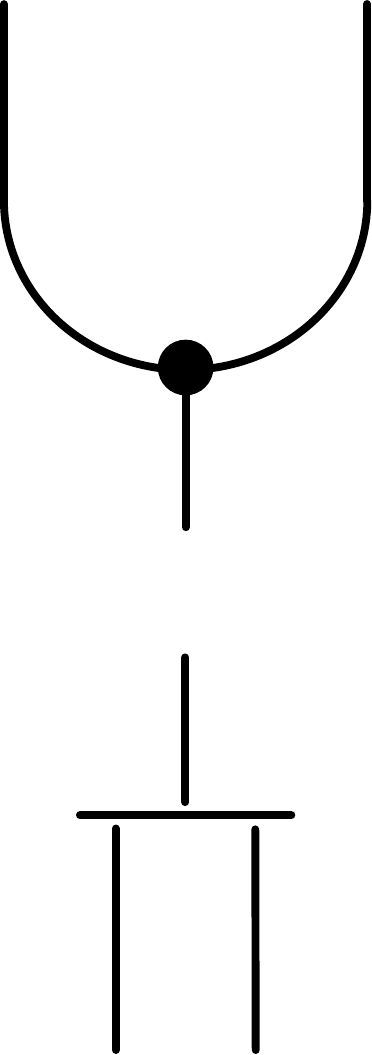}}}
	 \put(334,273){\scriptsize$\coend$} \put(375,273){\scriptsize$\coend$}
	 \put(363,216){\scriptsize{$\Delta_\coend$}} 
	 \put(354,198){\scriptsize{$\coend$}} 
   \put(332,170){\scriptsize$\iota_U$} 
	 \put(342,140){\scriptsize$U^\ast$} \put(361,140){\scriptsize$U$} 
   \put(395,194) {$~:=$}
	
	 \put(435,150) {\scalebox{.2}{\includegraphics[scale=1]{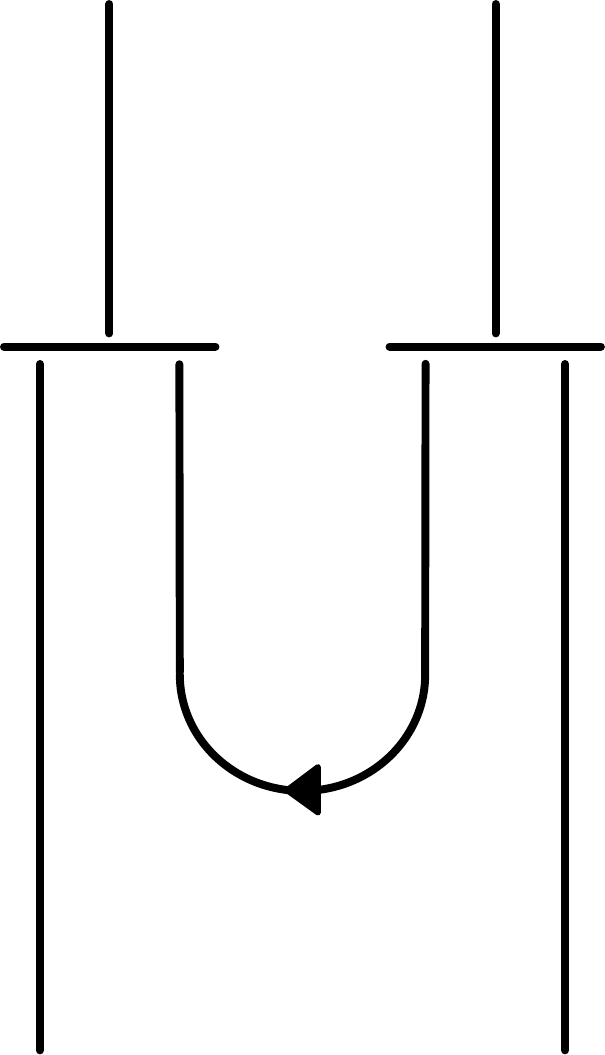}}}
	 \put(443,273){\scriptsize$\coend$} \put(487,273){\scriptsize$\coend$}
   \put(424,232){\scriptsize$\iota_U$} \put(505,232){\scriptsize$\iota_U$} 
	 \put(432,140){\scriptsize$U^\ast$} \put(494,140){\scriptsize$U$} 
	 \put(520,194) {$,$}

	 \put(42,10) {\scalebox{.2}{\includegraphics{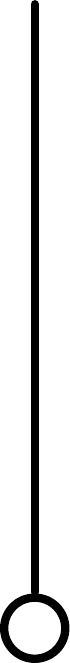}}}
	 \put(42,88){\scriptsize$\coend$} 
	 \put(30,22){\scriptsize$\eta_\coend$} 
	 \put(61,50) {$~:=$}
	
	 \put(95,10) {\scalebox{.2}{\includegraphics{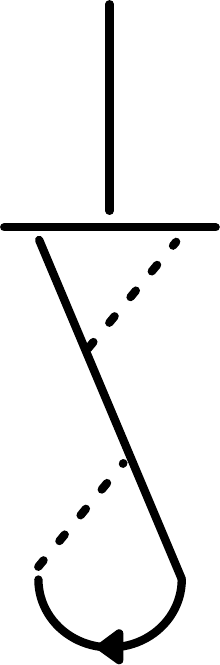}}} 
	 \put(103,88){\scriptsize$\coend$}
   \put(122,60){\scriptsize$\iota_{\tid}$}
	 \put(114,28){\scriptsize$\lunit$} \put(86,37){\scriptsize$\runit^{-1}$}
	 \put(140,50) {$,$}
	
	 \put(190,10){\scalebox{.2}{\includegraphics{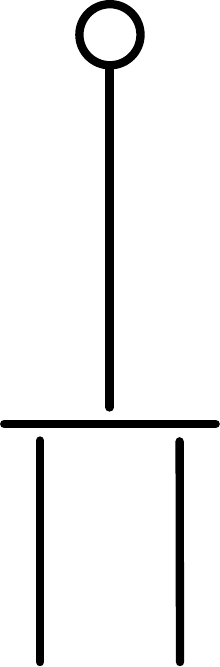}}}
	 \put(207,70){\scriptsize$\eps_\coend$} 
	 \put(180,30){\scriptsize$\iota_U$} 
	 \put(186,0){\scriptsize$U^\ast$} \put(207,00){\scriptsize$U$} 
	 \put(225,50) {$~:=$}
	
	 \put(260,10){\scalebox{.2}{\includegraphics{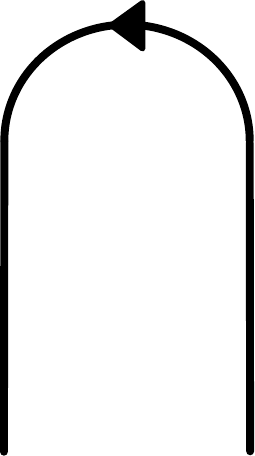}}}
	 \put(254,00){\scriptsize$U^\ast$} \put(284,00){\scriptsize$U$} 
	 \put(305,50) {$,$}
	
	 \put(375,10){\scalebox{.2}{\includegraphics{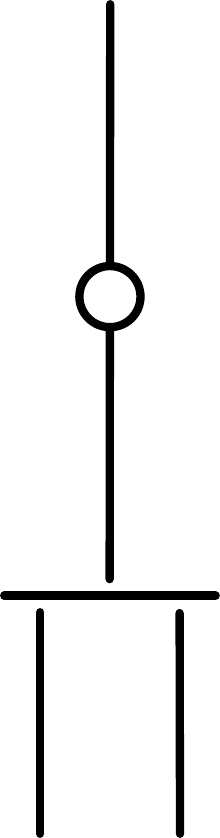}}}
	 \put(384,108){\scriptsize$\coend$}
	 \put(394,74){\scriptsize$S_\coend$} 
	 \put(365,30){\scriptsize$\iota_U$} 
	 \put(373,0){\scriptsize$U^\ast$} \put(392,00){\scriptsize$U$} 
	 \put(408,55) {$~:=$}
	
	 \put(438,10){\scalebox{.2}{\includegraphics{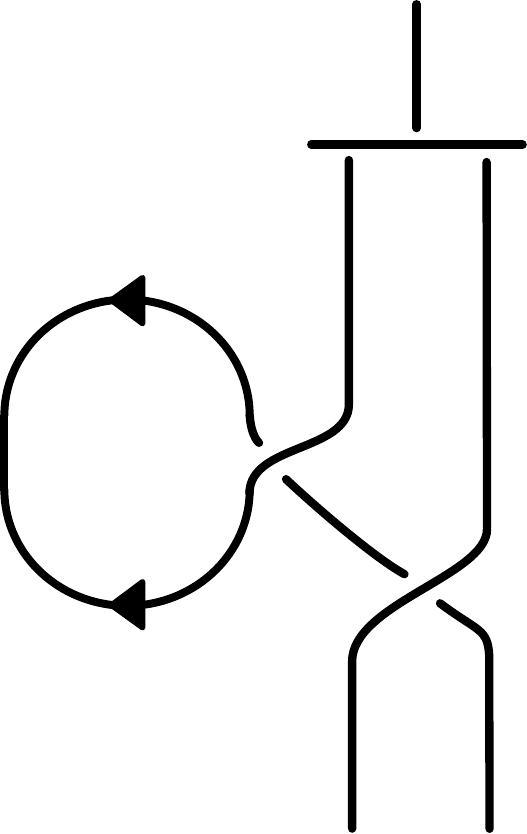}}}
	 \put(482,108){\scriptsize$\coend$}
	 \put(501,91){\scriptsize$\iota_{U^\ast}$} 
	 \put(472,0){\scriptsize$U^\ast$} \put(490,00){\scriptsize$U$} 
	 \put(515,50) {$.$}
}
\end{minipage}
\caption{Hopf algebra structure  on the coend $\coend\in\cat$. Here, $\gamma_{V,U}$ is the canonical isomorphism $U^\ast\tensor V^\ast \to (V\tensor U)^\ast$ defined by~\eqref{eq:gammaUV}.}
\label{fig:Hopf-coend}
\end{figure}
	
Recall 
our Hopf-algebra conventions in Definition~\ref{def:Hopf-C} and
also the isomorphism $\gamma_{U,V}$ from \eqref{eq:gammaUV}. 
The structure morphisms on $\coend$ are defined as follows.
\begin{enumerate}\setlength{\leftskip}{-1.5em}
\item {\em
 (Multiplication)} $\muc : \coend \ot \coend \to \coend$ is determined by the universal property of the double coend $\coend\otimes\coend$ via, for all $U,V \in \cat$,
\begin{align}
\label{eq:mult-L}
\muc \circ (\iota_U\tensor\iota_V) 
&~=~ 
\Big[ \, (U^\ast  U)  (V^\ast  V) \xrightarrow{\assoc^{-1}_{U^\ast,U,V^\ast  V}} U^\ast  (U  (V^\ast  V))
 \\ \nonumber & \hspace{4em}
\xrightarrow{\id\tensor\brC_{U,V^\ast  V}} U^\ast  ((V^\ast  V)  U) \xrightarrow{\id\tensor\assoc^{-1}_{V^\ast,V,U}} U^\ast  (V^\ast  (V  U))
 \\ \nonumber & \hspace{4em}
\xrightarrow{\assoc_{U^\ast,V^\ast,V U}} (U^\ast  V^\ast)  (V  U) \xrightarrow{\gamma_{V,U}\tensor\id} (V  U)^\ast  (V  U)
\xrightarrow{\iota_{V \tensor U}}\coend
 \,\Big] \ .
\end{align}
\item {\em (Unit)} $\eta_\coend : \one \to \coend$ is defined directly as 
\begin{align}
\label{eq:eta-L}
\eta_\coend 
&~=~
\Big[ \,
\one \xrightarrow{\lambda_{\one}^{-1}} \one\one
	\xrightarrow{\eqref{eq:iso11*}^{-1} \otimes \id} \one^*\one 
\xrightarrow{\iota_{\one}} \coend
 \,\Big]
\\ \nonumber
&~\overset{(*)}=~
\Big[ \,
\one 
\xrightarrow{\coev_{\one}} \one\one^*
\xrightarrow{\lambda_{\one^*}} \one^* 
\xrightarrow{\runit_{\one^*}^{-1}} \one^* \one
\xrightarrow{\iota_{\one}} 
\coend
 \,\Big]\ ,
\end{align}
	where (*) follows from the zig-zag identity for $\ev_{\one}$ and $\coev_{\one}$ and naturality of the unit-isomorphisms $\lambda$, $\runit$.
\item {\em (Coproduct)} $\Delta_\coend : \coend \to \coend \ot \coend$ is determined by the universal property of the coend $\coend$ via, for all $U \in \cat$,
\begin{align}
\label{eq:cop-L}
\Delta_\coend \circ \iota_U 
&~=~ 
\Big[ \, U^\ast U \xrightarrow{\id\tensor\lambda_U^{-1}}  U^\ast(\tid U) \xrightarrow{\id\tensor\coev_U\tensor\id} U^\ast((U U^\ast) U)
 \\ \nonumber & \hspace{4em}
\xrightarrow{\id\tensor\assoc^{-1}_{U ,U^\ast,U}} 
 U^\ast(U (U^\ast U)) \xrightarrow{\assoc_{U^\ast,U ,U^\ast U}} (U^\ast U) (U^\ast U) 
\xrightarrow{\iota_U\tensor\iota_U} \coend\tensor\coend  
 \,\Big] \ .
\end{align}
\item {\em (Counit)} $\eps_\coend : \coend \to \one$ is determined by, for all $U \in \cat$,
\begin{align}
\label{eq:eps-L}
\eps_\coend \circ \iota_U  \ ,
&~=~ 
\Big[ \,  U^\ast U \xrightarrow{\; \ev_U \;} \tid  \, \Big]\ .
\end{align}
\item {\em (Antipode)} $S_\coend : \coend \to \coend$ is determined by, for all $U \in \cat$, 
\begin{align}
\label{eq:antipode-L}
S_\coend \circ \iota_U 
&~=~ 
\Big[ \, U^* U \xrightarrow{\brC_{U^*,U}} U U^* \xrightarrow{\tilde\uiso_U\tensor \id} U^{**} U^* \xrightarrow{\iota_{U^*}} \coend  \,\Big]\ .
\end{align}
\end{enumerate}
These defining relations are given in string diagram notation in Figure \ref{fig:Hopf-coend}.

Finally, we define a pairing on $\coend$ using the universal property of $\coend \otimes \coend$ via the condition that for all $U,V \in \cat$,
\begin{align}\label{eq:omega-L}
	\omega_{\coend} \circ (\iota_U \otimes \iota_V)
~=~& \Big[\,
(U^\ast  U)  (V^\ast  V) \xrightarrow{\assoc^{-1}_{U^\ast,U,V^\ast  V}} U^\ast  (U  (V^\ast  V)) 
\xrightarrow{\id\tensor\assoc_{U,V^\ast,V}}  
 U^\ast  ((U  V^\ast)  V) 
\\  \nonumber & \hspace{3em}
\xrightarrow{\id\tensor (\brC_{V^\ast,U}\circ\brC_{U, V^\ast}) \tensor \id } U^\ast  ((U  V^\ast)  V)
\xrightarrow{\id\tensor\assoc^{-1}_{U,V^\ast,V}} 
 U^\ast  (U  (V^\ast  V)) 
\\ \nonumber & \hspace{3em}
 \xrightarrow{\id \otimes \id \otimes \ev_V} 
U^*(U \one) \xrightarrow{\id \otimes \runit_U} U^* U \xrightarrow{\ev_U} \one
\, \Big]  \ .
\end{align} 
In string diagram notation this reads
\be\label{hopf-pair}
   \ipic{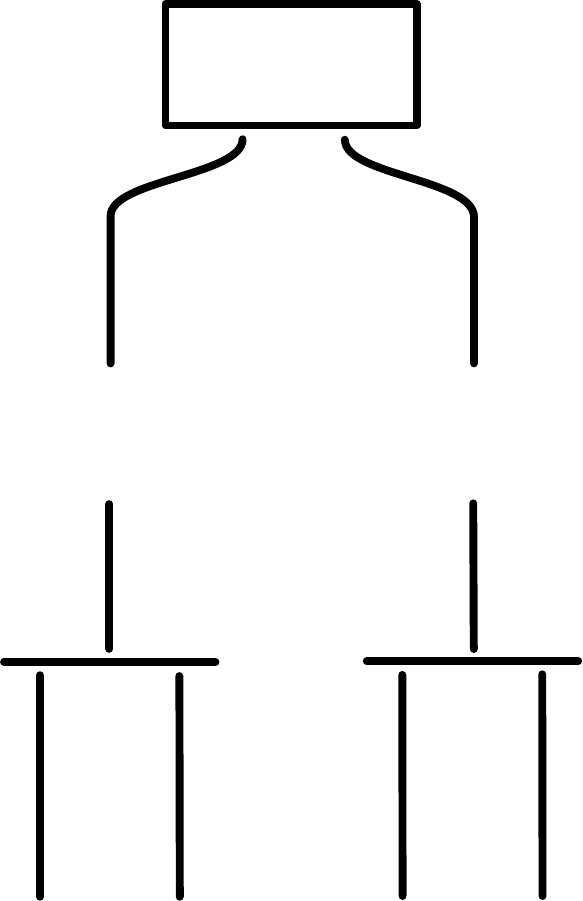}{.22}
	 \put(-55,-1){\scriptsize$\coend$} \put(-15,-1){\scriptsize$\coend$} 
   \put(-62,-18){\scriptsize$\iota_U$} \put(-23,-18){\scriptsize$\iota_V$}
	 \put(-63,-57){\scriptsize$U^\ast$} \put(-45,-57){\scriptsize$U$} \put(-24,-57){\scriptsize$V^\ast$} \put(-6,-57){\scriptsize$V$}
	 \put(-36,39){\footnotesize$\omega_{\coend}$} 
	 \quad\coloneqq\quad \ipic{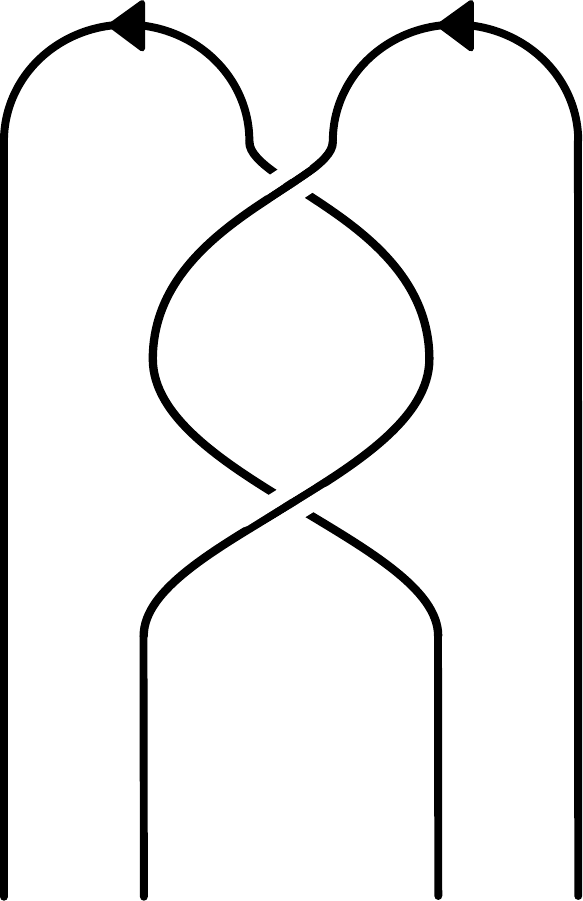}{.23}
	 \put(-67,-59){\scriptsize$U^\ast$} \put(-50,-59){\scriptsize$U$} \put(-22,-59){\scriptsize$V^\ast$} \put(-4,-59){\scriptsize$V$}
\ee

We gather all the structures so far defined on the coend $\coend$ in the following theorem.

\begin{theorem}[\cite{Lyubashenko:1995}]\label{thm:coend-Hopf+pairing}
Let $\cat$ be a braided monoidal category with left duals. If the coend $\coend$ from \eqref{eq:coend} exists, the morphisms defined by \eqref{eq:mult-L}--\eqref{eq:antipode-L} turn it into a Hopf algebra in $\cat$. The pairing \eqref{eq:omega-L} is a Hopf pairing for $\coend$.
\end{theorem}

\subsection{Relation between the universal Hopf algebra in $\cat$ and the coend $\coend$} \label{sec:univHopf-coend-iso}

Let $\cat$ be 	a
braided monoidal category with left duals. At this point we have introduced two Hopf algebras in $\cat$ together with a Hopf pairing, subject to existence of solutions to certain universal properties: on the one hand, the universal Hopf algebra $\rA$ from 
Theorem \ref{thm:univ-Hopf}, and on the other hand the coend $\coend$ from Theorem \ref{thm:coend-Hopf+pairing}. In this section we show that $\rA$ and $\coend$ are canonically isomorphic as Hopf algebras with Hopf pairing.

To describe the isomorphism, we need to consider two functors $\dinf,\natf \colon \cat \to \Set$. Denote by $\Din(F,V)$ the set of dinatural transformations $j\colon F \stackrel{..}{\longrightarrow} V$, 
	where $F = (-)^*\otimes(-) : \catop \times \cat \to \cat$ is the functor from \eqref{eq:F}. 
The first functor is $\dinf := \Din(F,-)$. The second functor is 
$\natf = \Nat(\id,\id\tensor -)$ 
as already defined in \eqref{eq:idt}.
We have the following simple lemma.

\begin{lemma}\label{lem:I-Din}
The family of maps of sets $\zeta_V\colon\dinf(V) \to \natf(V)$, $V \in \cat$, given by, for $j \in \dinf(V)$, 
\begin{equation}\label{eq:tilde}
(\zeta_V(j))_X = \tilde{j}_X ~:=~ \big[\,X \xrightarrow{\sim} \tid X \xrightarrow{\coev_X\tensor \id} (XX^*)X \xrightarrow{\,\sim\,} X(X^* X)
\xrightarrow{\id\tensor j_X} X  V \,\big] \ ,
\end{equation}
defines a natural isomorphism 
	$\zeta : \dinf \to \natf$.
\end{lemma}

\begin{proof}
Using the zig-zag property of the evaluation and coevaluation maps, one easily checks that the map
\be\label{eq:tilde-inv}
\zeta_V^{-1}\colon\quad 
\tilde{j}_X \;\mapsto\; \bigl[X^* X \xrightarrow{\id\tensor \tilde{j}_X}   X^* (X V) \xrightarrow{\,\sim\,} (X^* X) V
 \xrightarrow{\ev_X\tensor \id} \tid V
 \xrightarrow{\sim} V \bigr]
\ee
is the inverse to~\eqref{eq:tilde}.
 The naturality
of $\zeta$ follows from dinaturality of $j$.
\end{proof}

Suppose the coend $(\coend,(\iota_V)_{V \in \cat})$ from \eqref{eq:coend} exists in $\cat$. Recall from Remark \ref{rem:coend}\,(2) that there is a natural isomorphism 
$\cat(\coend,V)\xrightarrow{\;\sim\;} \Din(F,V)$, given by $f \mapsto f \circ \iota$. 
In other words, the coend $\coend$ represents the functor $\dinf$, while by definition the universal  Hopf algebra $\rA$ represents the functor $\natf$, see Section \ref{sec:rA}.
Therefore, we have following corollary to Lemma~\ref{lem:I-Din}.

\begin{corollary}\label{cor:coend-Nat-rep}
If $(\coend,\iota)$ is a coend for 
the functor $F =
(-)^* \otimes (-)$, then $(\coend,\varphi)$ with 
\be
\varphi_V : \cat(\coend,V) \to \Nat(\id,\id \otimes V)
\quad , \quad
(\varphi_V(f))_X = 
(\zeta_V(f \circ \iota))_X : X \to X \otimes V
\ee
represents $\natf$,
with $\zeta_V$ defined in~\eqref{eq:tilde}.

Conversely, if $(\rA,\varphi)$ represents $\natf$, then $(\rA,\iota_X=\zeta^{-1}_\rA(\nattr_X))$,  with $\zeta^{-1}$  from~\eqref{eq:tilde-inv}  and $\nattr$  from~\eqref{eq:nattr}, is a coend for the functor $F$.
\end{corollary}

In particular, the coend $\coend$ exists in $\cat$ if and only if the representing object $\rA$ exists.

\begin{proposition}\label{prop:univ-vs-coend}
Suppose the coend $(\coend,\iota)$ exists, and denote by $(\coend,\varphi)$ the corresponding representing object for $\natf$ obtained from Corollary \ref{cor:coend-Nat-rep}. The Hopf algebra structure morphisms and the Hopf pairing defined on $\coend$ in Theorem \ref{thm:univ-Hopf} and Proposition \ref{prop:univ-hopf-pairing} via the representing object property of $(\coend,\varphi)$ are equal to those defined on $\coend$ in Theorem \ref{thm:coend-Hopf+pairing} via the coend property of $(\coend,\iota)$.
\end{proposition}

The proof of this proposition is given in Appendix \ref{app:univ-vs-coend}.

In what follows we will mostly work with the description of the universal Hopf algebra as a coend, as this is the framework used in \cite{Lyubashenko:1995,Lyubashenko:1994tm} to obtain mapping class group actions on certain Hom spaces, see Section \ref{sec:SL2Z-action} below.

\section{Factorisable finite tensor categories}
	\label{sec:univ-Hopf-finite}

In this section we apply the construction of the universal Hopf algebra in the case of braided finite tensor categories. It is known that for such categories, the universal Hopf algebra exists (see Section \ref{sec:univ-Hopf-exists}). 
The Hopf pairing appears as one of several equivalent ways of characterising factorisability of such a category (see 
	Sections \ref{sec:factorisable-cat} and \ref{sec:fact-end-coend}). 

\subsection{Finite tensor categories} \label{subsec:finite-tens-cats}

Let $\ok$ be a field. Following \cite{Etingof:2003}, by a {\em finite tensor category} $\cat$ we mean:
\begin{itemize}
	\setlength{\leftskip}{-1.5em}
\item 
$\cat$ is a finite abelian $\ok$-linear category. Here {\em finite} means that $\cat$ is equivalent, as a $\ok$-linear category, to the category of finite-dimensional representations of a finite-dimensional $\ok$-algebra.
In particular, $\cat$ is essentially small.
\item
$\cat$ is a rigid monoidal category with simple tensor unit $\one$, such that the tensor product is a $\ok$-linear functor in each argument. As $\cat$ is rigid, the tensor product is automatically exact in each argument, see e.g.~\cite[Prop.\,2.1.8]{BK:2001}.
\end{itemize}
We note that in a finite tensor category, the dual of a projective object is again projective (see e.g.\ \cite[Sec.\,6.1]{EGNO-book}), and so each projective object is also injective. 

The finiteness condition implies the following useful representability property (see e.g.\ \cite[Cor.\,1.10]{DSS14}).

\begin{lemma}\label{lem:represent}
Let $\mathcal{A}$ be a finite $\ok$-linear abelian category and let $\fun : \mathcal{A} \to \vect_\ok$ be a $\ok$-linear left exact functor from $\mathcal{A}$ to finite-dimensional $\ok$-vector spaces. Then $\fun$ is representable, i.e.\ there is $A \in \mathcal{A}$ such that $\mathcal{A}(A,-)$ is naturally isomorphic to $\fun$.
\end{lemma}

This result follows from the more general observation that a $\ok$-linear left exact functor between two finite $\ok$-linear abelian categories admits a left adjoint (see e.g.\ \cite[Cor.\,1.9]{DSS14}). 
Indeed, in the case of the above lemma, if we denote the left adjoint of $\fun$ by $\mathcal{G}$, we have $\fun(X) \cong \Hom_\ok(\ok,\fun(X)) \cong \mathcal{A}(\mathcal{G}(\ok),X)$.

\subsection{Existence of the universal Hopf algebra}\label{sec:univ-Hopf-exists}

Let $\ok$ be a field and let $\cat$ be a $\ok$-linear braided finite tensor category. Using exactness of the tensor product, it is straightforward to verify that the functor $\natf : \cat \to \vect_\ok$ from \eqref{eq:idt} is left exact. 
Lemma \ref{lem:represent} now implies \cite{Majid:1993,Lyubashenko:1995}:

\begin{proposition}\label{prop:univHopf-exists}
The universal Hopf algebra from Theorem \ref{thm:univ-Hopf} exists in $\cat$.
\end{proposition}

Since we will mostly use the coend perspective (Proposition \ref{prop:univ-vs-coend}), we will denote the universal Hopf algebra in $\cat$ by $\coend$ in what follows.

We can describe $\coend$ more explicitly as a cokernel. Since $\cat$ is finite, it contains a projective generator $G$. The coend can be written as a quotient of $G^* \otimes G$: 

\begin{proposition}[{\cite[Cor.\,5.1.8]{Kerler:2001}}]
We have the short exact sequence
\begin{equation}\label{eq:L-P}
0\longrightarrow K \longrightarrow G^\ast\otimes G \xrightarrow{\;\;\;\pi\;\;\;} \coend \longrightarrow 0 \,
\end{equation}
where $K$ is the image of the map 
$
\bigoplus_i (f_i^\ast\otimes\id- \id\otimes f_i):\;  \bigoplus_i G^\ast\otimes G \longrightarrow G^\ast\otimes G 
$
and the  direct sum is taken over a basis $\{f_i \in \End_\cat(G)\}$.
\end{proposition}

Note that $K$ in the above proposition is independent of the choice of the basis in $\mathrm{End}_\cat(G)$.

\newcommand{\surj}{\to}

\begin{remark}
The coend $\coend$ is equipped with the family of dinatural transformations $\iota_X: X^{\ast}\tensor X\to \coend$. The latter can be defined in terms of the  surjective map $\pi$ from~\eqref{eq:L-P} as follows: note first that the map $\pi$ in~\eqref{eq:L-P} is by definition $\iota_G$; we fix then a surjective map $f_X: G^{\oplus m}\surj X$ (for some $m\geq1$) and define $\iota_X$ by the equality of the (composition of the) maps 
\be\label{eq:P-iota}
[X^*\tensor G^{\oplus m} \xrightarrow{\; \id_{X^\ast}\tensor f_X\;} X^*\tensor X  \xrightarrow{\; \iota_X\;} \coend ] 
= [X^*\tensor G^{\oplus m} \xrightarrow{\; f^*_X\tensor \id \;} (G^{\oplus m})^*\tensor G^{\oplus m}  \xrightarrow{\; \iota_{G^{\oplus m}}\;} \coend ]
\ee
or graphically 
\be\label{eq:P-iota-pic-GX} 
   \ipic{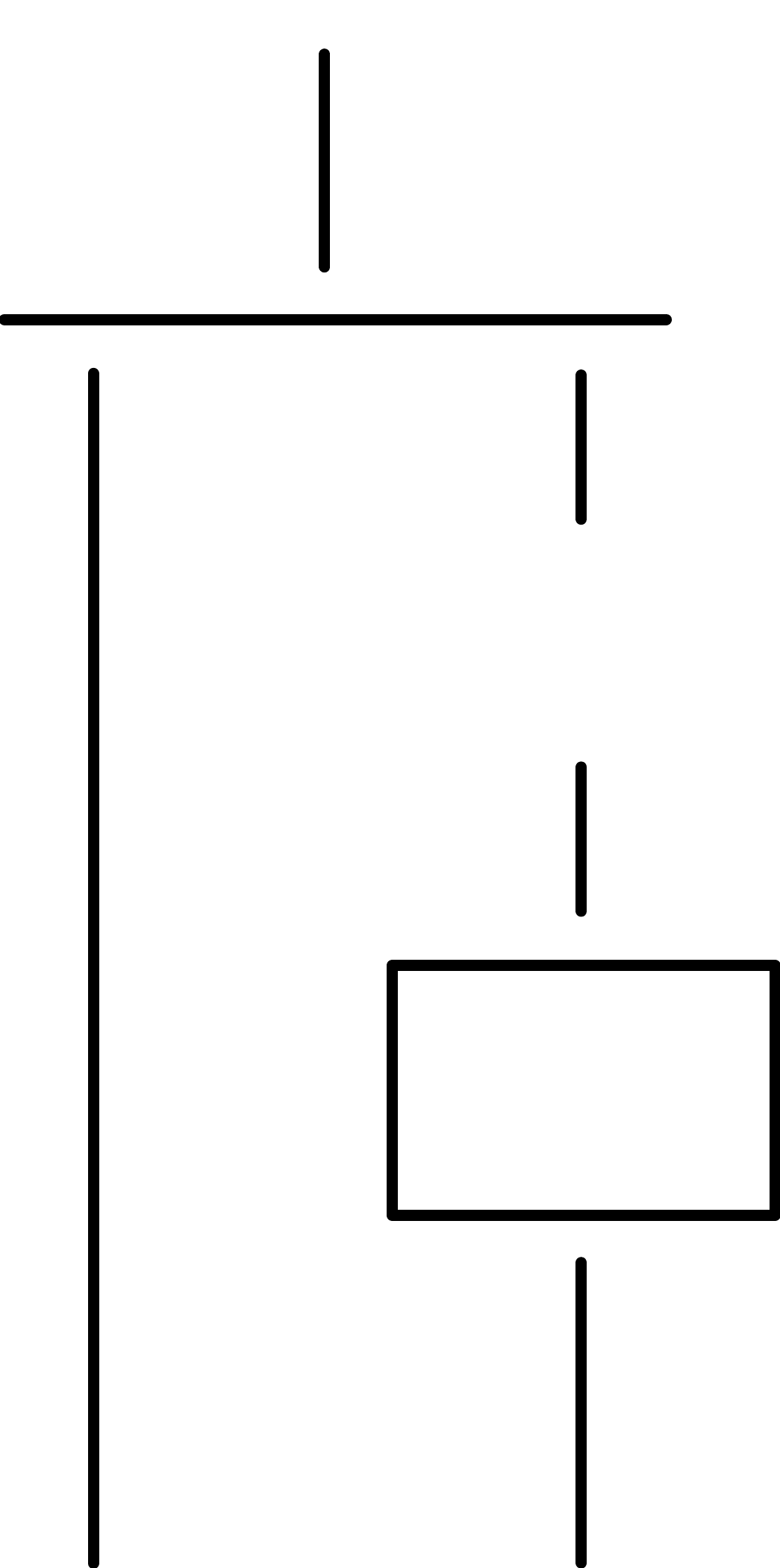}{.11}
	 \put(-34,55){\scriptsize$\coend$} 
	 \put(-24,40){\scriptsize$\iota_{X}$}
	 \put(-54,-64){\scriptsize$X^\ast$} 
	 \put(-19,7){\scriptsize$X$}
	 \put(-19,-24){\scriptsize$f_X$}
	 \put(-19,-64){\scriptsize$G^{\oplus m}$}
	 \quad = \quad
	 \ipic{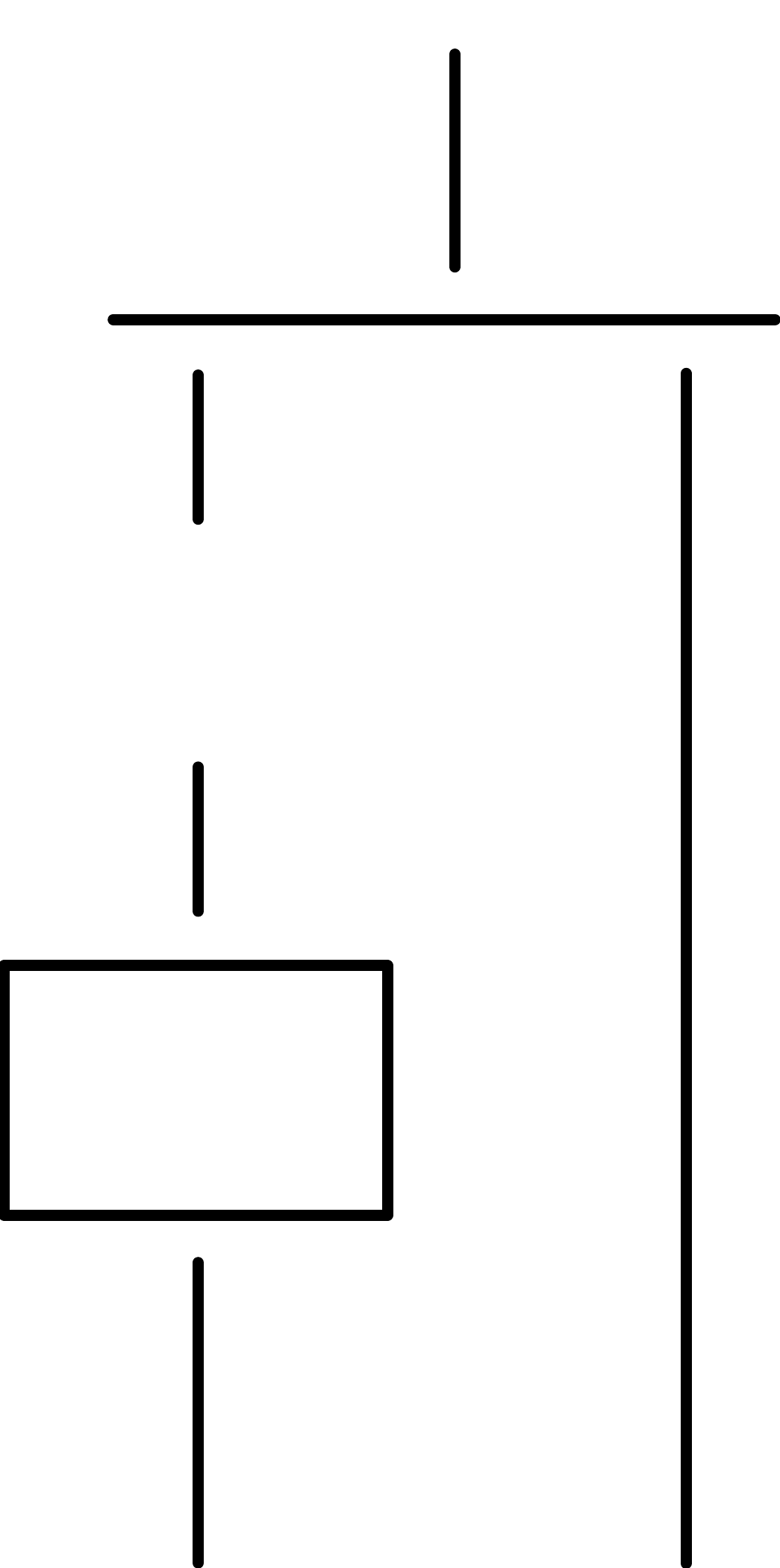}{.11}
	 \put(-26,55){\scriptsize$\coend$} 
	 \put(-17,40){\scriptsize$\iota_{G^{\oplus m}}$}
	 \put(-48,-64){\scriptsize$X^\ast$} 
	 \put(-51,7){\scriptsize$(G^{\oplus m})^*$}
	 \put(-46,-24){\scriptsize$f_X^*$}
	 \put(-13,-65){\scriptsize$G^{\oplus m}$}
\ee
where $\iota_{G^{\oplus m}}$ is defined using~\eqref{eq:P-iota} for $X=G$ and recall that $\iota_G=\pi$ is given to us. 
\end{remark}

We can describe $\coend$ as a quotient of an even 
smaller projective object. Namely, 
the surjective map $\pi$ in~\eqref{eq:L-P} factors through a `diagonal' product of the projective covers in $\cat$. Indeed, let  
\be\label{eq:IrrC-def}
\Irr(\cat)
\ee 
be a choice of representatives of the isomorphism classes of simple objects in $\cat$ and let us denote by $P_U$ 
a projective cover of~$U \in \Irr(\cat)$. 

\begin{proposition}
We have the 
	short
exact sequence
\begin{equation}\label{eq:L-P-2}
0\longrightarrow W \longrightarrow \bigoplus_{U\in\Irr(\cat)} P_U^\ast\otimes P_U \xrightarrow{\;\;\;\tilde\pi\;\;\;} \coend \longrightarrow 0 \,
\end{equation}
where the object $W$ (i.e.\ the kernel of $\tilde\pi$) is the union 
in $\bigoplus_{U} P_U^\ast\otimes P_U$ of
the images\footnote{
 That is, the smallest subobject of
$\bigoplus_{U} P_U^\ast\otimes P_U$ containing all these images.
}
\begin{equation}\label{eq:W-1}
 \im(f^* \tensor \id_{P_U}) 
 ~~,~~
 \im(\id_{P_V^*} \tensor  f) \ , \quad \text{for}\quad 
 U,V\in\Irr(\cat)~,~
 U\ne V~,~
 f: P_U\to P_V  \ ,
\end{equation}
and
\begin{equation}\label{eq:W-2}
\im (f^\ast\otimes\id_{P_U} - \id_{P_{U^*}}\otimes f)\  , \quad \text{for}\quad U\in\Irr(\cat)~,~f: P_U\to P_U\ .
\end{equation}
\end{proposition}

\begin{proof}
In the exact sequence~\eqref{eq:L-P}, we can  choose $G$ as the minimal projective generator $G=\oplus_{U\in\Irr(\cat)} P_U$  in $\cat$.
Denote the (primitive) idempotents in $\End_\cat(G)$ by 
$e_U: G\to P_U \to G$, 
for $U\in\Irr(\cat)$. The image of the map $e_U^\ast\otimes\id - \id\otimes e_U \in \End_\cat(G^*\tensor G)$ equals $\bigoplus_{V\ne U}P_U^\ast\otimes P_V \oplus P_V^\ast\otimes P_U$. Therefore, the kernel $K$ of $\pi$ contains all $P_U^\ast\otimes P_V$ such that $V \ncong U$, 
and the map $\pi$ factors through a map $\tilde\pi$ from the diagonal part of $G^*\tensor G$ to $\coend$: 
\begin{equation}\label{eq:L-P-2b}
\pi: \quad G^\ast \otimes G \longrightarrow \bigoplus_{U\in\Irr(\cat)} P_U^\ast\otimes P_U  \xrightarrow{\;\;\;\tilde\pi\;\;\;} \coend\ .
\end{equation}
We compute the kernel $W$ of $\tilde\pi$ using the rest of the
basis elements in $\End_\cat(G)$, those in the radical, and it gives the
span of~\eqref{eq:W-1} and~\eqref{eq:W-2}.
\end{proof}

\begin{remark}
If $\cat$ is semisimple, then by 
using~\eqref{eq:L-P-2} 
with $P_U=U$ we see that the coend $\coend$ is the direct sum over
 (isomorphism classes of)
 simple objects:
\begin{equation}\label{eq:coend-ss}
\coend = \bigoplus_{U\in\Irr(\cat)} U^\ast\tensor U \ .
\end{equation}
The corresponding family of dinatural transformations can be easily described using  the corresponding embeddings $i_{U}: U^*\otimes U\to \coend$, for details see~\cite[Lem.\,2]{Kerler:1996}.
\end{remark}

\subsection{Factorisability of a braided finite tensor category}\label{sec:factorisable-cat}

Let $\ok$ be a field and let $\cat$ be a $\ok$-linear braided finite tensor category.

By Proposition \ref{prop:univHopf-exists}, the universal Hopf algebra $\coend$ exists in $\cat$, and by Theorem \ref{thm:coend-Hopf+pairing} it is equipped with a Hopf pairing $\omega_\coend : \coend \otimes \coend \to \one$. Recall from Definition \ref{def:non-deg-Hopf-pair} that $\omega_\coend$ is called non-degenerate if the Hopf algebra homomorphism $\mathcal{D}_\coend : \coend \to \coend^*$ in \eqref{eq:Omega} is an isomorphism. 
Note that the kernel of $\mathcal{D}_\coend$ is the left annihilator of $\omega_\coend$.

\begin{definition}\label{def:fact-cat}
$\cat$ is called {\em factorisable} if the Hopf pairing $\omega_\coend$ is non-degenerate.\footnote{This definition is due to~\cite{Lyubashenko:1995, Kerler:2001} where the name `modular' is used. 
The usual definition of ``modular tensor category'' implies semisimplicity. But in view of fact that the qualifier ``modular'' is motivated by the projective action of the modular group, it would equally make good sense to speak of ``modular fusion category'' and of ``modular finite tensor category''. 
However, to avoid confusion we stick to the term ``factorisable'' in this paper, which is motivated from the application to Hopf algebras and quasi-Hopf algebras (see Section \ref{sec:ribbon-qHopf}).
}
\end{definition}

	In particular, 
in a factorisable finite tensor category the universal Hopf algebra $\coend$ is self-dual as a Hopf algebra.

Next we recall three more natural non-degeneracy conditions on the braiding of $\cat$ and then quote a theorem from \cite{Shimizu:2016} which states that for $\ok$ algebraically closed, they are all equivalent
to Definition~\ref{def:fact-cat}. 
A fourth equivalent
condition will be given later (Proposition \ref{prop:C-DD}).
We will need the $\ok$-linear map
\be \label{eq:Omega-def}
\Omega : \Cc(\one,\coend) \longrightarrow \Cc(\coend,\one) \quad , \quad a \mapsto \omega_\coend \circ (a \tensor \id) \circ \lambda_\coend^{-1}  \ .
\ee
The three conditions are:
\begin{enumerate}\setlength{\leftskip}{-1.5em}
\item The linear map $\Omega$ from~\eqref{eq:Omega-def} is an isomorphism.
\item Every transparent object in $\Cc$ is isomorphic to a direct sum of tensor units. ($T \in\Cc$ is {\em transparent} if for all $X \in \Cc$, $c_{X,T} \circ c_{T,X} = \id_{T\ot X}$.)
\item The canonical braided monoidal functor 
$\Cc \boxtimes \overline{\Cc} \to \Zc(\Cc)$ 
is an equivalence. (Here, $\boxtimes$ is the Deligne product, $\overline{\Cc}$ is the same tensor category as $\Cc$, but has inverse braiding, 
and $\Zc(\Cc)$ is the Drinfeld centre of $\Cc$.)
\end{enumerate}

\begin{theorem}[\cite{Shimizu:2016}]\label{thm:factequiv}
A braided finite tensor category over an algebraically closed field is factorisable if and only if any one of the conditions (1)--(3) is satisfied.
\end{theorem}

\begin{remark}
In the case that $\cat$ is semisimple, this equivalence was already known from \cite{Bruguieres:2000,Muger2001b}. $\cat$ is then a modular tensor category (minus the ribbon structure), and condition (1) above encodes the non-degeneracy of the $S$-matrix whose entries given by the 
	quantum traces of the
monodromy of pairs of simple objects (see \cite[Sec.\,5.1]{Shimizu:2016} for details).
\end{remark}

The following result is instrumental in the construction of the projective $SL(2,\oZ)$-action below.

\begin{proposition}[{\cite{Lyubashenko:1995} and \cite[Sect.\,5.2.3]{Kerler:2001}}]\label{prop:integrals-exist}
If $\cat$ factorisable, the coend $\coend$ has a two-sided (that is, a simultaneous left and right) integral
$\Lambda_\coend : \tid \to \coend$
satisfying
\be\label{eq:integral-normalisation}
\omega_\coend \circ (\Lambda_\coend \otimes \Lambda_\coend) \circ \lambda_{\one}^{-1}
~=~ k \, \id_{\tid}
\ee 
for some $k \in \ok^\times$. If $\Lambda_\coend'$ is another such integral, then $\Lambda_\coend' = r\, \Lambda_\coend$ for some $r \in \ok^\times$.
\end{proposition}

If $\ok$ has square roots we can normalise $\Lambda_\coend$
in \eqref{eq:integral-normalisation} such that $k=1$. In this normalisation, $\Lambda_\coend$ is unique up to a sign.

\newcommand{\DD}{\mathbb{D}_{\coend,\Gamma}}
\newcommand{\DDL}{\mathbb{D}_{\coend,\coend^*}}
\subsection{Factorisability as an isomorphism between end and coend}\label{sec:fact-end-coend}
Let $\cat$ be a finite 
 braided tensor category over a field $\ok$. 
Recall our
	functor 
$F=(-)^*\otimes (-)\colon \cat^{\mathrm{op}}\times \cat\to \cat$ 
	from \eqref{eq:F}
and that the coend $\coend$ represents the functor 
$\dinf= \Din(F, -)\colon \cat \to \vect_\ok$. 
The coend $\coend$
	can be thought of as the dual notion
to the \textit{end} of the functor 
\be
G:=(-)\otimes (-)^*:\quad  \cat\times \cat^{\mathrm{op}}\to \cat
\ee 
 (note the change of the order in the tensor product).
The end of $G$ is an object in $\cat$ representing the functor 
 $\Din(-, G)\colon \cat \to \vect_\ok$,
	see
 Definition~\ref{def:dinat-const} 
 (2). We will denote such an object (if it exists) as $\Gamma$ together with its family of 
 dinatural transformations $j_X\colon \Gamma \to X\otimes X^*$.  The dinaturality condition is now (compare with~\eqref{eq:dinat-iota-f})
\be\label{eq:dinat-j-f}
  (\id_{Y} \ot f^*) \circ j_Y =  (f \ot \id_{X^*}) \circ  j_X \qquad 
 \text{for all}
   \quad f\colon X\to Y\ .
\ee

An end $\Gamma$ exists by similar arguments as for coends in Section~\ref{sec:univ-Hopf-exists}. The existence  also follows from Lemma~\ref{lem:end-coend} given below.
	As was the case for coends, ends are
unique up to a unique isomorphism.

Consider a family of maps
\be
T_{X,Y}: \quad X^*\otimes X \to Y\otimes Y^*
\qquad , \quad X,Y \in \cat \ .
\ee
Suppose that for fixed 
$Y$ the family $(t_X)_{X \in \cat}$ with $t_X := T_{X,Y}$ is a dinatural transformation from $F$ to $Y \otimes Y^*$, and that for fixed
 $X$  the family 
$(t'_Y)_{Y\in\cat}$ with $t'_Y := T_{X,Y}$ is 
a dinatural transformation  
from $X^* \otimes X$ to~$G$
(see Definition~\ref{def:dinat-const}).
 Using the universal properties of ends and coends, one easily checks that there exists a unique  $D\colon \coend\to \Gamma$ such that $T_{X,Y}$ factors as 
$T_{X,Y} = j_Y\circ D\circ\iota_X$ for all $X,Y \in \cat$.

\newcommand{\T}{\mathbb{T}}
Consider now a special (``Hopf tangle'') dinatural transformation $\T_{X,Y}$  defined explicitly as
\begin{align}\label{eq:T-tangle}
\T_{X,Y} ~:=~ & \Bigl[ X^* X 
	\xrightarrow{~\sim~}  
X^* (X \one)  \xrightarrow{\id\tensor \coev_Y} X^* (X (Y Y^*)) 
 \xrightarrow{\id\tensor \assoc_{X,Y,Y^*}}  X^* ((X Y) Y^*)   \\\nonumber
& \hspace{3em} \xrightarrow{\id\tensor (\brC_{Y,X}\circ\brC_{X,Y})\tensor \id}  X^* ((X Y) Y^*) \xrightarrow{\id \tensor \assoc_{X,Y,Y^*}^{-1}} X^* (X (Y Y^*))  \\ \nonumber
& \hspace{3em}  \xrightarrow{\assoc_{X^*,X,YY^*}}  (X^* X) (Y Y^*)   \xrightarrow{\ev_X\tensor \id} \one(Y Y^*) 
\xrightarrow{~\sim~} Y Y^* \Bigr] \ .
\end{align}
The corresponding diagram is
\be\label{T-pic}
   \T_{X,Y} ~=\hspace*{1em}
\ipic{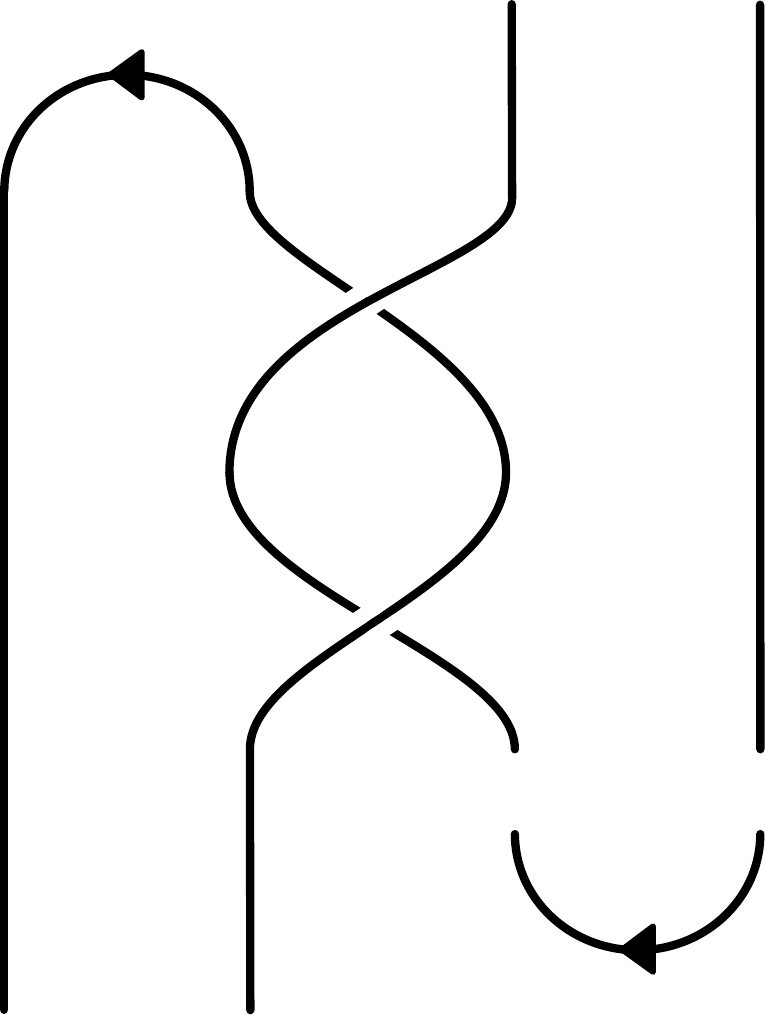}{.25}
	 \put(-96,-70){\scriptsize$X^\ast$} \put(-67,-70){\scriptsize$X$} \put(-34,-37.5){\scriptsize$Y$} \put(-3,-37.5){\scriptsize$Y^\ast$}
	 \put(-34,63){\scriptsize$Y$} \put(-3,63){\scriptsize$Y^\ast$}
	 \qquad .
 \ee
	By the above discussion, this map factors through a  
unique map $\DD: \coend\to \Gamma$ such that
\be\label{eq:T-DD}
\T_{X,Y}  =  j_Y\circ \DD \circ\iota_X \ .
\ee
	We call $\DD$ the {\em Drinfeld map} for the category $\Cc$. The reason for this is two-fold. Firstly, in case $\Cc$ is the category of finite-dimensional representations of a finite-dimensional quasi-triangular Hopf algebra, and for an appropriate choice of end and coend, $\DD$ is precisely the Drinfeld map, see Remark \ref{rem:DD-is-Drinfeld-for-Hopf} below. Secondly, 
	invertibility of the map $\DD$ provides another equivalent formulation of  factorisability of $\cat$:

\begin{proposition}\label{prop:C-DD}
$\cat$  is factorisable iff $\DD$ is invertible.
\end{proposition}

The proof requires the following lemma,
which makes use of the canonical isomorphism $X \to (\ldX)^*$ in a rigid category, which is given by
\begin{align}
\label{eq:dX-rigid-def}
d_X ~=~ \big[\, & X \xrightarrow{~\sim~} X\tid \xrightarrow{\id\ot\coev_{\ldX}} X(\ldX (\ldX)^*) 
\\ \nonumber &
\xrightarrow{~\sim~} (X\,\ldX) (\ldX)^* 
        \xrightarrow{\widetilde\ev_{X}\ot\id} \tid (\ldX)^*  \xrightarrow{~\sim~} (\ldX)^* \,\big] \ .
\end{align}

\begin{lemma}\label{lem:end-coend}
The pair 
$(\coend, \iota)$ is a coend of $F=(-)^*\otimes (-)$ iff the pair $(\coend^*, \hat{\iota})$ with 
\be\label{eq:hat-iota}
   \ipic{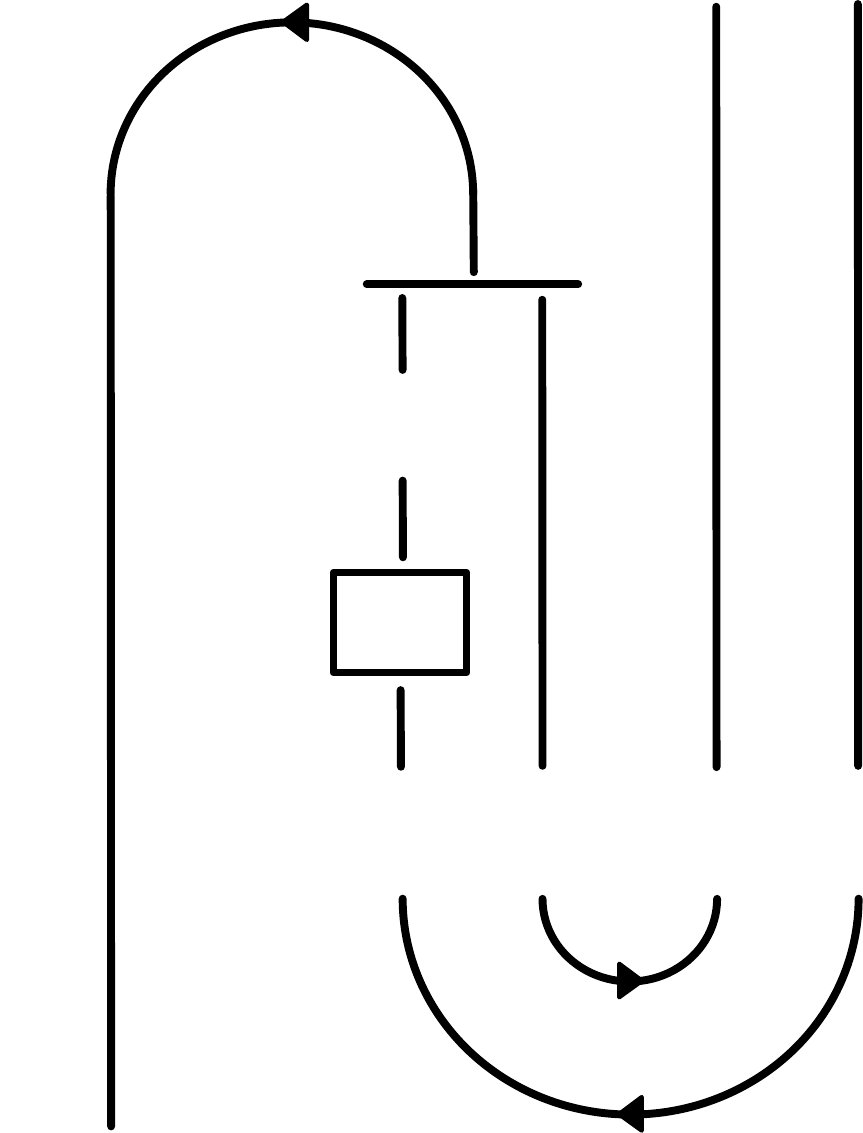}{.25}
    \put(-134,0){$\hat{\iota}_X~~~=$}
	 \put(-95,-77){\scriptsize$\coend^*$}  \put(-75,30){\scriptsize$\iota_{\ldX}$} 
	 \put(-22,-35){\scriptsize$X$} \put(-3,-35){\scriptsize$X^{\ast}$} \put(-22,70){\scriptsize$X$} \put(-3,70){\scriptsize$X^{\ast}$}
	 \put(-61,-35){\scriptsize$X$}  \put(-43,-35){\scriptsize$\ldX$}  \put(-61,-10){\scriptsize$d_X$} \put(-67,14){\scriptsize$(\ldX)^*$} 
\ee
 is an end
	for
the  functor $G= (-)\otimes(-)^*$.
\end{lemma}
\begin{proof}
The dinaturality condition~\eqref{eq:dinat-j-f} on $\hat{\iota}_X$ is 
	easily verified from the diagram \eqref{eq:hat-iota} using
the dinaturality of $\iota$ and naturality of $d_X$.
The universal property of $(\coend^*,\hat{\iota})$ is proven using the universal property of $(\coend,\iota)$.
\end{proof}

\begin{proof}[Proof of Proposition~\ref{prop:C-DD}]
By definition, $\cat$ is factorisable if 
the map $\mathcal{D}_{\coend}$ defined in~\eqref{eq:Omega} is invertible (Definitions~\ref{def:fact-cat} and~\ref{def:non-deg-Hopf-pair}).

Consider the end $(\coend^*, \hat{\iota})$ from Lemma \ref{lem:end-coend}. We will start by showing that $\DDL=\mathcal{D}_{\coend}$. By the unique factorisation property explained above, it is enough to show that for all $X,Y \in \cat$ we have $\hat\iota_Y\circ \DDL \circ \iota_X = \hat\iota_Y\circ \mathcal{D}_{\coend} \circ \iota_X$. Using \eqref{eq:T-DD} we see that we need to show 
\be\label{eq:C-DD-aux1}
\hat\iota_Y\circ \mathcal{D}_{\coend} \circ \iota_X = \T_{X,Y}
\ee
with $\T_{X,Y}$ as in~\eqref{eq:T-tangle}.
Substituting \eqref{eq:Omega}, \eqref{hopf-pair}, \eqref{eq:hat-iota} and \eqref{eq:dX-rigid-def} into the left hand side of \eqref{eq:C-DD-aux1} and using the zig-zag identity for the duality maps on $\coend$ once gives
\be
   \ipic{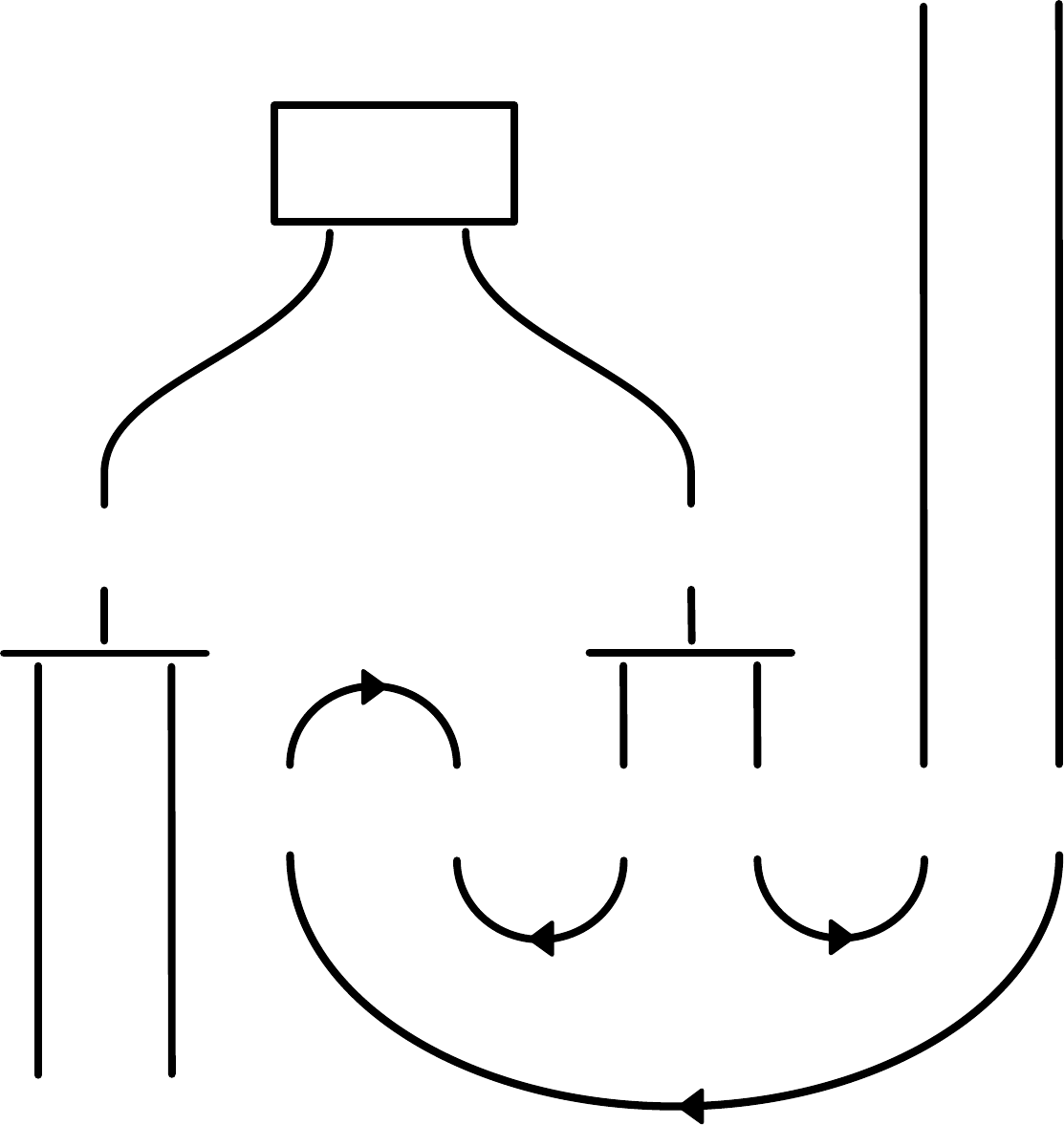}{.25}
	 \put(-135,-74){\scriptsize$X^\ast$} \put(-116,-74){\scriptsize$X$} \put(-124,-2){\scriptsize$\coend$}  \put(-138,-7){\scriptsize$\iota_{X}$}
	 \put(-38,-7){\scriptsize$\iota_{{}^*Y}$} \put(-100,-35){\scriptsize$Y$} \put(-83,-35){\scriptsize${}^*Y$}
	 \put(-69,-35){\scriptsize$({}^*Y)^{\ast}$} \put(-42,-35){\scriptsize${}^*Y$} \put(-21,-35){\scriptsize$Y$} \put(-3,-35){\scriptsize$Y^{\ast}$}
	 \put(-50,-2){\scriptsize$\coend$} \put(-91,47){$\omega_\coend$}
	 \put(-20,73){\scriptsize$Y$} \put(-3,73){\scriptsize$Y^{\ast}$}
	 \qquad \overset{(*)}= \qquad
	 \ipic{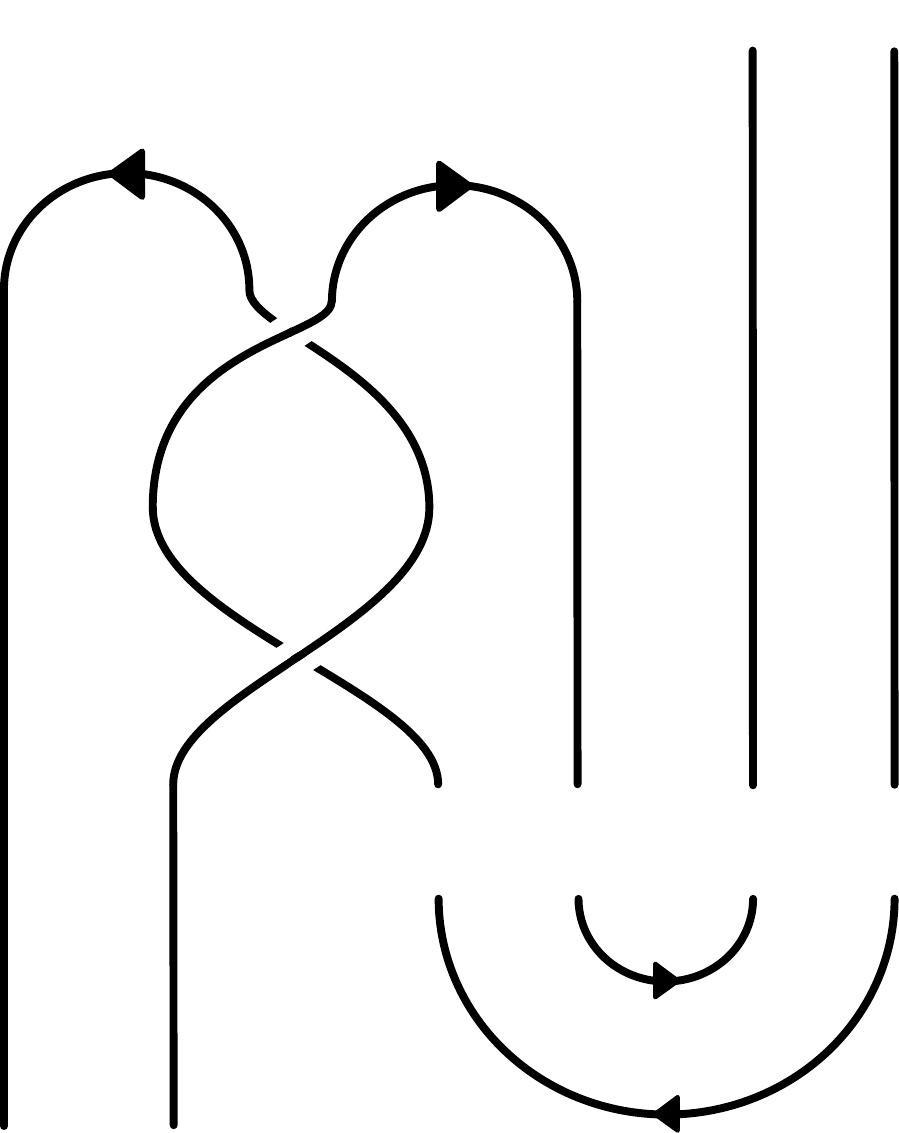}{.25}
	 \put(-113,-76){\scriptsize$X^\ast$} \put(-91,-76){\scriptsize$X$}
	 \put(-61,-36){\scriptsize$Y$} \put(-43,-36){\scriptsize${}^*Y$} \put(-22,-36){\scriptsize$Y$} \put(-5,-36){\scriptsize${}^*Y$}
	 \put(-20,66){\scriptsize$Y$} \put(-3,66){\scriptsize$Y^{\ast}$}
	 \quad ,
\ee
	where in (*) the zig-zag identity for duality morphisms was used once again.
Using the zig-zag identity once more,
we find the right hand side of \eqref{eq:C-DD-aux1}.

For an end $(\Gamma,j)$,   there is a unique isomorphism $\phi: (\coend^*,\hat\iota) \xrightarrow{\;\sim\;} (\Gamma, j)$ such that the diagram
\be
\xymatrix
{
&X\tensor X^*&\\
\coend^{*}\ar[rr]^{\phi}\ar[ur]^{\hat\iota_X}&& \Gamma\ar[ul]_{j_X}
}
\ee
commutes. Then we have $\DD=\phi\circ \DDL$ because by definition  we have  
\be
 j_Y\circ \DD \circ \iota_X =\T_{X,Y} =  \hat{\iota}_Y \circ \DDL\circ\iota_X =  j_Y\circ \phi\circ \DDL\circ\iota_X \ .
 \ee
  Therefore, $\DD$ is invertible iff $\mathcal{D}_\coend$ is and so iff the Hopf pairing $\omega_\coend$ is non-degenerate.
\end{proof}

\section{$SL(2,\oZ)$-action for factorisable finite tensor categories} 
\label{sec:SL2Z}

For factorisable 
finite ribbon
categories $\cat$ with universal Hopf algebra $\coend$, 
one can define a projective $SL(2,\oZ)$-action on $\cat(\one,\coend)$ (Section \ref{sec:SL2Z-action}), and
on 
the $\ok$-vector space $\End(\id_\Cc)$.
	The second action is independent of the choice of $\coend$
 (Proposition~\ref{prop:ST-on-EndId}).
	In Section \ref{sec:intchar-natendo} we discuss different ways to transport internal characters from $\cat(\one,\coend)$ to $\End(\id_\Cc)$.

\subsection{Projective $SL(2,\oZ)$-action}\label{sec:SL2Z-action}

Let $\ok$ be a
field and let $\cat$ be a factorisable finite 
tensor category over $\ok$.
We assume that \eqref{eq:integral-normalisation} has a solution with $k=1$ (e.g.\ if $\ok$ is algebraically closed).
We furthermore assume that $\cat$ is ribbon, and we will denote the ribbon twist on $V \in \cat$ by $\theta_V\colon V \to V$.

In this section we review from \cite{Lyubashenko:1995} the projective $SL(2,\oZ)$-action on the Hom-space 
$\cat(1,\coend)$ and on the vector space $\End(\id_\cat)$, the natural endomorphisms of the identity functor.

We start by introducing the monodromy morphism $\mathcal Q \colon \coend\tensor\coend\to\coend\tensor\coend$ as
\begin{align}\label{eq:Q-map}
\mathcal Q \circ (\iota_U \tensor \iota_V) ~=~ \Big[ 
& (U^\ast  U)  (V^\ast  V) \xrightarrow{\assoc^{-1}_{U^\ast,U,V^\ast  V}} U^\ast  (U  (V^\ast  V)) 
\xrightarrow{\id\tensor\assoc_{U,V^\ast,V}}  
U^\ast  ((U  V^\ast)  V) 
\\  \nonumber & 
\xrightarrow{\id\tensor (\brC_{V^\ast,U}\circ\brC_{U, V^\ast}) \tensor \id } U^\ast  ((U  V^\ast)  V)
\xrightarrow{\id\tensor\assoc^{-1}_{U,V^\ast,V}} 
U^\ast  (U  (V^\ast  V)) 
\\ \nonumber & 
\xrightarrow{\assoc_{U^\ast,U,V^\ast  V}} (U^\ast  U)  (V^\ast  V) \xrightarrow{\iota_U\tensor \iota_V} \coend\tensor \coend
\Big] \ , 
\end{align} 
or, in string diagram notation, 
\be\label{coend-Q}
   \ipic{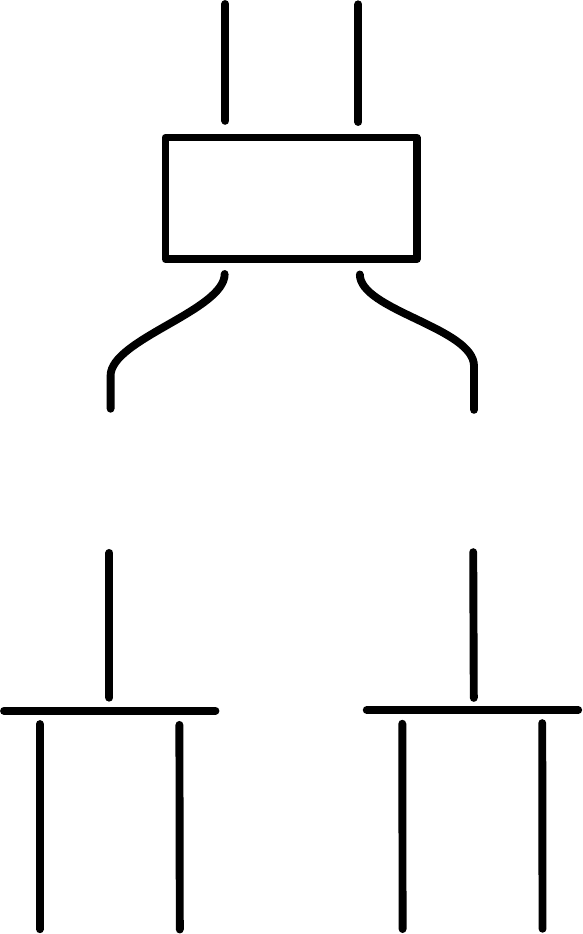}{.23}
	 \put(-42,55){\scriptsize$\coend$} \put(-28,55){\scriptsize$\coend$} 
	 \put(-56,-5){\scriptsize$\coend$} \put(-16,-5){\scriptsize$\coend$} 
   \put(-66,-22){\scriptsize$\iota_U$} \put(-7,-22){\scriptsize$\iota_V$}
	 \put(-65,-60){\scriptsize$U^\ast$} \put(-47,-60){\scriptsize$U$} \put(-24,-60){\scriptsize$V^\ast$} \put(-8,-60){\scriptsize$V$}
	 \put(-36,27){\footnotesize$\mathcal Q$} 
	 \quad = \quad
	 \ipic{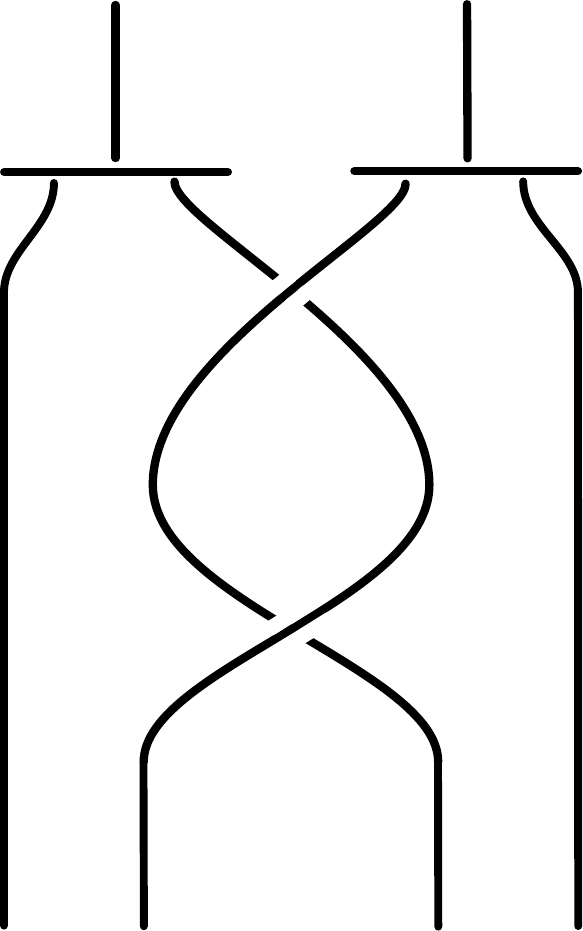}{.23}
	 \put(-55,55){\scriptsize$\coend$} \put(-17,55){\scriptsize$\coend$} 
	 \put(-66,39){\scriptsize$\iota_U$} \put(-7,39){\scriptsize$\iota_V$}
	 \put(-68,-60){\scriptsize$U^\ast$} \put(-51,-60){\scriptsize$U$} \put(-22,-60){\scriptsize$V^\ast$} \put(-4,-60){\scriptsize$V$}
\ee
The Hopf pairing in \eqref{eq:omega-L} is related to $\mathcal Q$ as
\begin{align} \label{def:Hopf-pair}
	\omega_\coend~=~ \big[\, \coend\coend \xrightarrow{\mathcal Q}
	\coend\coend \xrightarrow{\eps_\coend\tensor\eps_\coend}
	\one\one \xrightarrow{\sim} \one \,\big] \ .
\end{align}

Following \cite{Lyubashenko:1995}, we introduce two endomorphisms $\modS,\modT : \coend \to \coend$, which will then be used to define the action of the $S$- and $T$-generator of $SL(2,\oZ)$. 
We take the integral given to us by Proposition \ref{prop:integrals-exist} to be normalised such that
\be
\omega_\coend \circ (\Lambda_\coend \otimes \Lambda_\coend) \circ \lambda_{\one}^{-1}
~=~ \id_{\tid} 
\ee 
(as is possible by our assumptions on $\ok$).
Recall that this fixes $\Lambda_\coend$ up to a sign.
We set
\be \label{eq:cat-ST}
\modS  = \lambda_\coend \circ (\eps_\coend\tensor\id) \circ \mathcal Q \circ (\id \tensor \Lambda_\coend)  \circ \runit^{-1}_\coend ~~,\quad
\modT \circ \iota_U  = \iota_U \circ (\id \otimes \theta_U) 
~~\text{ for all }~~U\in\cat \ .
\ee
In terms of string diagrams we have
\be\label{fig:eq-Tac}
   \modS\quad
	=
   \quad
	 \ipic{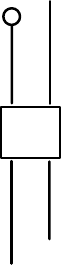}{.8}
	 \put(-23,-60){\scriptsize$\coend$} \put(-8,-50){\scriptsize$\Lambda_\coend$}
	 \put(-16,-2){\scriptsize$\mathcal Q$}	 \put(-8,54){\scriptsize$\coend$} 
	 \quad\qcq\quad
   \ipic{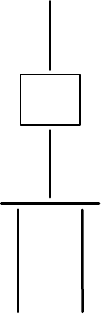}{.8}
	 \put(-22,64){\scriptsize$\coend$}
	 \put(-22,17){$\modT$}
	 \put(-10,-12){\scriptsize$\iota_U$}
	 \put(-37,-69){\scriptsize$U^\ast$} \put(-10,-69){\scriptsize$U$}
   \quad
	=
\quad
	 \ipic{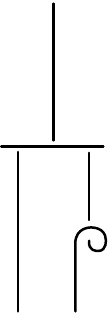}{.8}
	 \put(-22,64){\scriptsize$\coend$}
	 \put(-10,12){\scriptsize$\iota_U$}
	 \put(-37,-69){\scriptsize$U^\ast$} \put(-15,-69){\scriptsize$U$}
\ee

\begin{theorem}[{\cite{Lyubashenko:1995}}]\label{thm:MCG-1hole-torus}
The endomorphisms $\modS,\modT$ from \eqref{eq:cat-ST} satisfy
\begin{align} \label{eq:ST-endo-coend}
(\modS\modT)^3=\lambda \, \modS^2 \ , \qquad \modS^2 = S_\coend^{-1} \ , \qquad \ ,
\end{align} 
for some constant $\lambda\in\ok^\ast$.
\end{theorem}

It is not hard to verify directly  from the definition in \eqref{eq:antipode-L} that the antipode of $\coend$ squares to the ribbon twist,
\be\label{eq:S2-ribbontwist}
 S_\coend \circ S_\coend~=~\twist_\coend \ .
\ee
Thus, the relations in Theorem \ref{thm:MCG-1hole-torus} also imply $\modS^4 =  \theta_\coend^{-1}$.

By the $S$- and $T$-generators of $SL(2,\oZ)$ we mean the $2{\times}2$ matrices $\mathbf{S} = \big( \begin{smallmatrix} 0 & -1 \\ 1 & 0 \end{smallmatrix} \big)$ 
and
$\mathbf{T} = \big( \begin{smallmatrix} 1 & 1 \\ 0 & 1 \end{smallmatrix} \big)$,
respectively. One can describe $SL(2,\oZ)$ as the group freely generated by $\mathbf{S}$ and $\mathbf{T}$ subject to the relations
\be\label{eq:SL2Z-gen-and-rel}
	(\mathbf{S}\mathbf{T})^3 = \mathbf{S}^2
	\quad , \quad
	\mathbf{S}^4 = \id \ .
\ee
Therefore,
as an immediate consequence of Theorem \ref{thm:MCG-1hole-torus} we have:

\begin{corollary}\label{cor:SL2Z-on-C1L}
The $\ok$-vector space $\cat(\one,\coend)$ carries a projective action of $SL(2,\oZ)$ where the action of $\mathbf{S}$ and $\mathbf{T}$ is given by, for $f \in \cat(\one,\coend)$,
\be\label{eq:SL2Z-on-C1L}
	\mathbf{S}.f := \modS \circ f
	\quad , \quad
	\mathbf{T}.f := \modT \circ f \ .
\ee
\end{corollary}

\begin{proof}
The first relation in \eqref{eq:SL2Z-gen-and-rel} is just the first relation in \eqref{eq:ST-endo-coend} (up to the projectivity factor). The second relation in \eqref{eq:SL2Z-gen-and-rel} follows from \eqref{eq:ST-endo-coend} and \eqref{eq:S2-ribbontwist}, together with naturality of the ribbon twist:
$\mathbf{S}^4.f = \modS^4 \circ f = \theta_\coend^{-1} \circ f
= f \circ \theta_{\one}^{-1} = f$.
\end{proof}

The universal Hopf algebra $\coend$ is only unique up to unique isomorphism, and in the above projective representation of $SL(2,\oZ)$ already the underlying vector space depends on $\coend$. It can be helpful to have a variant of the action which is manifestly independent of the choice of $\coend$
(up to the choice of the sign of the integral). 
This can be achieved by transporting the action to $\End(\id_\cat)$, as we explain next (see \cite{Lyubashenko:1995,Shimizu:2015}).

We start by defining two $\ok$-linear isomorphisms
\be\label{eq:RadPsi-def}
\Rad : \cat(\coend,\one) \to \cat(\one,\coend)
\qquad,\qquad
\psi : \End(\id_\Cc) \to \cat(\coend,\one) \ .
\ee
Their values on $f \in \cat(\coend,\one)$ and $\alpha \in \End(\id_\Cc)$ are determined by
\begin{align}\label{eq:rhopsi-def}
\Rad(f) &~=~ \big[\,
\one \xrightarrow{\Lambda_\coend}
\coend \xrightarrow{\Delta_\coend}
\coend\coend \xrightarrow{f \otimes \id} \one \coend 
\xrightarrow{\sim} \coend \,\big] \ ,
\\ \nonumber
\psi(\alpha) \circ \iota_X 
&~=~ \big[\,
X^*X \xrightarrow{\id\tensor \alpha_X} X^*X \xrightarrow{\ev_X} \one \,\big]
 \quad \text{for all } X\in\Cc \ .
\end{align}
The inverses of $\Rad$ and $\psi$ can be given explicitly. For $a \in \cat(\one,\coend)$ and $f \in \cat(\coend,\one)$ we have
\begin{align}\label{eq:rhopsi-inv}
\Rad^{-1}(a) &~=~ \big[\, \coend \xrightarrow{\sim} \one \coend \xrightarrow{a \ot \id} \coend\coend
\xrightarrow{S_\coend \ot \id} \coend\coend \xrightarrow{\mu_\coend} \coend \xrightarrow{\coint_\coend} \one \,\big] \ ,
\\ \nonumber
\psi^{-1}(f)_X 
&~=~ \big[\, X \xrightarrow{\sim} \one X \xrightarrow{\coev_X\tensor\id} (XX^\ast)X \xrightarrow{\sim} X(X^\ast X) 
  \xrightarrow{\id\tensor\iota_X} X\coend \xrightarrow{\id\tensor f} X\one \xrightarrow{\sim} X \,  \big] \ ,
\end{align}
where the second line applies again for all $X \in \cat$.
In terms of the isomorphisms $\psi$ and $\Rad$, together with the isomorphism $\Omega$ from \eqref{eq:Omega-def} and the ribbon twist $\theta$, we define $\modS_\cat , \modT_\cat \in 
\End_\ok(\End(\id_\cat))$ as
\begin{align}\label{eq:SCTC}
\modS_\cat &~=~ \big[\, \End(\id_\Cc) \xrightarrow{\psi} \cat(\coend,\one)
 \xrightarrow{\Rad} \cat(\one,\coend) \xrightarrow{\Omega} 
 \cat(\coend,\one)
  \xrightarrow{\psi^{-1}} \End(\id_\Cc) \,\big] \ ,
\\ \nonumber 
\modT_\cat &~=~ \big[\, \End(\id_\Cc) \xrightarrow{\theta \circ (-)} \End(\id_\Cc) \,\big]  \ .
\end{align}
We collect the results reviewed in this section in the following proposition.

\begin{proposition}\label{prop:ST-on-EndId}
The $\ok$-vector space $\End(\id_\Cc)$ carries a projective action of $SL(2,\oZ)$ where the actions of $\mathbf{S}$ and $\mathbf{T}$ are given by, for $\alpha \in \End(\id_\Cc)$,
\be\label{eq:ST-on-EndId}
	\mathbf{S}.\alpha := \modS_\cat(\alpha)
	\quad , \quad
		\mathbf{T}.\alpha
	:= \modT_\cat(\alpha) \ .
\ee
 Moreover, $\modS_\cat$ and $\modT_\cat$ in \eqref{eq:SCTC} are independent of $(\coend, \intL)$ up to the choice of sign of $\intL$.
\end{proposition}

\begin{proof}
We first show the identity
\be\label{eq:ST-on-EndId_aux1}
\big[\,
\coend\coend
\xrightarrow{\id \otimes \Delta_\coend}
\coend(\coend\coend)
\xrightarrow{\sim}
(\coend\coend)\coend
\xrightarrow{\omega_\coend\tensor\id}
\one\coend
\xrightarrow{\sim}
\coend
\,\big]
~=~
\big[\,
\coend\coend
\xrightarrow{\mathcal Q}
\coend\coend
\xrightarrow{\eps_\coend \otimes \id}
\one\coend
\xrightarrow{\sim}
\coend
\,\big] \ .
\ee
To establish this equality we verify that it holds when precomposed with $\iota_U \otimes \iota_V$ for all $U,V\in\cat$. Indeed, substituting the defining relations in \eqref{eq:cop-L}, \eqref{eq:omega-L}, \eqref{eq:Q-map} and \eqref{eq:eps-L} one finds that \eqref{eq:ST-on-EndId_aux1} is equivalent to the following identity, which we give in terms of string diagrams
\be
\put(-3,-63){\scriptsize$U^\ast$} \put(14,-63){\scriptsize$U$} \put(49,-63){\scriptsize$V^\ast$} \put(118,-63){\scriptsize$V$}
\put(108,45){\scriptsize$\coend$}
\put(187,-71){\scriptsize$U^\ast$} \put(204,-71){\scriptsize$U$} \put(239,-71){\scriptsize$V^\ast$} \put(255,-71){\scriptsize$V$}
\put(243,65){\scriptsize$\coend$}
\ipic{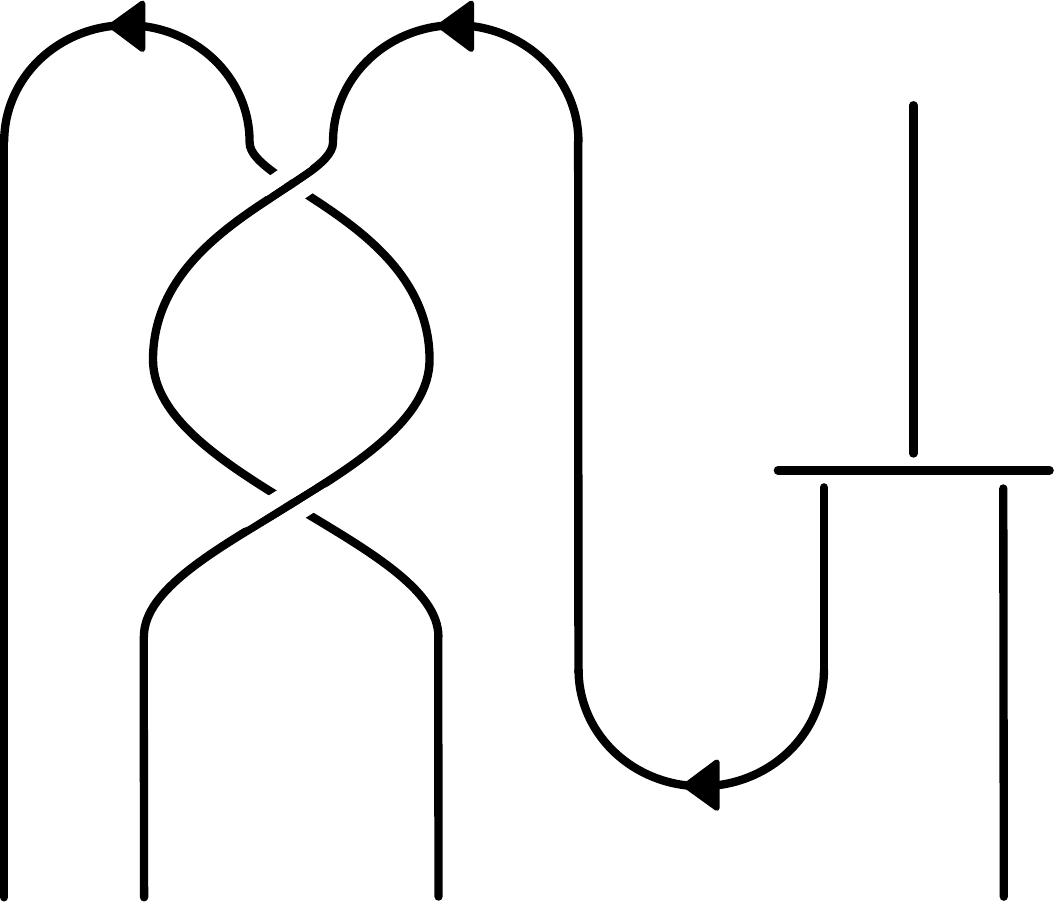}{.25} \qquad = \qquad \ipic{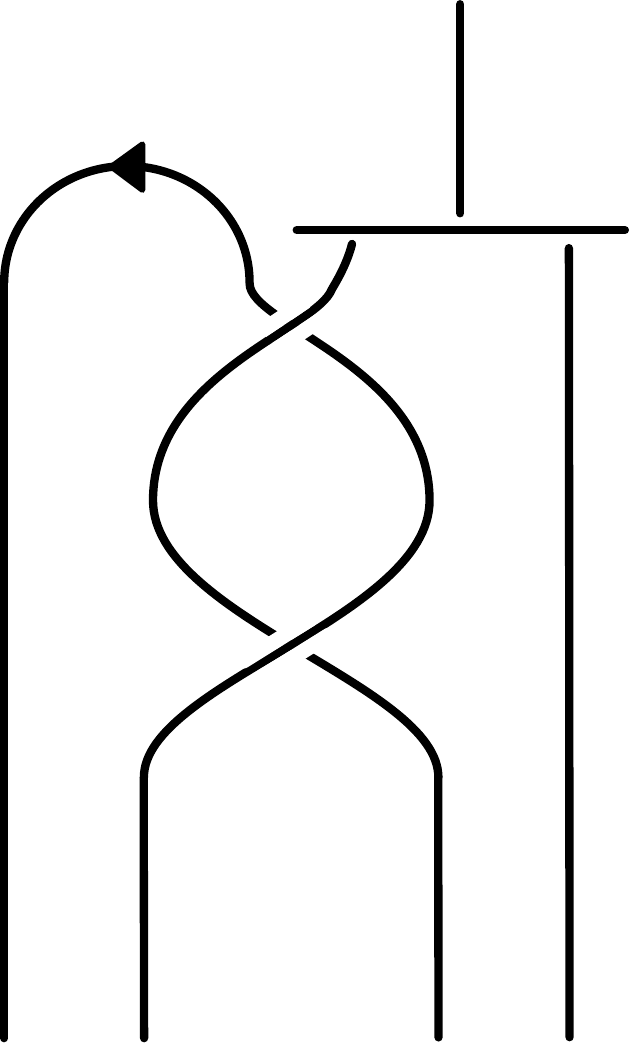}{.25}
\ee
This identity clearly holds by the zig-zag identity for duality morphisms. 
Thus also \eqref{eq:ST-on-EndId_aux1} holds. 

Precomposing \eqref{eq:ST-on-EndId_aux1} with $\id \otimes \Lambda_\coend$ and comparing to the definition of $\Omega$ and $\Rad$ in \eqref{eq:Omega-def} and \eqref{eq:rhopsi-def}, as well as to the definition of $\modS$ in \eqref{eq:cat-ST}, shows that for all $f \in \cat(\one,\coend)$,
\be\label{eq:Rad-Omega-S}
	\Rad(\Omega(f)) ~=~ \modS \circ f \ .
\ee
It is now straightforward to check that conjugating the action of $\mathbf{S}$ and $\mathbf{T}$ in \eqref{eq:SL2Z-on-C1L} with 
the $\ok$-linear isomorphism
$\Rad \circ \psi$ gives the action in \eqref{eq:ST-on-EndId}.

Next we show 
 the independence 
of $\modS_\cat$ and $\modT_\cat$ from the choice of universal Hopf algebra 
	$(\coend,\iota)$ and integral $\Lambda_{\coend}$ (up to sign).
For $\modT_\cat$ there is nothing to do. 
For $\modS_\cat$, let 
$(\coend',\iota')$ be another choice of coend
	and $\Lambda_{\coend'}$ a choice of (normalised) integral for $\coend'$. Let
$h : \coend \to \coend'$ be the unique isomorphism satisfying $h \circ \iota_X = \iota'_X$ for all $X \in \cat$. 
Then using the defining relations in Figure~\ref{fig:Hopf-coend}
for the Hopf algebra structure and the Hopf pairing~\eqref{hopf-pair} but written for the coend $\coend'$ and equating maps corresponding to the same dinatural transformations, we find the relations
\begin{align}
\mu_{\coend}  &= h^{-1} \circ \mu_{\coend'} \circ (h\otimes h)\ , 
& \Delta_{\coend} &= (h^{-1}\otimes h^{-1}) \circ \Delta_{\coend'} \circ h\ ,\\ \nonumber
\omega_{\coend} &= \omega_{\coend'} \circ (h\otimes h)\ ,
& \eps_\coend &= \eps_{\coend'}\circ h\ .
\end{align}
{}From
this and Proposition~\ref{prop:integrals-exist} we get the relation $\Lambda_{\coend'} = \pm \, h\circ\Lambda_\coend$ between the normalised integrals.
Denote by $\Omega'$, $\Rad'$, $\psi'$ the maps in \eqref{eq:Omega-def} and \eqref{eq:rhopsi-def}, but computed for $\coend'$. 
Using the relations between $\coend$ and $\coend'$ just stated,
one easily verifies that, for $x \in \cat(\one,\coend')$, $f \in \cat(\coend',\one)$, $\alpha \in \End(\id_\cat)$,
\begin{align}\label{eq:ST-on-EndId_aux2}
	\Omega'(x) &= \Omega(h^{-1} \circ x) \circ h^{-1} \ ,
	& \psi'(\alpha) &= \psi(\alpha) \circ h^{-1} \ ,
\\ \nonumber
	\Rad'(f) &= \pm \, h \circ \Rad(f\circ h) \ ,
	& (\psi')^{-1}(f) &= \psi^{-1}(f \circ h) \ .
\end{align}
Substituting this into the definition of $\modS_\cat$ in \eqref{eq:SCTC} one arrives at 
$\modS'_\cat = \pm\,\modS_\cat$.
\end{proof}

\subsection{Internal characters and corresponding natural endomorphisms} 
\label{sec:intchar-natendo}

Let $\ok$ be an algebraically closed field and let $\cat$ be a $\ok$-linear factorisable and pivotal finite tensor category.

For the comparison to the conformal field theory calculation of the $SL(2,\oZ)$-action from \cite{Gainutdinov:2016qhz} in the companion paper \cite{FGRprep} and for the explicit form of the Verlinde formula in Section \ref{sec:intchar-qHopf} below, we will need to recall the definition and some properties of internal characters.

The {\em internal character}	of $V \in \cat$ is the element $\chi_V \in \Cc(\one,\coend)$ given by \cite{Fuchs:2010mw,Shimizu:2015} 
\be\label{eq:chiV-def}
\chi_V ~=~ 
\big[\, \one \xrightarrow{\widetilde\coev_V} V^* \ot V   \xrightarrow{\iota_{V}} \coend \,\big] \ ,
\ee
where we use the convention in \cite{Gainutdinov:2016qhz}.

Denote by $\Gr(\cat)$ the Grothendieck ring of $\cat$, and 
	recall the definition of $\Irr(\cat)$ from \eqref{eq:IrrC-def}.
	As
$\cat$ is finite, $\Gr(\cat)$ is the free $\oZ$-linear span of $[U]$, $U \in \mathrm{Irr}(\cat)$. We will abbreviate
$\Gr_\ok(\cat) := \ok \otimes_\oZ \Gr(\cat)$ for the $\ok$-linearised 
Grothendieck ring. 
The following theorem was proved in \cite[Cor.\,4.2]{Shimizu:2015}
under more general assumptions (in particular for non-braided $\cat$).

\begin{theorem}\label{thm:chiinj}
The assignment $V \mapsto \chi_V$
induces a $\ok$-linear map $\chi : \Gr_\ok(\Cc) \to \Cc(\one,\coend)$. 
	The map $\chi$ is injective.
\end{theorem}

We remark that if in addition $\ok$ is of characteristic zero, then also the composition $\Gr(\cat) \to \Gr_\ok(\cat) \xrightarrow{\chi} \Cc(\one,\coend)$ is injective.

Next we use the maps $\Rad^{-1}$ and $\psi^{-1}$ from \eqref{eq:rhopsi-inv} to transport $\chi_V$ 
to $\End(\id_{\cat})$,
\be\label{eq:tildechi-def}
\phi_V ~:=~ \psi^{-1}(\Rad^{-1}(\chi_V)) \ .
\ee
{}From 
the proof of Proposition \ref{prop:ST-on-EndId} and from Theorem \ref{thm:chiinj} we conclude:

\begin{corollary}\label{cor:phiM-Gr-indepL}
The $\phi_V$ only depend on the class $[V]$ of $V$ in $\Gr(\cat)$
and on the choice of sign for the integral $\intL$, but are otherwise independent of the choice of the coend $\coend$.
The set $\{ \phi_U \,|\, U \in \Irr(\cat) \} \subset \End(\id_{\cat})$ is $\ok$-linearly independent.
\end{corollary}

\begin{proof}
Using the notation from the proof of Proposition \ref{prop:ST-on-EndId}, it remains to note that in addition to \eqref{eq:ST-on-EndId_aux2} we have 
$\chi'_V = h \circ \chi_V$ and $(\Rad')^{-1}(x) = \Rad^{-1}(h^{-1} \circ x) \circ h^{-1}$.
\end{proof}

After applying $\modS_\cat$ to $\phi_V$, the expression simplifies to a ``Hopf link operator'' as considered in \cite{Creutzig:2016fms}, see \cite[Rem.\,3.10\,(2)]{Gainutdinov:2016qhz},
\be\label{eq:hopf-link-op}
	\modS_\cat(\phi_V)_X ~=~ \ipic{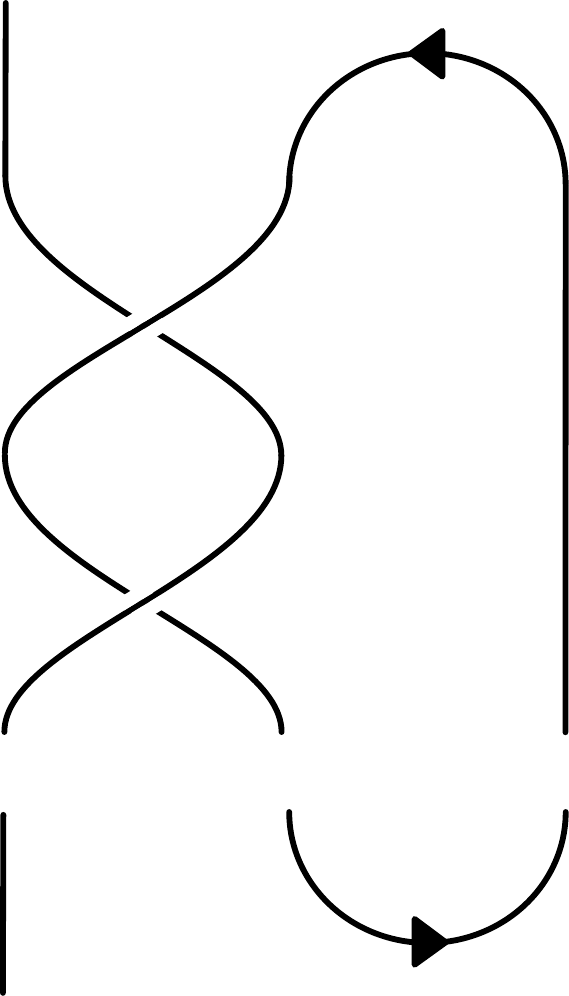}{.3} \ .
	\put(-12,-43){\scriptsize $V$} \put(-94,-43){\scriptsize $X$} \put(-94,75){\scriptsize $X$} \put(-54,-43){\scriptsize $V^\ast$}
\ee
In terms of formulas, this reads
\begin{align}\label{eq:S(phi)-generalcat}
	\modS_\cat(\phi_V)_X 
&~=~
\big[
 X 
\xrightarrow{\runit_X^{-1}}  X\one
\xrightarrow{\id \ot \widetilde\coev_{V} } X (V^{*} V)
\xrightarrow{\alpha_{X,V^*,V}} (XV^{*}) V
\\ \nonumber 
& \hspace{2em} 
\xrightarrow{ (c_{V^*,X} \circ c_{X,V^*})\ot \id} (XV^{*}) V
\xrightarrow{\alpha^{-1}_{X,V^*,V}} X (V^{*} V) 
\xrightarrow{\id\ot \ev_{V}}  X\one
\xrightarrow{\runit_X}X
\big] \ .
\end{align}

Combining Theorem \ref{thm:chiinj} 
and \cite[Sec.\,4.5]{Fuchs:2010mw} gives (see also \cite[Thm.\,3.11\,\&\,Prop.\,3.14]{Shimizu:2015}, as well as \cite{Creutzig:2016fms} and \cite[Thm.\,3.9]{Gainutdinov:2016qhz}):

\begin{theorem}\label{thm:S_C(phi_M)-algebramap}
The assignment 
$[V] \mapsto \modS_\cat(\phi_V)$
is an injective $\ok$-algebra homomorphism
$\Gr_\ok(\cat) \to \End(\id_\cat)$.
\end{theorem}

This theorem implies Equation \eqref{eq:verlinde-general}. If $\ok$ has characteristic zero, it allows one to compute the structure constants $N_{UV}^{~W}$ of $\Gr(\cat)$.

\medskip

In case that $\cat$ is ribbon, a short calculation with string diagrams shows that the antipode of $\coend$ acts on internal characters as $S_\coend \circ \chi_V = \chi_{V^*}$. Combining this with $\modS^2= S_\coend^{-1}$ from Theorem \ref{thm:MCG-1hole-torus} and transporting everything to $\End(\id_\cat)$ gives:

\begin{lemma}
Let $\cat$ be in addition ribbon. Then $\modS_\cat(\modS_\cat(\phi_V)) = \phi_{V^*}$.
\end{lemma}

In this sense, we can think of $\modS_\cat^2$ as implementing ``charge conjugation'' on the internal characters (and on their images in $\End(\id_\cat)$).

For $\cat$ pivotal 	(but not necessarily ribbon),
after a short calculation with string diagrams one can find the following expression for the $\phi_V$, which is a bit lengthy as we write out associators for later use:
\begin{align}\label{eq:phiV-rib-def}
(\phi_V)_X ~=~ \big[\, & X  \xrightarrow{\sim} 
\one (X \one)
\xrightarrow{\coev_{XV^*} \otimes \id\otimes \widetilde\coev_{V}  }\left\{(XV^*)(XV^*)^*\right\} \bigl(X  (V^{*} V) \bigr)
 \\ \nonumber & 
 \xrightarrow{\id\otimes\alpha_{X,V^{*},V}}  \left\{(XV^*)(XV^*)^*\right\} \bigl((X V^{*}) V\bigr)
    \\ \nonumber & 
  \xrightarrow{\alpha^{-1}_{ XV^* , (XV^*)^* ,(X V^{*}) V}}  (XV^*) \left\{ (XV^*)^* \bigl((X V^{*}) V\bigr)\right\}
   \\ \nonumber & 
 \xrightarrow{\id\otimes \alpha_{(XV^*)^*, X V^{*}, V}}  (XV^*)\left\{\bigl((XV^*)^* (X V^{*})\bigr) V\right\}
 \\ \nonumber & 
\xrightarrow{\id \otimes ( \coint_\coend \circ \iota_{XV^*} ) \otimes \id}
(X V^*) \left\{\one  V\right\} \xrightarrow{\sim}(X V^*)   V 
 \\ \nonumber & 
\xrightarrow{\alpha^{-1}_{X,V^*,V}}
X (V^*   V) \xrightarrow{\id\otimes \ev_V}  X\one  \xrightarrow{\sim} X\,\big] \ ,
\end{align}
or graphically
\be\label{eq:phi_M-explicit}
	(\phi_V)_X ~= \ipic{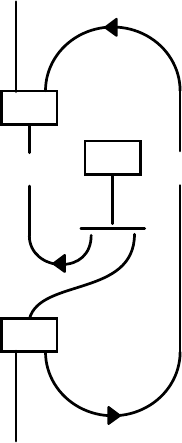}{.75} 
     \put (-64,-88) {\scriptsize$ X $}
     \put (-64,84) {\scriptsize$ X $}
     \put (-59,38) {\scriptsize$ \id $}
     \put (-59,-44) {\scriptsize$ \id $}
     \put (-48,2) {\scriptsize$ \iota_{XV^*} $}
     \put (-32,20.5) {\scriptsize$ \Lambda^{\mathrm{co}}_\coend $}
     \put (-64,16) {\scriptsize$ X{}V^* $}
     \put (-4,16) {\scriptsize$ V $}
 \quad .
\ee
For the derivation of this diagram and for a more detailed study of the properties of $\phi_V$, see~\cite{GR-prep}.

\medskip

\section{Ribbon quasi-Hopf algebras}\label{sec:ribbon-qHopf}

In this section we introduce our conventions for 
	ribbon
quasi-Hopf algebras $A$ and for the structure maps on the categories $\rep A$ of their finite-dimensional representations. We show that the factorisability of $\rep A$ is equivalent the factorisability condition on $A$ given in \cite{[BT]}. 

\subsection{Conventions} \label{subsec:conventions-ribbon-qHopf}
We begin with the definition of a quasi-Hopf algebra
$A$~\cite{Dr-quasi} and we mainly  follow the conventions in~\cite[Sec.\,16.1]{ChPr}. 
We will make the following
\begin{quote}
	\textit{Assumption:} We will only consider quasi-Hopf algebras $A$ such 
	  that 
	the unit isomorphisms $\lambda_U$ and $\runit_U$  in $\rep A$ are as in $\vect_\ok$.
\end{quote}	 
This simplifies for example the counit conditions~\eqref{eq:eps-Delta} and~\eqref{eq:counital} below as they do not involve non-trivial invertible elements $l$ and $r$.

We will use Sweedler's sum notation with primes $'$ for the coproduct $\Delta(a) \in A \otimes A$ of an element  $a \in A$, and with subscript numbers ${}_{1,2,\dots}$ for elements of tensor products of $A$. For example, 
\be\label{eq:expansion-convention}
	\Delta(a) = \sum_{(a)} a' \otimes a''
	\quad , \qquad
	X = \sum_{(X)} X_1 \otimes X_2 \otimes X_3\quad
	\text{for}~~ X \in A^{\otimes 3} \ .
\ee

\begin{definition} \label{def:quasi-Hopf} \textit{A quasi-Hopf algebra} over a field $\ok$ is a unital associative algebra $A$ over $\ok$ together with
\begin{itemize}\setlength{\leftskip}{-1.5em}
\item an algebra homomorphism 
	$\eps: 
	A\to \oC$ (the {\em counit}),
\item an algebra homomorphism $\Delta: A\to A\tensor A$ (the {\em coproduct}),
\item an algebra anti-homomorphism $S: A\to A$ (the {\em antipode}),
\item a multiplicatively invertible element $\as\in A\tensor A\tensor A$ (the {\em coassociator}),
\item elements $\Salpha$, $\Sbeta\in A$ (the {\em evaluation} and {\em coevaluation element}, respectively).
\end{itemize}
These data are subject to the conditions:
\begin{itemize}\setlength{\leftskip}{-1.5em}
\item counitality and coassociativity:
\begin{align}\label{eq:eps-Delta}
	&(\eps\tensor\id)\circ\Delta ~=~ \id ~=~  (\id\tensor\eps)\circ\Delta \ ,
\\
\label{eq:as-intertwiner}
&\big((\Delta\tensor\id)(\Delta(a)) \big) \cdot \as ~=~ \as \cdot \big( (\id\tensor\Delta)
(\Delta(a))\big) \quad \text{ for all  } a\in A \ ,
\end{align}
\item the coassociator $\as$ is counital and a $3$-cocycle:
\begin{align}\label{eq:counital}
	& (\id\tensor\eps\tensor\id)(\as) = \one\tensor\one \ ,
\\
\label{eq:3-cocycle}
& (\Delta\tensor\id\tensor\id)(\as)\cdot (\id\tensor\id\tensor\Delta)(\as) = (\as\tensor\one)\cdot(\id\tensor\Delta\tensor\id)(\as)\cdot(\one\tensor\as) \ ,
\end{align}
\item the antipode conditions:

\begin{align}
\label{eq:Salpha-1}
	& \sum_{(a)}S(a') \, \Salpha \, a'' = \eps(a) \, \Salpha\ ,\qquad
	\sum_{(a)}a' \, \Sbeta  \, S(a'') = \eps(a) \, \Sbeta
\quad \text{ for all } a\in A \ ,
\\ &
\label{eq:Salpha-2}
\sum_{(\as)}S(\as_1) \, \Salpha  \, \as_2 \, \Sbeta  \,  S(\as_3) = \one\ ,\qquad
\sum_{(\as^{-1})}(\as^{-1})_1 \, \Sbeta \,  S((\as^{-1})_2) \, \Salpha   \, (\as^{-1})_3 = \one \ ,
\end{align}
for an expansion 
$\as= \sum_{(\as)}\as_1\tensor \as_2\tensor \as_3\in A\tensor A\tensor A$ 
and similarly for $\as^{-1}$, 
cf.\ \eqref{eq:expansion-convention}.
\end{itemize}
\end{definition}

\begin{remark}
\mbox{}
\begin{enumerate}\setlength{\leftskip}{-1.5em}
\item
We  note that the antipode $S$, as well as $\Salpha$ and $\Sbeta$ are uniquely determined up to the conjugation by a unique element $U$:
if the triple $\tilde{S}$, $\tilde{\Salpha}$, $\tilde{\Sbeta}$ gives another antipode structure in~$A$ then
there exists a unique element $U\in A$ such that
 \begin{equation}
\tilde{S}(a) = U S(a)U^{-1}\ ,\qquad
\tilde{\Salpha} = U \Salpha\ ,\qquad
\tilde{\Sbeta} = \Sbeta U^{-1} \ ,
\end{equation}
see~\cite[Prop.\,1.1]{Dr-quasi} for details.
\item 
Every Hopf algebra is also a quasi-Hopf algebra for which $\as = \one \tensor \one \tensor \one$ and $\Salpha = \Sbeta = \one$.
\end{enumerate}
\end{remark}

Let us denote by $\tau$ the symmetric braiding in vector spaces, 
i.e.\ for vector spaces $U,V$ and $u\in U$, $v\in V$ we set 
\be
	\tau_{U,V}(u \otimes v) ~=~ v \otimes u \ .
\ee

\begin{definition}\label{def:quasi-triang_for_quasi-hopf}
A quasi-Hopf  $A$ is 
\textit{quasi-triangular} if it is
		equipped with an invertible element $R\in A\tensor A$, called \textit{the universal R-matrix}, such that
\begin{itemize}\setlength{\leftskip}{-1.5em}
\item $R$ relates the coproduct with the opposite coproduct
 $\Delta^{\operatorname{op}}:=\tau_{A,A} \circ\Delta$ as
\begin{equation}\label{eq:RD=DopR}
R \, \Delta(a) = \Delta^{\mathrm{op}}(a) \, R
\quad \text{ for all } a\in A \ ,
\end{equation}
\item the quasi-triangularity conditions hold:
\begin{equation}\label{eq:R-mat-hex12}
\begin{split}
(\Delta\tensor\id)(R) &= 
	 (\as^{-1})_{231}
\, R_{13} \, \as_{132} \, R_{23}\,  \as^{-1} \ ,
\\
(\id\tensor\Delta)(R) &= \as_{312} \, R_{13} \, \as_{213}^{-1} \, R_{12} \, \as \ .
\end{split}
\end{equation}
Here we set $\as_{231}= \sum_{(\as)}\as_2\tensor \as_3\tensor \as_1$ and $R_{13}=\sum_{(R)}R_1\tensor\one\tensor R_2$, \textit{etc.}
\end{itemize}
\end{definition}

\begin{remark}\label{rem:q-Hopf-gives-mon}
The data of a quasi-triangular quasi-Hopf algebra $A$ allows one to turn $\rep A$ into a $\ok$-linear braided
	category with left duals as follows.
\begin{itemize}\setlength{\leftskip}{-1.5em}
\item
The associativity isomorphism $\alpha_{U,V,W} : U \tensor (V \tensor W) \to (U \tensor V) \tensor W$ for the tensor product (over $\ok$) of $A$-modules $U,V,W$ is given by
\begin{equation} \label{RepA-coass}
\alpha_{U,V,W}(u\tensor v\tensor w) = \as.(u\tensor v\tensor w)\ ,
\end{equation}
for any elements $u\in U$, $v \in V$, $w \in W$. The 3-cocycle condition on $\as$ is equivalent to the commutativity of the pentagon diagram for $\alpha$.

\item The antipode structure on $A$ gives rise to left duals for $\rep A$.
Namely, the left dual $U^*$ for $U$ in $\rep A$ is the vector space dual to $U$ together with the $A$-action
\begin{equation}\label{eq:dualU}
(a\cdot f)(u) := f(S(a)u) \ ,\qquad u\in U,\quad f\in U^*, \quad a\in A \ .
\end{equation}
The elements $\Salpha$ and $\Sbeta$ enter the definition of the evaluation and coevaluation morphisms as
\begin{equation}\label{eq:ev-coev}
\ev_U: \; \phi\tensor u \mapsto \phi(\Salpha . u)~~,\qquad \coev_U: \; 1\mapsto \sum_i (\Sbeta . u_i)\tensor u^*_i \ ,
\end{equation}
where $\phi\in U^*$, $u\in U$, and $\{u_i\}$ is a basis of $U$ with $\{u^*_i\}$ is the corresponding dual basis. 

\item
The  braiding isomorphisms $\sigma_{U,V}$ in  $\rep A$ are defined in terms of  the  universal R-matrix as
\be \label{RepA-braid}
\sigma_{U,V}(u \otimes v) = \tau_{U,V}(R. (u \otimes v)) \ .
\ee
Due to~\eqref{eq:R-mat-hex12},  the isomorphisms $\sigma_{U,V}$ satisfy 
	 the hexagon axioms of a braided monoidal category.
	Applying the linear map $\id\tensor\eps\tensor\id$ to both equations in~\eqref{eq:R-mat-hex12} and using the counital condition~\eqref{eq:counital}, 
we obtain the following result for a quasi-Hopf algebra under our \textit{Assumption}~\cite[Sec.\,3]{Dr-quasi}:
\begin{equation}\label{eq:eps-R-compat}
(\eps\tensor \id)(R) = \one = (\id\tensor \eps)(R) \gp
\end{equation}
 These equalities  correspond to the commutativity of the diagram involving the left and right unit isomorphisms and the braiding.\footnote{
	One can also demand the identities \eqref{eq:eps-R-compat} as part of the definition of quasi-triangularity in Definition~\ref{def:quasi-triang_for_quasi-hopf}. In this case one does not need to require that $R$ is invertible, see \cite{Bulacu:2003}.}
\end{itemize}
We will often use the {\em monodromy element}
\be\label{eq:monodromy-el}
	M ~:=~ R_{21} \, R ~\in ~ A \otimes A \ .
\ee
It describes the double braiding in $\rep A$: $\sigma_{V,U} \circ \sigma_{U,V}(u \otimes v) = M.(u\otimes v)$.

Let us give string diagrams for the structure maps in Remark \ref{rem:q-Hopf-gives-mon}. The action $A\ot U\to U$ of $A$ on a left $A$-module $U$ will be written as
\be 
\ipic{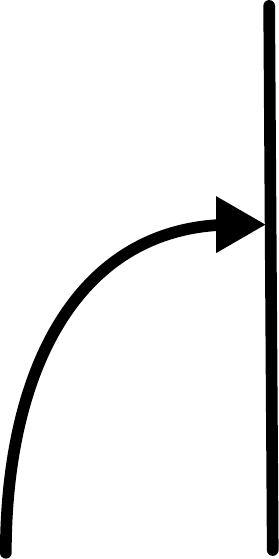}{.17}
\put(-27,-31){{\scriptsize $A$}}  \put(-4,-31){{\scriptsize $U$}} \put(-4,25){{\scriptsize $U$}} 
\quad\raisebox{2.5em}{\framebox{\scriptsize{$\vect_\ok$}}}
\ee
Here, the $\vect_\ok$ in a box indicates that the corresponding string diagram is to be taken in $\vect_\ok$.
For the structure morphisms \eqref{RepA-coass}, \eqref{eq:ev-coev} and \eqref{RepA-braid} we get
\begin{align}
\assoc_{U,V,W}(u\ot v\ot w) &~=~
\ipic{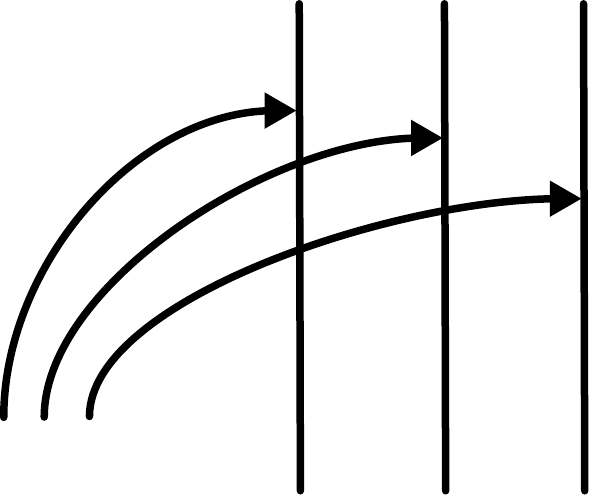}{.25}
\put(-68,-30){\scriptsize $\Phi$} \put(-38,-38){{\scriptsize $U$}}  \put(-21,-38){{\scriptsize $V$}} \put(-4,-38){{\scriptsize $W$}} 
\hspace{.9em}\raisebox{2.7em}{\framebox{\scriptsize{$\vect_\ok$}}}  \qcq  \\[10pt] 
\ev_U(u) &~=~ \ipic{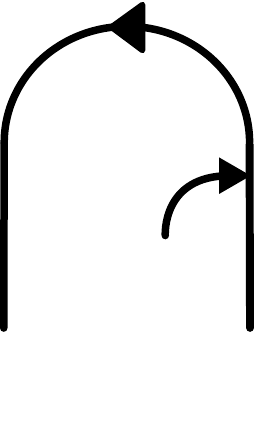}{.25}
\put(-34,-23){{\scriptsize $U^*$}}  \put(-14,-10){{\scriptsize $\Salpha$}} \put(-4,-23){{\scriptsize $U$}} 
\ \raisebox{3em}{\framebox{\scriptsize{$\vect_\ok$}}} \qcq 
\coev_U(u) ~=~ \ipic{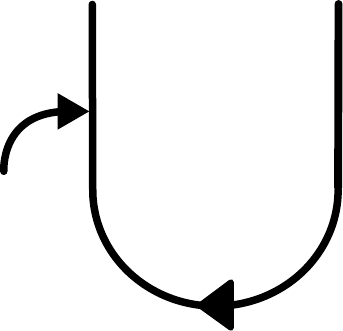}{.25}
\put(-34,23){{\scriptsize $U$}}  \put(-45,-10){{\scriptsize $\Sbeta$}} \put(-4,23){{\scriptsize $U^*$}} 
\ \raisebox{4em}{\framebox{\scriptsize{$\vect_\ok$}}}\qcq \\[10pt]  
\sigma_{U, V}(u\ot v) &~=~\ipic{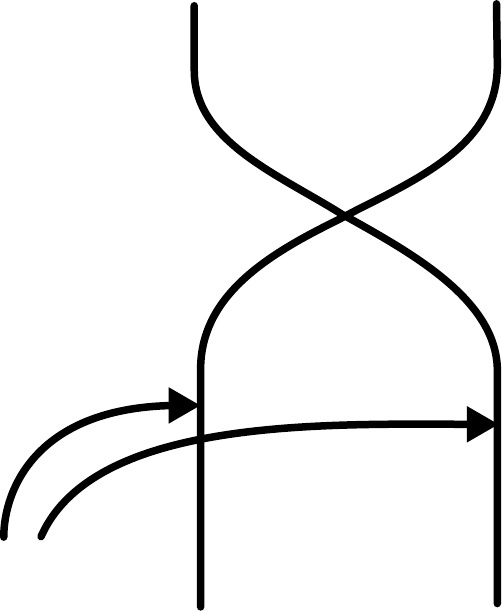}{.25}
\put(-63,-37){\scriptsize $R$} \put(-40,-46){{\scriptsize $U$}}  \put(-4,-46){{\scriptsize $V$}} \put(-40,40){{\scriptsize $V$}}  \put(-4,40){{\scriptsize $U$}} 
\quad \raisebox{4em}{\framebox{\scriptsize{$\vect_\ok$}}}\qcq 
\end{align}
for any $u\in U$, $v\in V$  and $w\in W$.
Note that since the string diagrams are in $\vect_\ok$, 
	the diagrams for duality maps and symmetric braiding represent the duality maps and braiding of $\vect_\ok$, not those of $\rep A$.
	
We will denote the standard pivotal structure of $\vect_\ok$ by 
\be\label{eq:pivotal-vect}
\delta^{\vect} : (-) \rightarrow (-)^{**}
\quad , \quad  \delta^{\vect}_V(v) = \langle-,v\rangle \ ,
\ee
where $V \in \vect_\ok$, $v \in V$
	and $\langle-,-\rangle$ denotes the pairing between $V^*$ and $V$.
\end{remark}

\begin{definition}\label{def:fact-qHopf}
A finite-dimensional quasi-triangular quasi-Hopf algebra $A$ is called \textit{factorisable} if its representation category $\rep A$ is factorisable in the sense of Definition~\ref{def:fact-cat}.\footnote{
	Definition \ref{def:fact-cat} is formulated in terms of braided finite tensor categories, but we have so far only introduced left duals. However,
given $V \in \rep A$, one can use the left duality morphisms $\ev_{V^*}$ and $\coev_{V^*}$ from above together with the Drinfeld isomorphism $u_V$ from \eqref{eq:uiso} to obtain right duality morphisms as in \eqref{eq:left-duals-def} such that $V^*$ is right dual to $V$ (in addition to being left dual to $V$).  We will not spell this out here as these right duals are different from the ones in the ribbon case that we intend to use later (see Section \ref{sec:qHopf-pivot}). But we stress that $\rep A$ is rigid, and in fact is a braided finite tensor category over $\ok$.
}
\end{definition}

\newcommand{\Dm}{\mathcal{D}_\coend}
\newcommand{\hDm}{\hat{\mathcal{D}}_\coend}

\newcommand{\pRel}{\boldsymbol{p}}
\newcommand{\pLel}{\tilde{\boldsymbol{p}}}
\newcommand{\qRel}{\boldsymbol{q}}
\newcommand{\qLel}{\tilde{\boldsymbol{q}}}

\newcommand{\ourQ}{Q}

\begin{remark}\label{rem:fact-qHopf-def}
~\\[-1.9em]
\begin{enumerate}\setlength{\leftskip}{-1.5em}
\item
An equivalent way to phrase Definition \ref{def:fact-qHopf} is to say that
$\Dm: \coend \to \coend^*$ from \eqref{eq:Omega} is invertible for the universal Hopf algebra $\coend$ in $\cat=\rep A$.
We will show in Section~\ref{sec:coend-repA} how this somewhat indirect definition of factorisability can be made explicit in terms of data of the quasi-Hopf algebra $A$. 
Namely, we will see that we can take $\coend = A^*$
and that the composition 
$A^* \xrightarrow{\Dm} A^{**} \xrightarrow{(\delta^{\vect}_A)^{-1}} A$
 can be written as $\phi \mapsto (\id \otimes \phi)( \hDm )$, 
where we used the
  isomorphism $\delta^{\vect}_A$ from~\eqref{eq:pivotal-vect}.
The element $\hDm \in A\tensor A$ is given by
\be\label{eq:Q-for-fact}
\hDm ~= 
\sum_{(X),(W)} S(W_3 X_2') W_4 X_2'' \otimes S(W_1 X_1') W_2 X_1''
\ ,
\ee
where 
$X \in A^{\otimes 2}$, $W \in A^{\otimes 4}$ are defined as
\begin{align}\label{eq:Rem-X}
X &~=~ \sum_{(\Phi)} \Phi_1 \otimes \Phi_2 \Sbeta S(\Phi_3)\ ,
\\ \nonumber
W&~=~(\one\tensor\Salpha\tensor\one\tensor\Salpha)\cdot
(\one\tensor\as^{-1})\cdot(\one\tensor M
\tensor\one)\cdot(\one\tensor\as)
\cdot(\id\tensor\id\tensor\Delta)(\as^{-1}) \ ,
\end{align}
and $M$ is as in \eqref{eq:monodromy-el}.
Thus, Definition \ref{def:fact-qHopf} states that $A$ is factorisable if and only if $\hDm \in A \otimes A$ is a non-degenerate copairing of vector spaces,
i.e.\ iff $\hDm = \sum_{i\in I} a_i\otimes b_i$ for two bases $\{ a_i \; |
 \; i\in I\}$ and $\{b_i \; | \; i\in I\}$ of $A$.

\item
Let us specialise the factorisability condition to the case that $A$ is a Hopf algebra. Then we have the trivial coassociator $\as=\one^{\otimes3}$ and $\Salpha=\Sbeta=\one$.
	Equation
\eqref{eq:Q-for-fact} reduces to 
$\hDm = \sum_{(M)} S(M_2) \otimes M_1$ and 
the map
 $(\delta^{\vect}_A)^{-1} \circ \Dm : A^* \to A$ becomes
$\phi \mapsto S \circ \big(	(\phi\tensor \id)(M) \big)$.
Thus $(\delta^{\vect}_A)^{-1} \circ \Dm$
is equal to the well-know Drinfeld mapping~\cite{[Drinfeld]}
composed with the antipode.
We conclude that $\hDm \in A \otimes A$ is a non-degenerate copairing if and only if the Drinfeld mapping
 is invertible. 
The latter condition is the usual definition of a factorisable Hopf algebra~\cite{RS}.
\end{enumerate}
\end{remark}

\subsection{Drinfeld twist} \label{subsec:Dtwist}

By definition, the antipode of a quasi-Hopf algebra is an algebra anti-homomorphism. However, 
in contrast to Hopf algebras it is in general not an coalgebra anti-homomorphism,  i.e.\ the equality $\Delta \bigl(S(a)\bigr) = (S\tensor S)(\Delta^{\operatorname{op}}(a))$ may not hold. Instead, the right hand side is conjugated by the Drinfeld twist~\cite{Dr-quasi}. The Drinfeld twist is the invertible element $\Dt\in A\tensor A$ given by
\begin{equation}\label{def:F}
\Dt = \sum_{(\as)} (S\tensor S)(\Delta^{\operatorname{op}}(\as_1)) \cdot \gamma \cdot \Delta\big( \as_2\Sbeta S(\as_3) \big)    
\end{equation}
with
\begin{equation} 
\gamma = \sum_{(X)} (S(X_2)\Salpha X_3) \tensor (S(X_1)\Salpha X_4) 
\quad \text{where}
\quad X = (\one\tensor\as)\cdot (\id\tensor\id\tensor\Delta)(\as^{-1}) \ .
 \label{def:gamma2}
\end{equation}
In terms of $\Dt$, $\Delta$ and $\Delta^{\operatorname{op}}$ are related by (see \cite{Dr-quasi})
 \begin{equation}\label{eq:anti-coalg}
\Dt\Delta \bigl(S(a)\bigr) \Dt ^{-1} = (S\tensor S)(\Delta^{\operatorname{op}}(a))  \ , \qquad a\in A \ .
\end{equation} 

\begin{lemma}[\cite{Dr-quasi}]
As morphisms in $\rep A$, $\gamma_{N,M}$ and
$\tilde\gamma_{N,M}$
from \eqref{eq:gammaUV} and \eqref{eq:gamma-tildeUV}, respectively, are given by  
\begin{align}\label{eq:gamma}
\big(\gamma_{N,M}(\varphi\tensor\psi)\big)(n\tensor m) &= (\psi\tensor\varphi) (\Dt. \, n\tensor m) \ , \\
\tilde\gamma_{N,M}(\varphi\tensor\psi\tensor n\tensor m) &= (\psi\tensor\varphi) (\gamma. \, n\tensor m) \ ,
\end{align}
where $m\in M$, $n\in N$, $\varphi\in M^\ast$, $\psi\in N^\ast$ and $\Dt$ and $\gamma$ are defined in \eqref{def:F} and \eqref{def:gamma2}, respectively. 
\end{lemma}
\begin{proof}
We begin with $\tilde\gamma$. Let $X=(\one\ot\Phi)\cdot(\id\ot\id\ot\Delta)(\Phi^{-1})$. 
Then \eqref{eq:gamma-tildeUV} gives
\begin{align}
\tilde\gamma_{N,M}(\varphi\ot\psi\ot n\ot m) &= \ev_M \circ (\id\ot\ev_N\ot\id)(X\,.\,\varphi\ot\psi\ot n\ot m) \ .
\end{align}
Note that 
\be \ev_N(a\ot b\,.\,\psi\ot n)=\psi(S(a)\Salpha b\,.\,n) \ . \ee 
Indeed, we have
\be\label{proof-gamma-b}
   \ipic{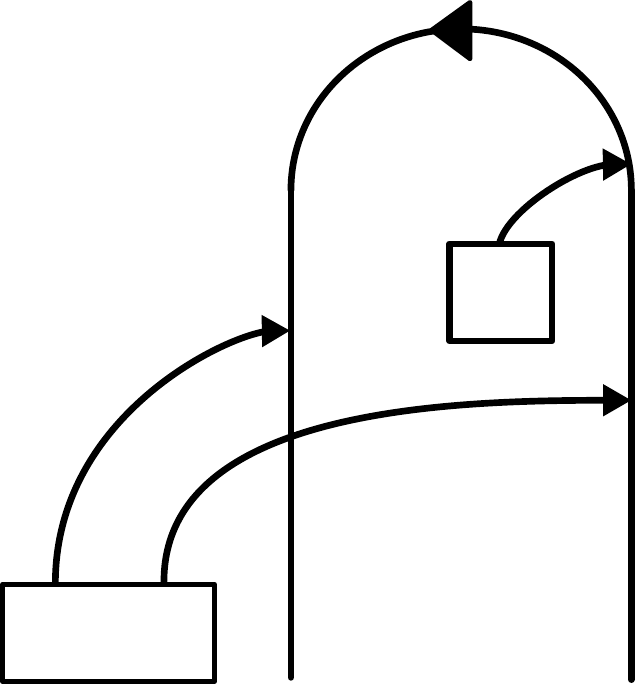}{.25}
	 \put(-45,-50){{\scriptsize $N^\ast$}}\put(-5,-50){{\scriptsize $N$}}
	 \put(-74,-38){\scriptsize$a\ot b$} \put(-20,4){\scriptsize$\Salpha$} 
	 \raisebox{4em}{\framebox{\scriptsize{$\vect_\ok$}}}
	 \quad = \quad 
	 \ipic{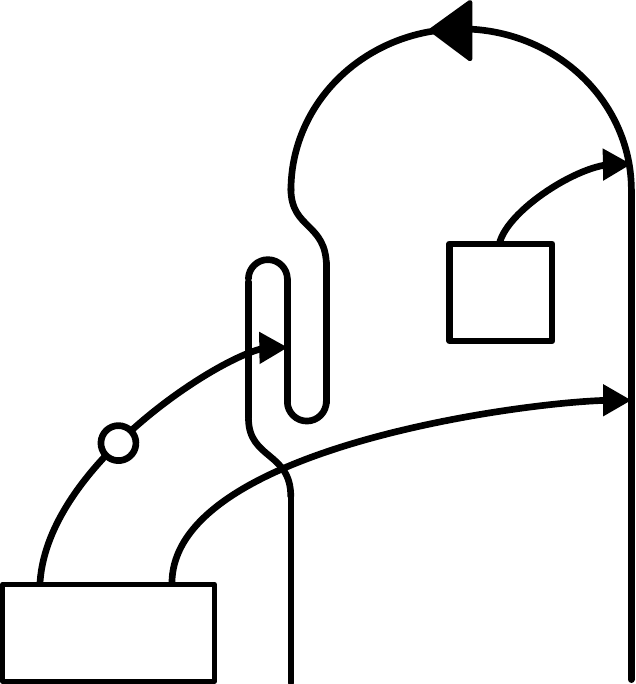}{.25}
	 \put(-45,-50){{\scriptsize $N^\ast$}}\put(-5,-50){{\scriptsize $N$}}
	 \put(-74,-38){\scriptsize$a\ot b$} \put(-20,4){\scriptsize$\Salpha$} 
	 \quad = \quad 
	 \ipic{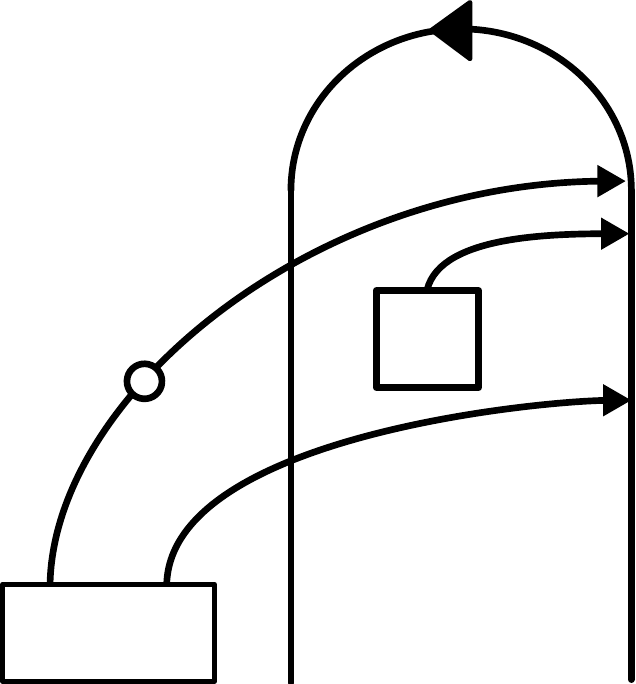}{.25}
	 \put(-45,-50){{\scriptsize $N^\ast$}}\put(-5,-50){{\scriptsize $N$}}
	 \put(-74,-38){\scriptsize$a\ot b$} \put(-29,-2){\scriptsize$\Salpha$}  
		\qp 
\ee
It follows then	from the above equalities
that
\begin{align}
\tilde\gamma_{N,M}(\varphi\ot\psi\ot n\ot m) 
&= \sum_{(X)}\ev_M (X_1\ot X_4\,.\,\varphi\ot m) \, \psi(S(X_2)\Salpha X_3\,.\,n) \\ \nonumber
&= \sum_{(X)}\varphi (S(X_1)\Salpha X_4\,.\, m) \, \psi(S(X_2)\Salpha X_3\,.\,n) \\ \nonumber
&= \sum_{(X)}(\psi\ot\varphi)(S(X_2)\Salpha X_3\ot S(X_1)\Salpha X_4 \,.\, n\ot m) \\ \nonumber
&= (\psi\ot\varphi)(\gamma \,.\, n\ot m) \ .
\end{align}

Equation \eqref{eq:gamma} follows by recalling \eqref{eq:gammaUV} and using the identity
\be\label{proof-gamma}
   \ipic{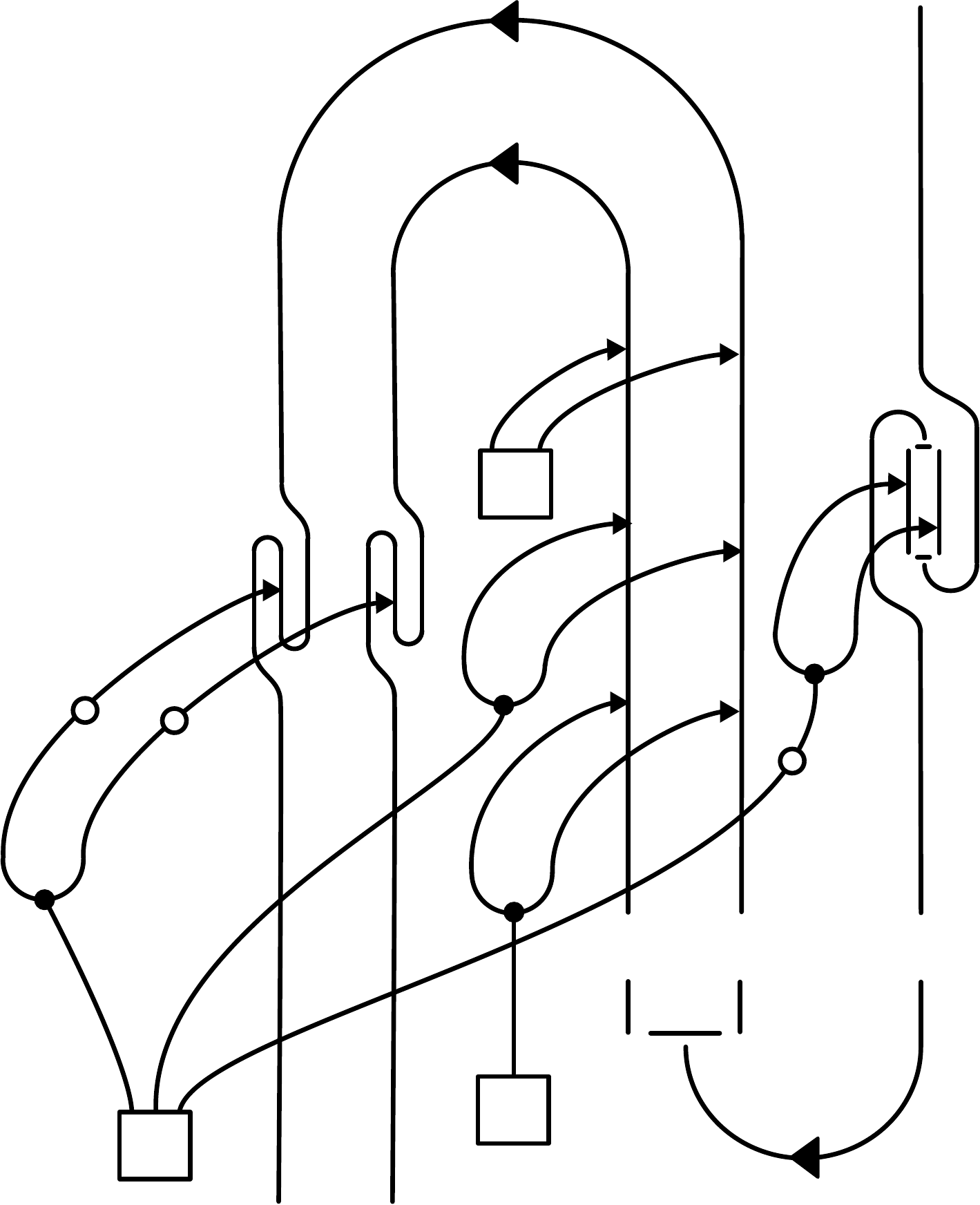}{.25}
	 \put(-127,-113){{\scriptsize $N^\ast$}}\put(-108,-113){{\scriptsize $M^\ast$}}
	 \put(-147,-98){\scriptsize$\Phi$} \put(-85,-91){\scriptsize$\Sbeta$} \put(-84,19){\scriptsize$\gamma$}
	 \put(-65,-63){{\scriptsize $N$}} \put(-48,-63){{\scriptsize $M$}}\put(-21,-63){{\scriptsize $(NM)^\ast$}}
	 \raisebox{8em}{\framebox{\scriptsize{$\vect_\ok$}}}
	 \quad = \quad
	 \ipic{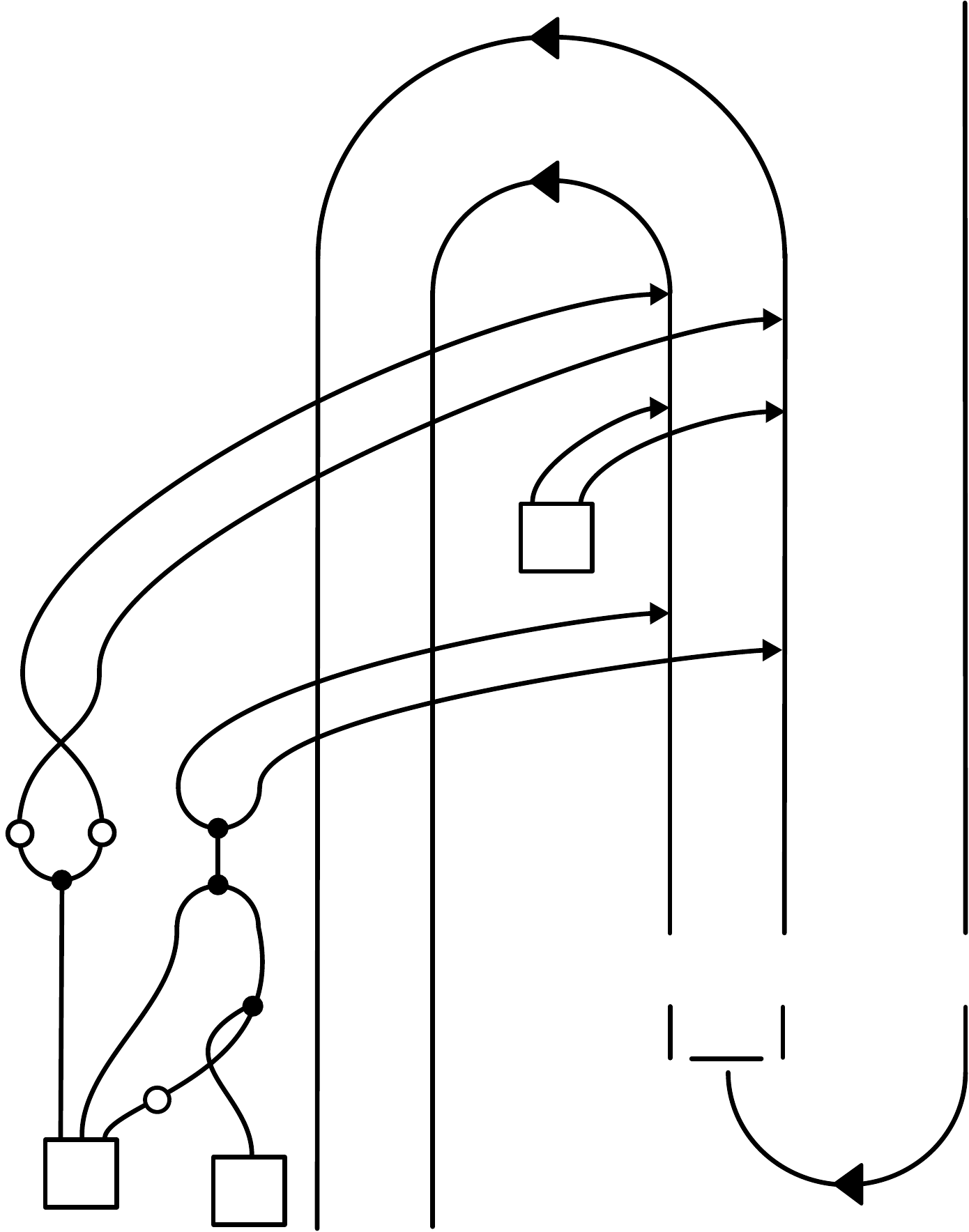}{.25}
	 \put(-118,-114){{\scriptsize $N^\ast$}}\put(-98,-114){{\scriptsize $M^\ast$}}
	 \put(-157,-99){\scriptsize$\Phi$} \put(-128,-102){\scriptsize$\Sbeta$} \put(-74,12){\scriptsize$\gamma$}
	 \put(-55,-63){{\scriptsize $N$}} \put(-38,-63){{\scriptsize $M$}}\put(-14,-63){{\scriptsize $(NM)^\ast$}}
\ee
\end{proof}

\subsection{Drinfeld element} \label{sec:Drinfeld-element}
Recall Drinfeld's canonical natural isomorphism 
	$u : \id_\cat \to (-)^{**}$ 
from \eqref{eq:uiso}. 
In terms of the data of the quasi-triangular quasi-Hopf algebra, $u_U$, for $U\in \rep A$, becomes
\be\label{def:Drinfeld-element-vect} 
   u_U \quad = \quad \ipic{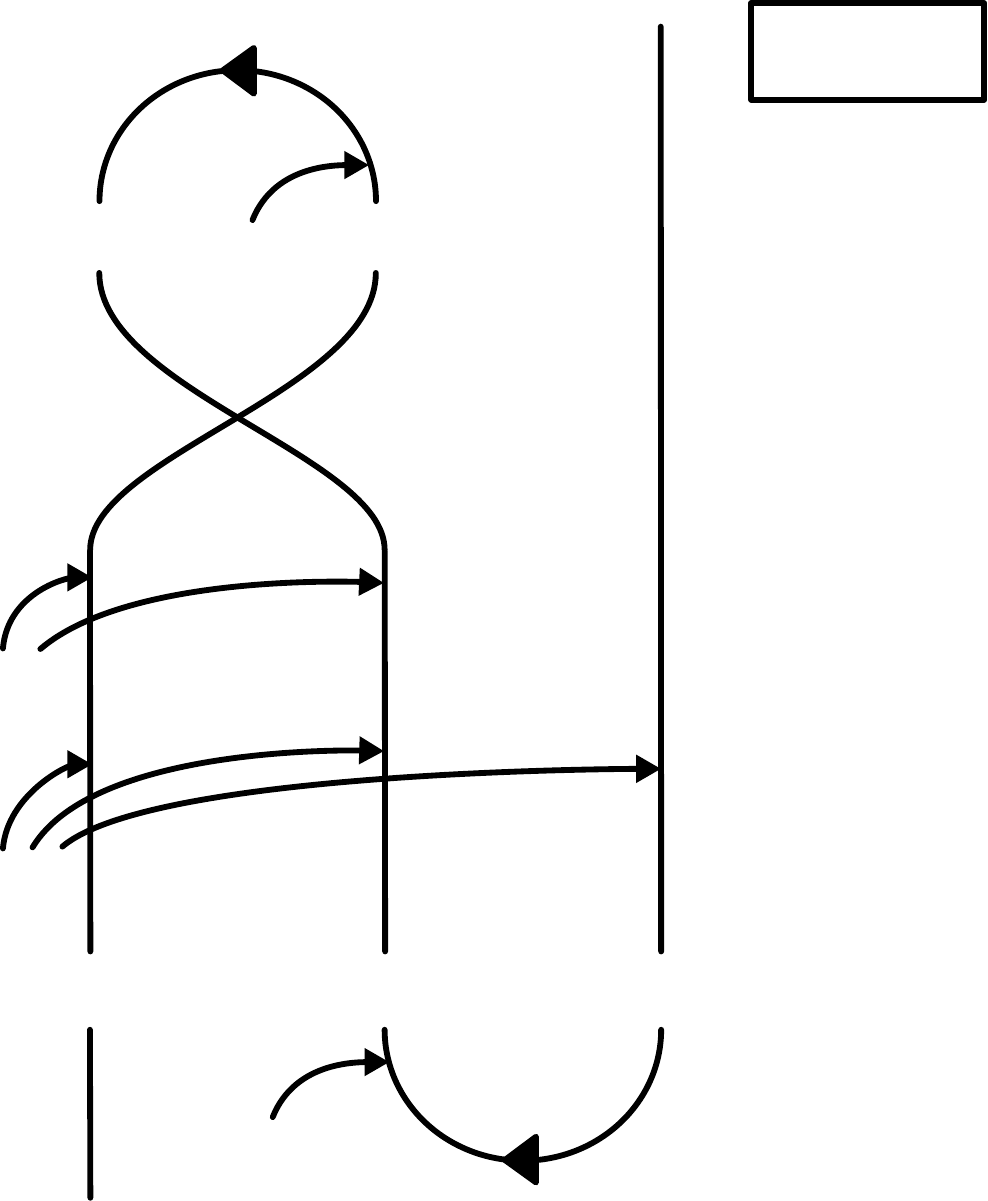}{.3}
	 \put(-133,-59){\scriptsize$U$} \put(-90,-59){\scriptsize$U^\ast$}  \put(-50,-59){\scriptsize$U^{\ast\ast}$}
	 \put(-133,50){\scriptsize$U^\ast$} \put(-91,50){\scriptsize$U$}
	 \put(-142,-16){\scriptsize$R$}  \put(-141,-45){\scriptsize$\Phi$} 
	 \put(-110,-82){\scriptsize$\Sbeta$}  \put(-110,48){\scriptsize$\Salpha$} 
	 \put(-27,77){\scriptsize$\vect_\ok$}
\ee

Abbreviate $X=(\Salpha\ot\one\ot\one)\cdot (R\ot\one) \cdot \Phi \cdot (\one\ot\Sbeta\ot\one)$
and recall the  standard pivotal structure of $\vect_\ok$ in~\eqref{eq:pivotal-vect}.
 Then we have
\be\label{def:Drinfeld-element-X} 
    u_U \quad = \quad \ipic{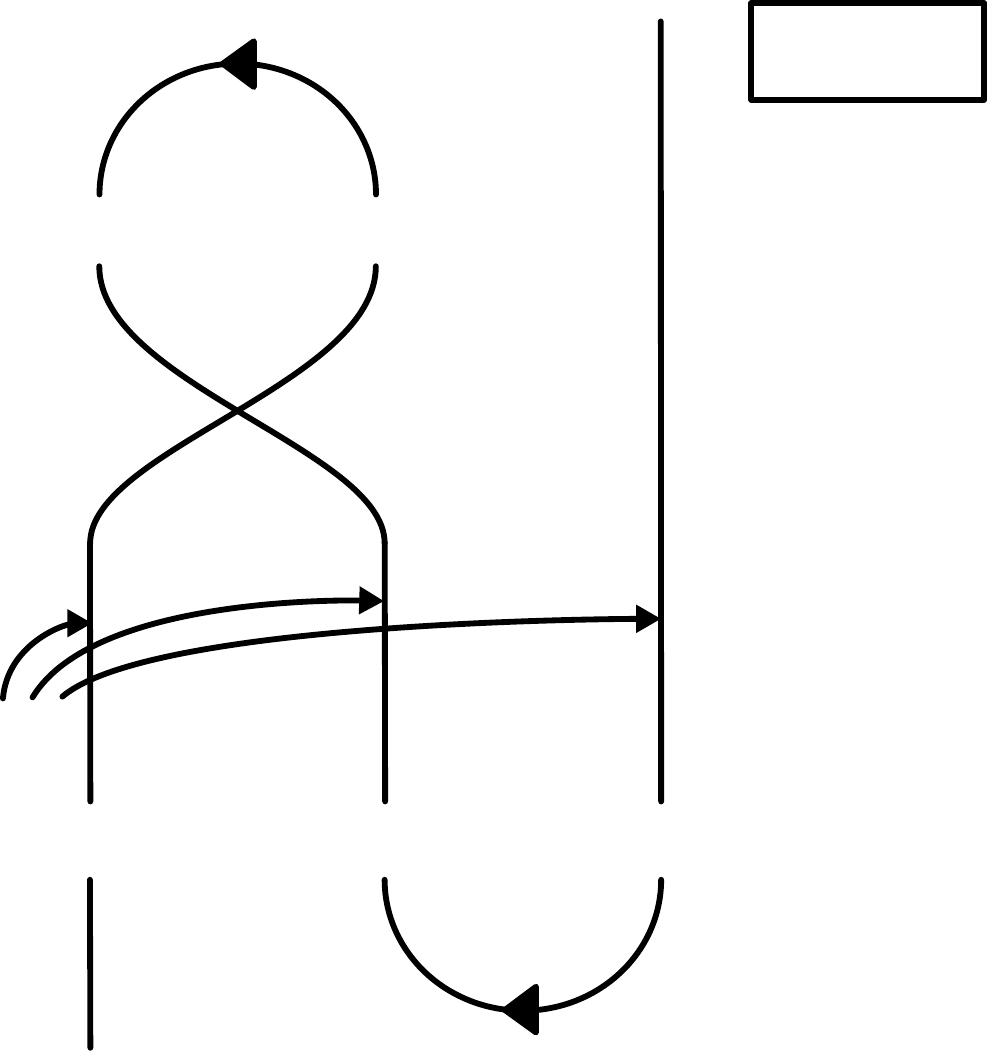}{.3}
	 \put(-133,-48){\scriptsize$U$} \put(-90,-48){\scriptsize$U^\ast$}  \put(-50,-48){\scriptsize$U^{\ast\ast}$}
	 \put(-133,40){\scriptsize$U^\ast$} \put(-91,40){\scriptsize$U$}
	 \put(-144,-35){\scriptsize$X$} 
	 \put(-27,66){\scriptsize$\vect_\ok$}
	 \quad = \quad \ipic{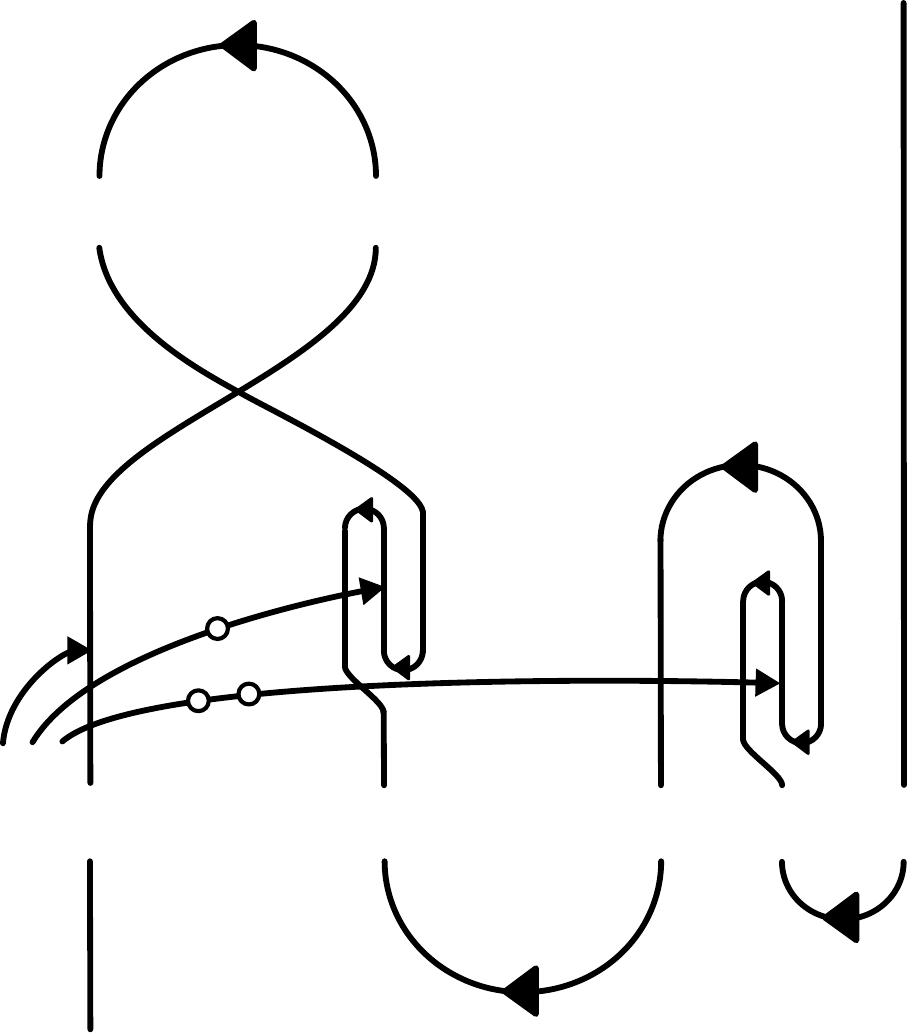}{.3}
	 \put(-121,-47){\scriptsize$U$} \put(-80,-47){\scriptsize$U^\ast$}  \put(-40,-47){\scriptsize$U^{\ast\ast}$} \put(-20,-47){\scriptsize$U^\ast$} \put(-3,-47){\scriptsize$U^{\ast\ast}$}
	 \put(-123,41){\scriptsize$U^\ast$} \put(-81,41){\scriptsize$U$}
	 \put(-134,-44){\scriptsize$X$} 
	 \quad = \quad \ipic{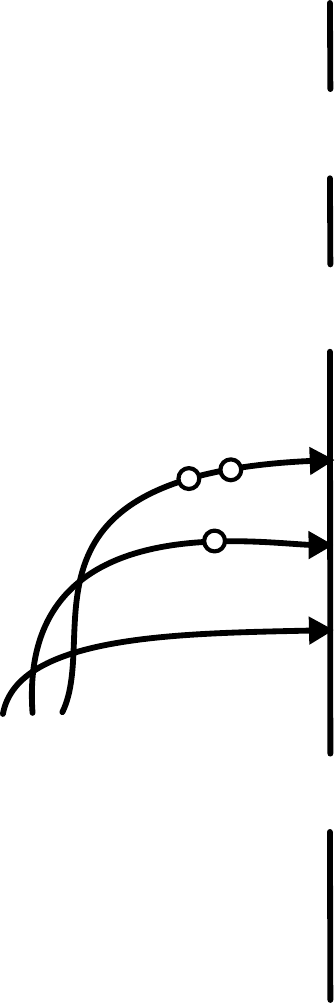}{.3}
	 \put(-3,-45){\scriptsize$U$} \put(-3,25){\scriptsize$\delta_{U}^{\vect}$} \put(-3,50){\scriptsize$U^{**}$}
	 \put(-47,-41){\scriptsize$X$} \qp
\ee
We conclude that 
\be\label{eq:drinfeld-iso-ddual}
u_U=\delta_{U}^{\vect}\circ (\sqs\act (-)) \ ,
\ee 
where $\sqs \in A$ is the {\em Drinfeld element}. Explicitly,
\begin{equation}\label{sqs}
\sqs = \sum_{(\as),(R)} S(\as_2\Sbeta S(\as_3)) \, S(R_2) \, \Salpha \, R_1 \, \as_1 \ .
\end{equation}
The Drinfeld element satisfies
\be\label{eq:S2-inner}
	S^2(a) = \sqs a\sqs^{-1} \ ,
\ee	
for any $a\in A$, 
see \cite[Sect.\,3]{[AC]}. 

The corresponding calculation for the variant $\tilde u$ in \eqref{eq:uiso-tw} gives the same expression with $R$ replaced by $R^{-1}$. For later reference, we state it explicitly:
\be\label{eq:tilde-sqs}
\tilde \uiso_U=\delta_{U}^{\vect}\circ (\tilde\sqs\act (-)) 
\quad , \quad
\tilde\sqs = \sum_{(\as),(R^{-1})} S(\as_2\Sbeta S(\as_3)) \, S((R^{-1})_2) \, \Salpha \, (R^{-1})_1 \, \as_1 \ .
\end{equation}

\subsection{Ribbon element}\label{sec:ribbon-element}

 A quasi-triangular quasi-Hopf algebra $A$ is called \textit{ribbon}  if
it contains a ribbon element $\ribbon$ defined in the same way as for ordinary Hopf algebras \cite{Yorck}:

\begin{definition}\label{def:ribbonel}
A nonzero central element $\ribbon\in A$ is called a {\em ribbon element} if it
satisfies
\begin{equation}\label{def-ribbon}
  M \cdot \Delta(\ribbon)~=~\ribbon\tensor\ribbon,\qquad
 S(\ribbon)=\ribbon.
\end{equation}
\end{definition}

In a ribbon quasi-Hopf algebra $A$ we have the identities~\cite{[AC],Yorck}
\begin{equation}\label{u-ribbon}
   \ribbon^2= \sqs S(\sqs) \quad ,\quad \eps(\ribbon)=1 \ ,
\end{equation}
where $\sqs$ is the canonical Drinfeld element defined in~\eqref{sqs}.
By convention, the ribbon twist $\theta_U$ on an object $U$ is given by acting with $\ribbon^{-1}$:
\be\label{eq:RepA-ribbon}
	\theta_U ~=~ \ribbon^{-1}.(-) \ .
\ee

\subsection{Pivotal structure}\label{sec:qHopf-pivot}

For a ribbon category $\cat$ with ribbon twist $\theta$, one can define a pivotal structure $\delta_X : X \to X^{**}$ by $\delta_X =  u_X \circ \theta_X$, where $u_X$ is as in \eqref{eq:uiso}, see e.g.\ \cite[Sec.\,8.10]{EGNO-book}. 
Recall from \eqref{eq:uiso-tw} that $\tilde u_X$ is obtained from $u_X$ by using the inverse braiding. 
A short calculation shows that \eqref{eq:theta-braiding-prop} implies $u_X \circ \theta_X = \tilde u_X \circ \theta_X^{-1}$. Thus we can also write $\delta_X = \tilde u_X \circ \theta_X^{-1}$.

The pivotal structure and the left duality morphisms allow one to turn the left dual $U^*$ of an object $U$ into a right dual. The right duality morphisms are
\begin{align}
	\widetilde\ev_U
	~&=~ \big[\,
	UU^* \xrightarrow{\delta_U \otimes \id}
	U^{**}U^*
	\xrightarrow{\ev_{U^*}}
	\one
	\, \big],\\ \nonumber
	\widetilde\coev_U
	~&=~ \big[\,
	\one
	\xrightarrow{\coev_{U^*}}
	U^*U^{**}
	\xrightarrow{\id\otimes\delta_U^{-1}}
	U^*U 
	\, \big] \ .
\end{align}

Let now $A$ be a finite-dimensional ribbon quasi-Hopf algebra.
Combining \eqref{eq:drinfeld-iso-ddual} and \eqref{eq:RepA-ribbon}, we see that
in the category $\rep A$ the pivotal structure takes the form
\be
	\delta_U
	~=~ 
	\delta_{U}^{\vect}\circ (\ribbon^{-1}\sqs\act (-))
	~:~ U \to U^{**} \ .
\ee
The evaluation and coevaluation morphisms are (compare to \eqref{eq:ev-coev}), 
\be\label{eq:qHopf-tilde-evcoev}
\widetilde\ev_U: \; w\tensor\phi \mapsto \phi(S(\Salpha)\ribbon^{-1}\sqs . w)~~,
\quad\qquad \widetilde\coev_U: \; 1\mapsto \sum_i w^*_i \tensor (\sqs^{-1}\ribbon S(\Sbeta) . w_i) \ ,
\ee
where $\phi\in U^*$, $w\in U$, and $\{w_i\}$ is a basis of $U$ with dual basis $\{w^*_i\}$.

Using the ribbon structure, we can give a relation between the two variants $\sqs$ and $\tilde\sqs$ of the Drinfeld element. 
Namely,
	from $u_X \circ \theta_X = \tilde u_X \circ \theta_X^{-1}$ we get
$\sqs \ribbon^{-1} = \tilde\sqs \ribbon$ and combining this with \eqref{u-ribbon} gives
\begin{equation}\label{sqs-iv}
\tilde\sqs = S(\sqs^{-1}) \ .
\end{equation}
Applying $S$ to both sides and using \eqref{eq:S2-inner} and \eqref{eq:tilde-sqs} gives an explicit formula for $\sqs^{-1}$. An alternative expression is given in \cite[Thm.\,2.6]{Bulacu:2003}.

\section{Coends for quasi-triangular quasi-Hopf algebras} \label{sec:coend-repA}

We begin this section by describing explicitly the coend object $\coend$ in  $\rep \Ho$ for a quasi-triangular quasi-Hopf algebra $\Ho$. Then we give explicit expressions for its Hopf-algebra structure morphisms and Hopf pairing in terms of the defining data of  $A$.
Finally, we discuss an equivalent factorisability condition for $\rep \Ho$
and properties of integrals and internal characters.

All string diagrams in this section are taken in $\vect_\ok$, and we will drop the label $\raisebox{.1em}{\framebox{\scriptsize{$\vect_\ok$}}}$ from the diagrams.
	Recall that this means that duality maps in string diagrams refer to those of $\vect_{\ok}$, not those of $\rep A$.

\subsection{The coend $\coend$ in $\rep A$} \label{subsec:coend-repA}

To describe the coend, we will need to discuss the coadjoint representation of $A$, as well as an equivalent way of writing it. The {\em adjoint representation} $\rho^\mathrm{adj}_A : A \otimes A \to A$ of $A$ on itself is given by
\be
\rho^\mathrm{adj}_A ~=~ \ipic{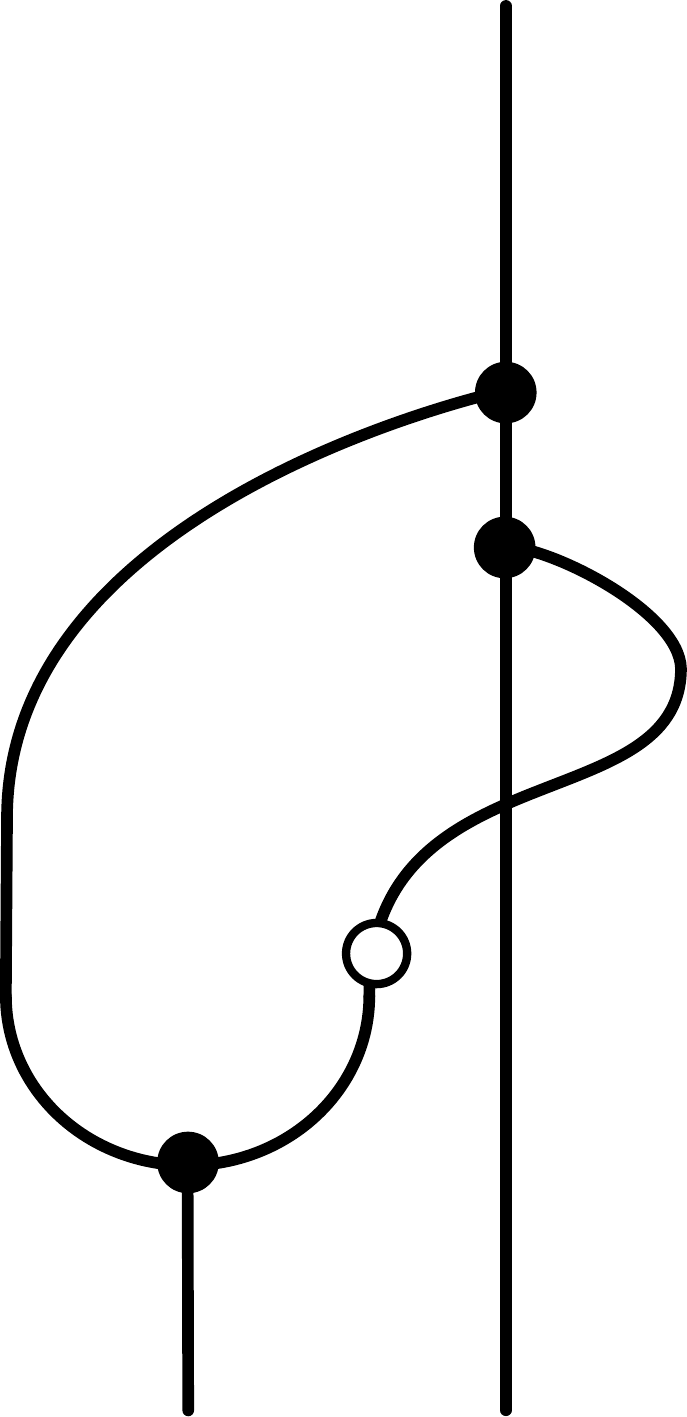}{0.133}
	 \put(-15,48){\scriptsize$\Ho$}
	 \put(-37,-54){\scriptsize$\Ho$} \put(-15,-54)  {\scriptsize$\Ho$} \gp
\ee
By definition, the dual $\rho^{\mathrm{adj}^*}_{A^*} : A \otimes A^* \to A^*$ of the adjoint representation is given by
\be
\rho^{\mathrm{adj}^*}_{A^*} ~=~ \ipic{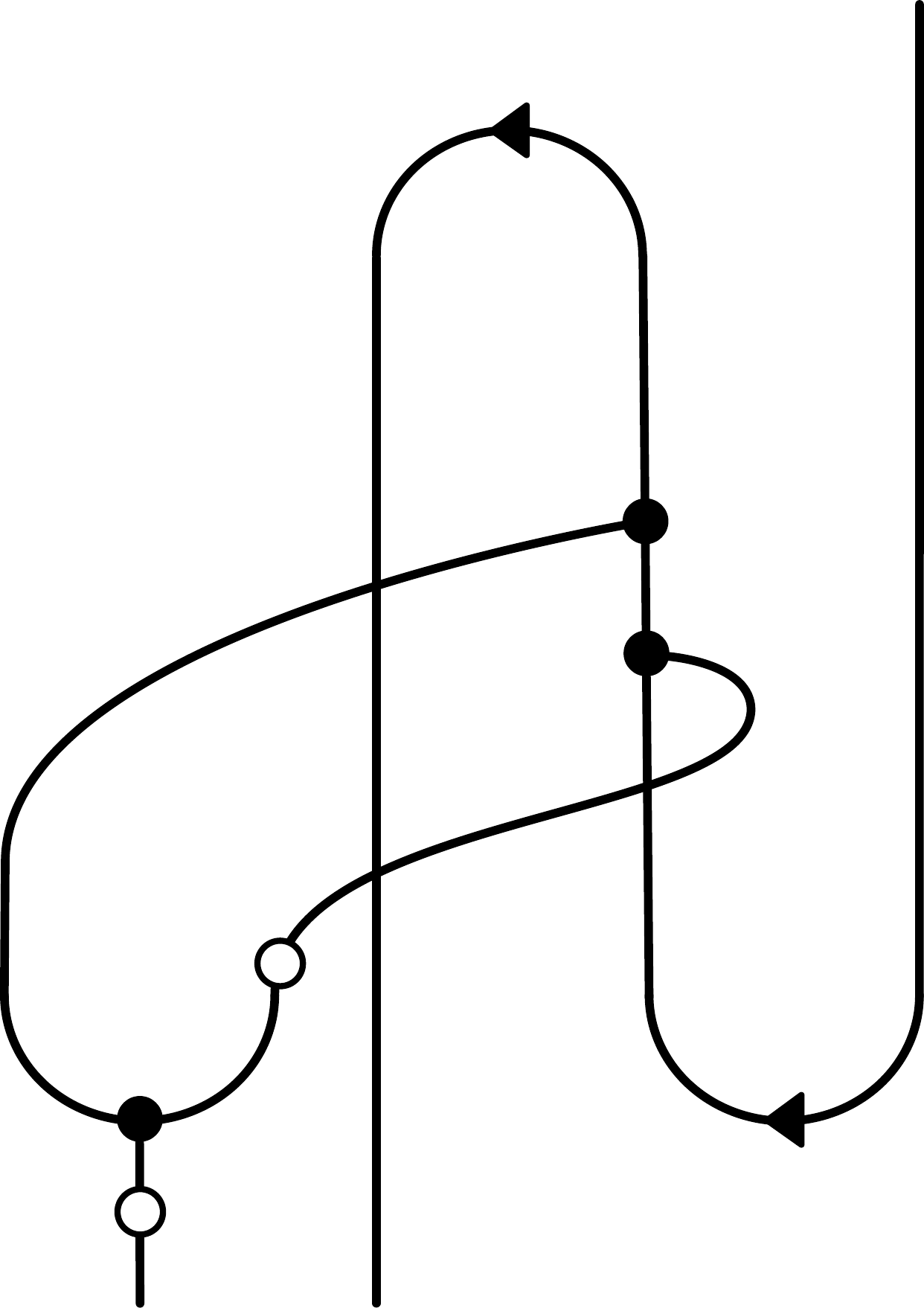}{0.13}
	 \put(-3,58){\scriptsize$\Ho^\ast$}
	 \put(-70,-64){\scriptsize$\Ho$} \put(-50,-64) {\scriptsize$\Ho^\ast$} \gp
\ee
We will show below that, as for Hopf algebras (see \cite{Lyubashenko:1994tm,Kerler:1996}), the coend $\coend$ in $\rep A$ can be taken to be the dual of the adjoint representation. However, again as for Hopf algebras \cite{Kerler:1996}, we find it convenient to work with the action $\coa$ on $A^*$ from Figure \ref{fig:coend-qHopf}, which we refer to as the {\em coadjoint representation}.

\newcommand\adjiso{E}
The dual of the adjoint representation and the coadjoint representation are isomorphic. Indeed, define the map $\adjiso : A \to A$ as
\be
	\adjiso(a)
	~=~ \sum_{(\Dt)} S^{-1}\big(\Dt'\,a\,S(\Dt'')\big)
\ee
where $\Dt$ is
defined in \eqref{def:F}. It is straightforward to verify from invertibility of $\Dt$ that $\adjiso$ is invertible, and from \eqref{eq:anti-coalg} that 
\be
	\adjiso^* \,:\, (A^*,\rho^{\mathrm{adj}^*}_{A^*}) \longrightarrow (A^*,\coa)
\ee
is an isomorphism of $A$-modules.

\begin{figure}
\eqpicnn{def:coend-repA} {285} {50} {\setlength\unitlength{.85pt}
   \put(-2,54)  {$ \coa\coloneqq $}
   \put(45,10)  {\scalebox{.13}{\includegraphics{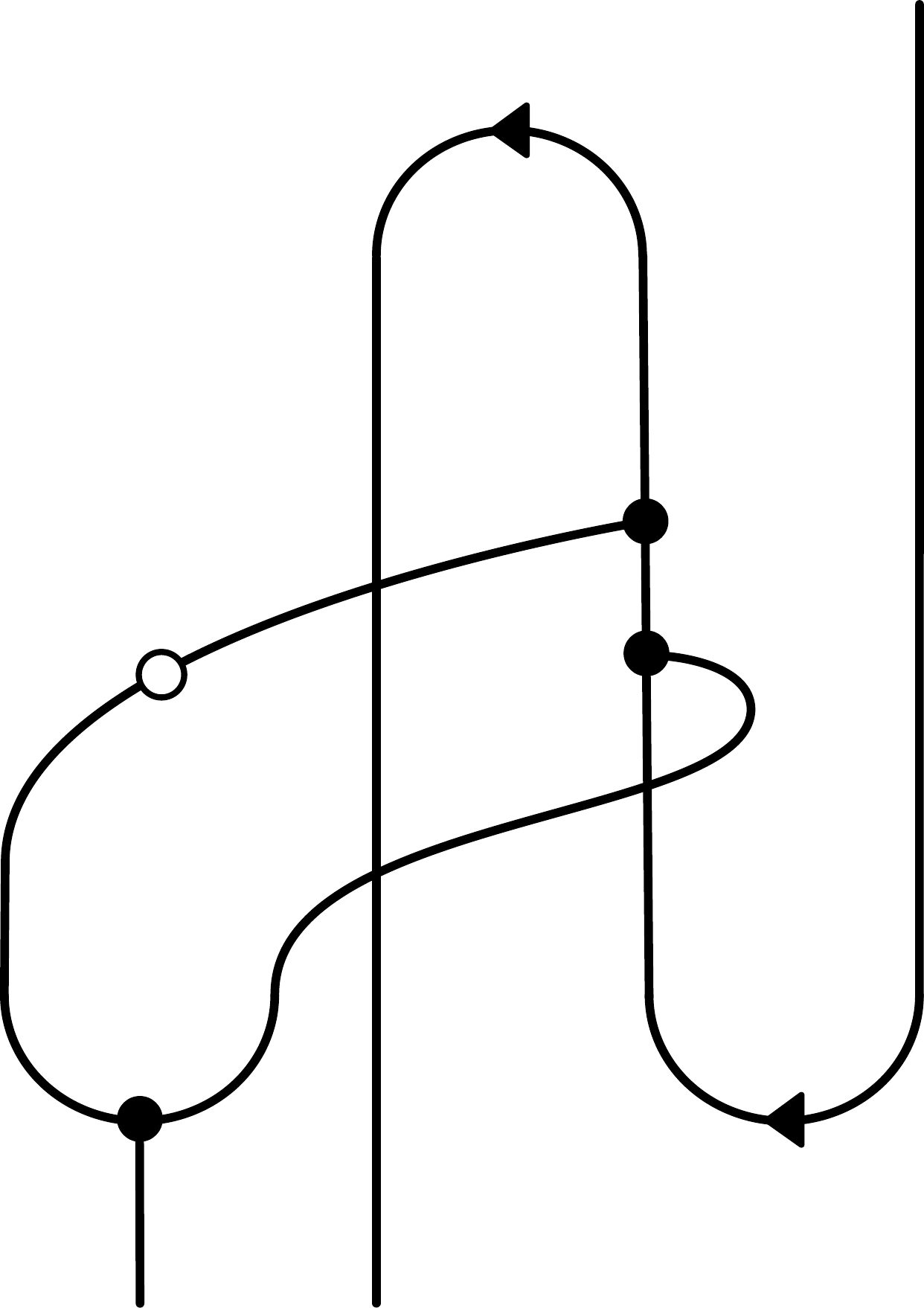}}}
	 \put(131,143){\scriptsize$\Ho^\ast$}
	 \put(53,0){\scriptsize$\Ho$} \put(77,0)  {\scriptsize$\Ho^\ast$}
   \put(208,54){$\iota_M\coloneqq$}
   \put(258,10) {\raisebox{0\height}{\includegraphics[scale=0.13]{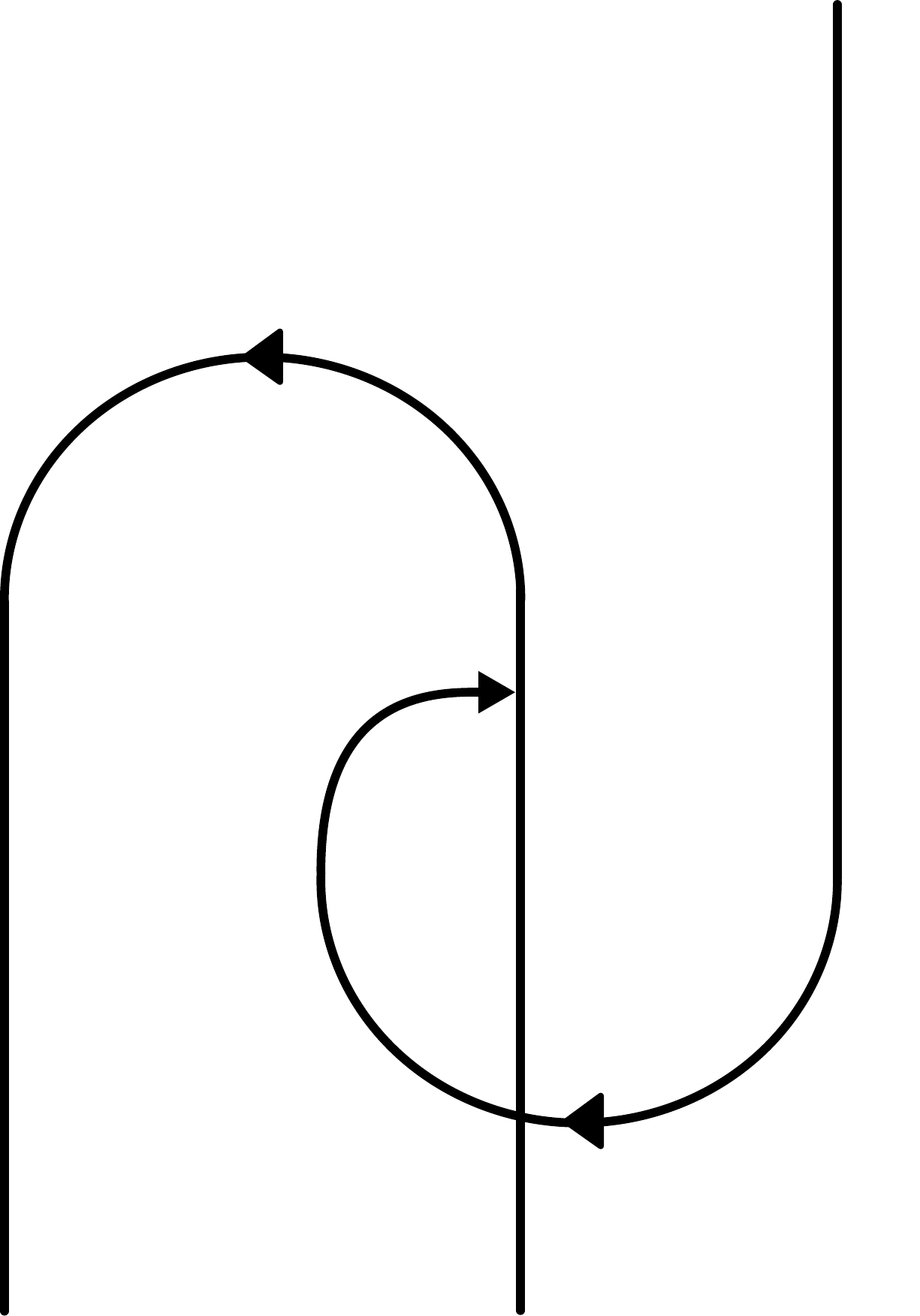}}}
	 \put(338,143){\scriptsize$\Ho^\ast$}
	 \put(250,00){\scriptsize$M^\ast$} \put(300,0)  {\scriptsize$M$}
	}
\caption{The coend object $\coend=\Ho^*$ for a quasi-Hopf algebra $\Ho$. The diagrams here are in $\vect_\ok$, i.e.\ the braidings are the usual flips of the vector spaces, etc.}
\label{fig:coend-qHopf}	
\end{figure}

We are now ready to show that $(A^*,\coa)$ can serve as the coend in $\rep A$.

\begin{proposition}\label{prop:coend-RepH}
Let $\Ho$ be a finite-dimensional quasi-Hopf algebra over a field $\ok$. 
\begin{enumerate}\setlength{\leftskip}{-1.5em}
\item
The coend object~\eqref{eq:coend} in $\rep \Ho$ can be chosen to be the coadjoint representation $\coend = (\Ho^\ast,\coa)$, together with the dinatural family
$\iota\equiv(\iota_M\colon M^\ast\tensor M \to \Ho^\ast)$ given by (see Figure~\ref{fig:coend-qHopf})
\begin{equation*}
\iota _M:\quad   \varphi\tensor m \; \mapsto\; \bigl(a \mapsto \varphi(a. m)\bigr)\ , \qquad m\in M, \; \varphi\in M^\ast, \; a\in A \ .
\end{equation*}
\item
The unique morphism $g: \coend\to B$ from~\eqref{eq:din-diag} is, for a dinatural transformation $(B,\phi)$,
$$
g = \left[\coend= \Ho^\ast \xrightarrow{\sim} \Ho^\ast\tensor\one \xrightarrow{\id\tensor \eta_\Ho} \Ho^\ast\tensor\Ho \xrightarrow{\phi_\Ho}  B \right]\ ,
$$
where $\eta_\Ho$ is the unit morphism of $A$ and $\phi_\Ho$ is  evaluated on the regular representation.
\end{enumerate}
\end{proposition}

In the case that $\Ho$ is a Hopf algebra, this proposition was proven in~\cite[Sec.\,3.3]{Lyubashenko:1994tm} and~\cite[Lem.\,3]{Kerler:1996}. The proof for the quasi-Hopf case is very similar and we reproduce it here for completeness. 

\begin{proof}[Proof of Proposition~\ref{prop:coend-RepH}]
We show that $(\coend,\iota)$ satisfies the universal property of the coend. As part of the argument, we show that $g$ is as stated in part (2).

\noindent
{\em $\bullet$ $\iota_M$ is an $A$-module intertwiner:}
We use the identity 
\begin{align}
	\iota_M (a\act(\varphi\ot m))= \big( h\mapsto \sum_{(a)}\varphi\big(S(a')ha''.m\big) \big) 
	\qquad , \quad h \in \Ho \ .
\end{align}
Indeed, we have 
\be\label{proof:diag-coend-repA} 
   \ipic{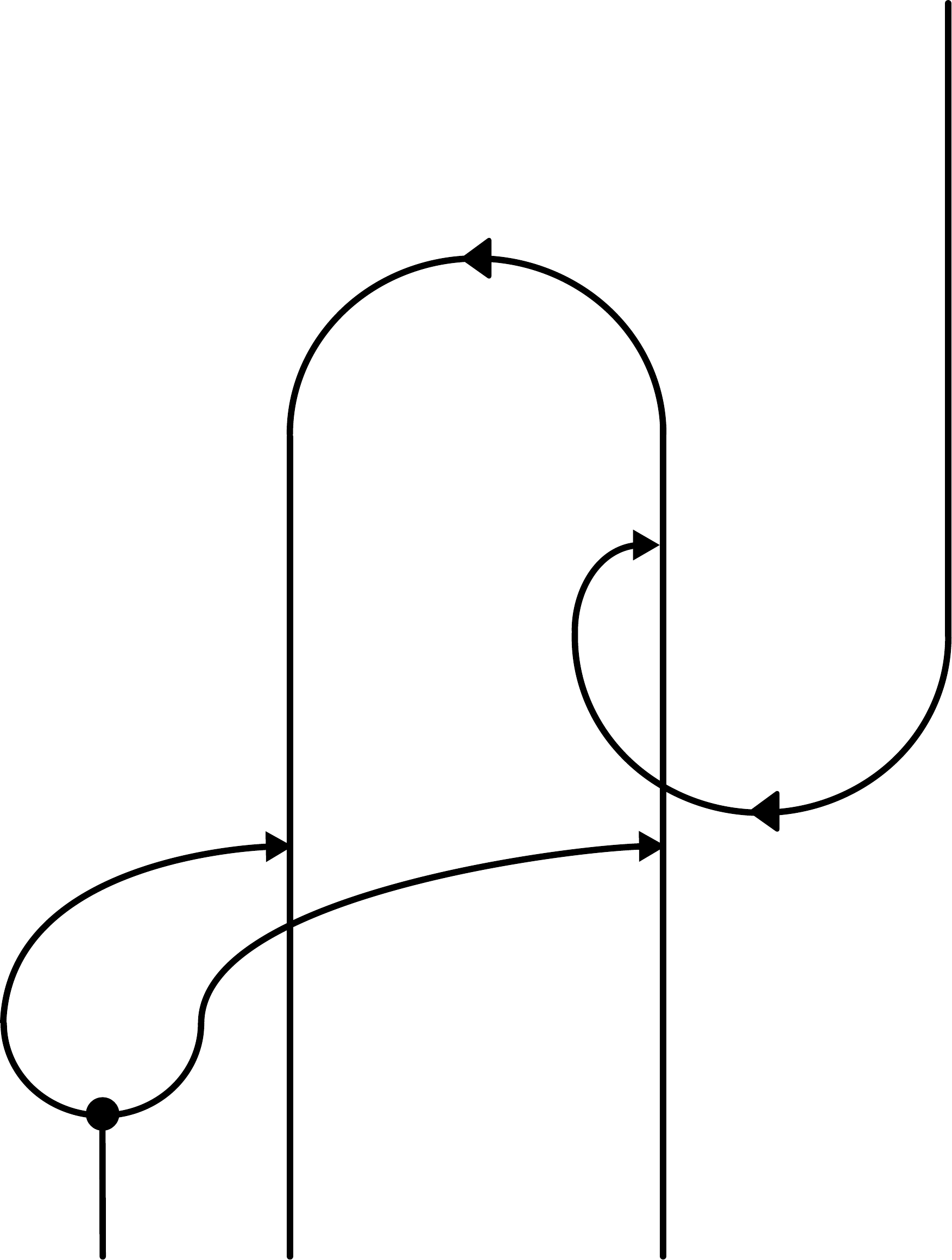}{.13}
	 \put(-3,76){\scriptsize$\Ho^\ast$}
	 \put(-103,-81){\scriptsize$\Ho$} \put(-82,-81){\scriptsize$M^\ast$} \put(-39,-81)  {\scriptsize$M$}
   \quad = \quad
   \ipic{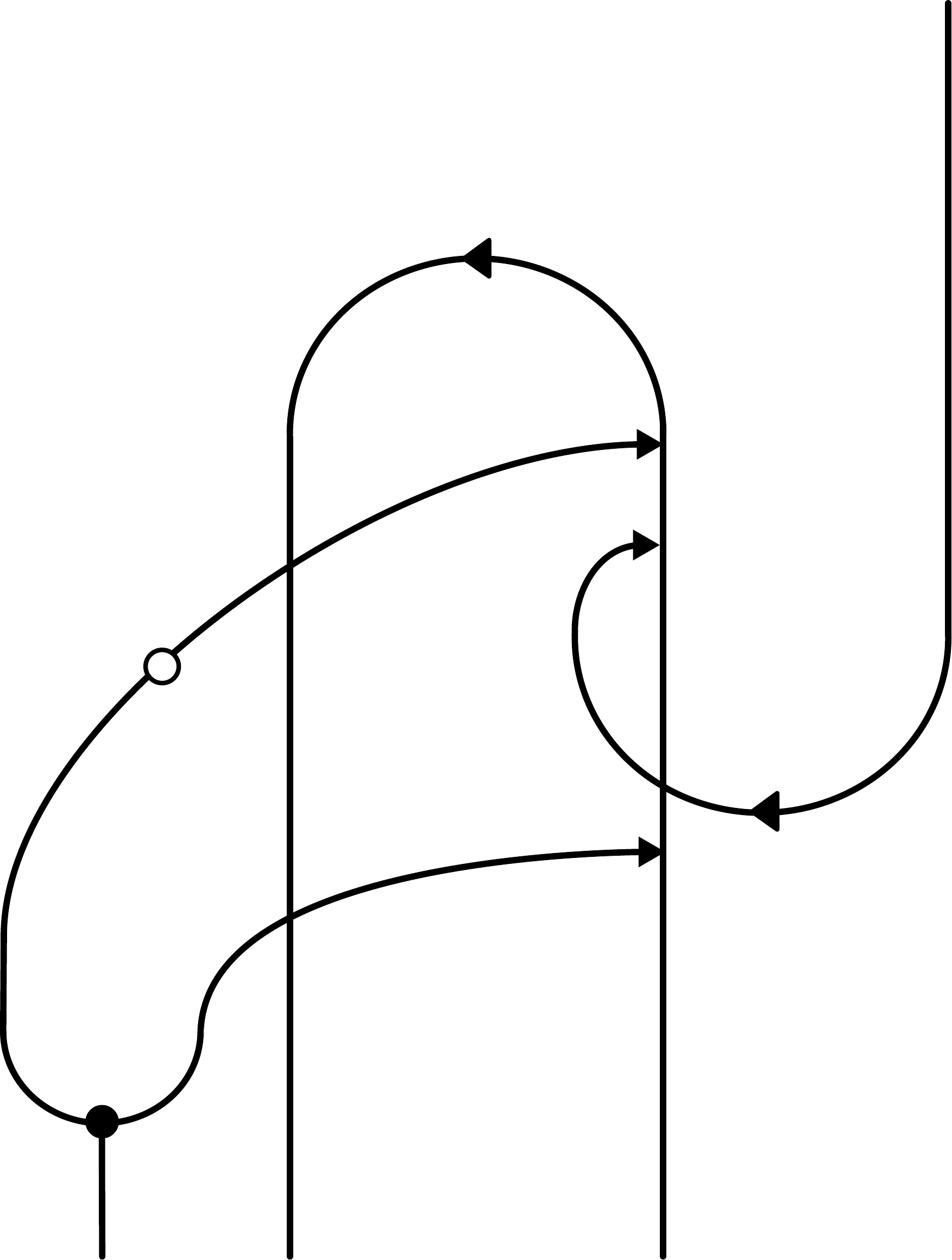}{.13}
	 \put(-3,76){\scriptsize$\Ho^\ast$}
	 \put(-103,-81){\scriptsize$\Ho$} \put(-82,-81){\scriptsize$M^\ast$} \put(-39,-81)  {\scriptsize$M$}
	 \quad = \quad
	 \ipic{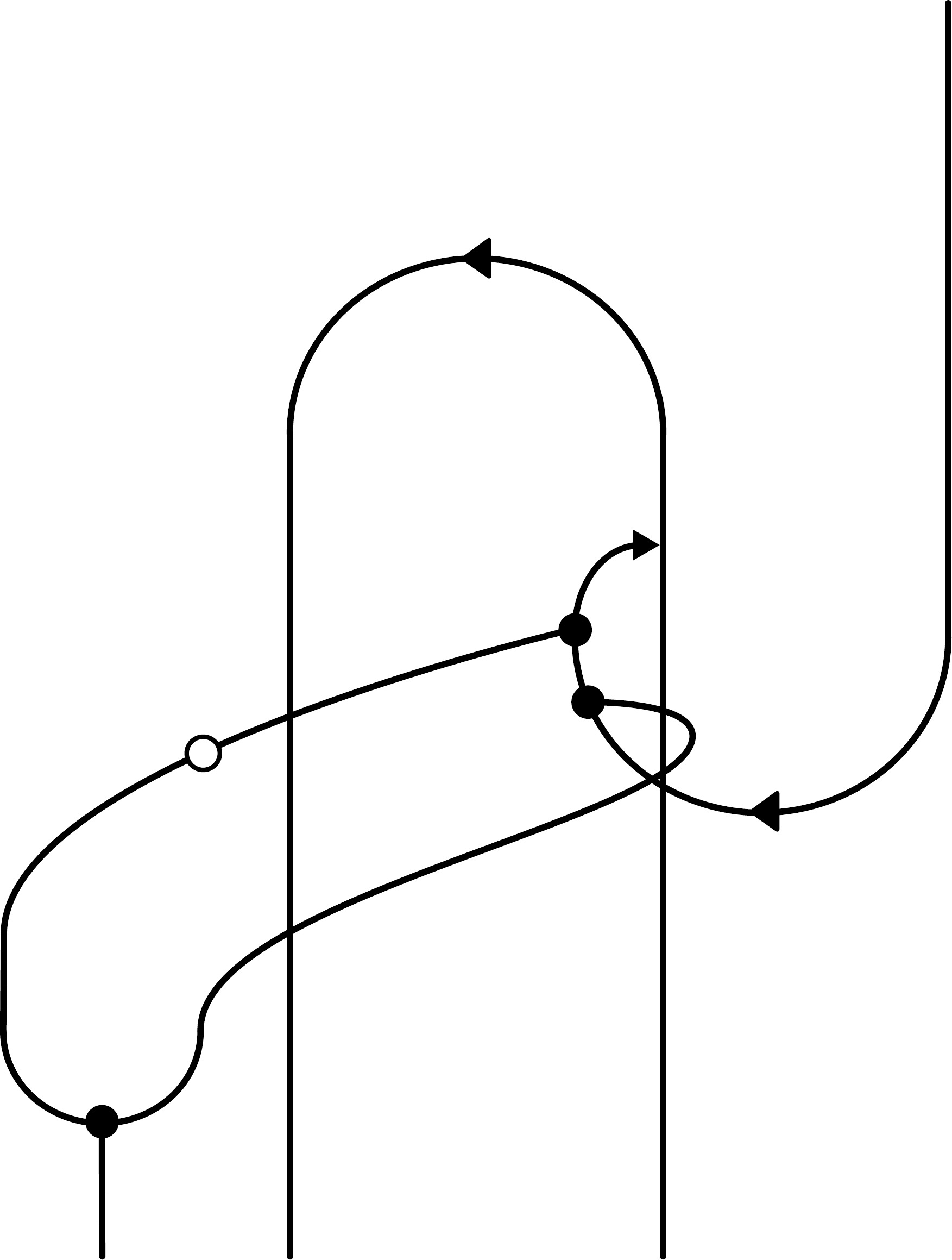}{.13}
	 \put(-3,76){\scriptsize$\Ho^\ast$}
	 \put(-103,-81){\scriptsize$\Ho$} \put(-82,-81){\scriptsize$M^\ast$} \put(-39,-81)  {\scriptsize$M$} \gp
\ee
Then the intertwiner property follows from 
$\sum_{(a)}\varphi(S(a')(-)a''.m)
=\coa(a\ot\varphi(-\,.m))=
\coa(a\ot\iota_M(\varphi\ot m))$.
Pictorially, it can be seen easily by adding a zig-zag at the last string of the last diagram in \eqref{proof:diag-coend-repA} and moving the coadjoint action of $a$ on the other end.

\noindent
{\em $\bullet$ $\iota_M$ is dinatural:}
We need to show that for all $A$-module maps $f : M \to N$ we have $\iota_N \circ (\id \otimes f) = \iota_M \circ (f^* \otimes \id)$ as maps $N^* \otimes M \to A^*$. Evaluating on $\varphi \in N^*$, $m \in M$ and $a \in A$ gives
\begin{align}
\big[\iota_N \circ (\id \otimes f)(\varphi \otimes m)\big](a)
&=
\varphi(a.f(m))
=
\varphi(f(a.m))
=
(f^* (\varphi))
(a.m)
\nonumber\\
&=
\big[\iota_M \circ (f^* \otimes \id)(\varphi \otimes m)\big](a) \ .
\end{align}

\noindent
{\em $\bullet$ $g$ from part (2) is an $A$-module map:}
For $x \in A$ denote by $L_x,R_x : A \to A$ the left and right multiplication by $x$: $L_x(a) = xa$ and $R_x(a) = ax$. Note that $R_x$ is an $A$-module intertwiner of $A$ seen as a left module over itself.
Since $\phi$ is dinatural, we have
\be\label{proof:diag-coend-repA_phiARx}
	\phi_A \circ (\id \otimes R_x) = \phi_A \circ ((R_x)^* \otimes \id) \ .
\ee
The coadjoint action of $a \in A$ is $\coa(a \otimes -) = \big(\sum_{(a)} L_{S(a')}\circ R_{a''}\big)^* : A^* \to A^*$, see Figure \ref{fig:coend-qHopf}. We compute, for $\varphi \in A^*$,
\begin{align}
	g(a.\varphi)
	&=
	\sum_{(a)} g(R_{a''}^*\circ L_{S(a')}^*(\varphi))
	\overset{\text{def.\,$g$}}=
	\sum_{(a)} 
	\phi_A\big( (R_{a''}^*\circ L_{S(a')}^*(\varphi)) \otimes 1_A \big)
\\ \nonumber &
	\overset{\eqref{proof:diag-coend-repA_phiARx}}=
	\sum_{(a)} 
	\phi_A\big( (L_{S(a')}^*(\varphi)) \otimes R_{a''}(1_A) \big)
	=
	\sum_{(a)} 
	\phi_A\big( (L_{S(a')}^* \otimes L_{a''})(\varphi \otimes 1_A) \big)
\\ \nonumber &
	\overset{(*)}=
	a.
	\phi_A\big( \varphi \otimes 1_A \big) 
	\overset{\text{def.\,$g$}}= a.g(\varphi) \ .
\end{align}
Here, (*) amounts to the observation that $\sum_{(a)} L_{S(a')}^* \otimes L_{a''}$ gives the action of $a$ on $A^* \otimes A$, where $A$ is 
	the left regular module
and $A^*$ is the corresponding dual module, together with the fact that $\phi_A$ is an $A$-module intertwiner.

\noindent
{\em $\bullet$ $g$ makes \eqref{eq:din-diag} commute:}
We need to show that for every $A$-module $M$, $g \circ \iota_M = \phi_M$.
Consider the map $R_m : A \to M$, $R_m(a) = a.m$. This is an $A$-module intertwiner, for $A$ the left regular
 module over itself.
We compute, for $\varphi \in M^*$ and $m \in M$,
\begin{align}
g \circ \iota_M(\varphi \otimes m)
&= \phi_A\big( \iota_M(\varphi \otimes m) \otimes 1_A \big)
= \phi_A\big( \varphi(-.m) \otimes 1_A \big)
\\ \nonumber &
= \phi_A\big( 
(R_m)^*  (\varphi)
 \otimes 1_A \big)
\overset{(*)}= \phi_M\big( \varphi \otimes (R_m(1_A))\big)
= \phi_M\big( \varphi \otimes m \big) \ ,
\end{align}
where (*) is dinaturality of $\phi$.

\noindent
{\em $\bullet$ $g$ is unique:}
Let $h : \coend \to B$ be an $A$-module map satisfying $h \circ \iota_M = \phi_M$ for all $A$-modules~$M$. We need to show that $h=g$.
But this is immediate if one chooses $M=A$. Indeed, for all $\varphi \in A^*$ we have $\iota_A(\varphi \otimes 1_A) = \varphi$ and so
\be
h(\varphi) = h \circ \iota_A(\varphi \otimes 1_A)
= \phi_A(\varphi \otimes 1_A) = g(\varphi) \ ,
\ee
by definition of $g$.
\end{proof}

We make a similar statement for an end $\Gamma$, recall the discussion in Section~\ref{sec:fact-end-coend}.
	This statement is proven in \cite[Lem.\,5.4]{Sakalos:2014} and can also be verified by arguments analogous to those in Proposition \ref{prop:coend-RepH}.

\begin{proposition}\label{prop:end-RepH}
Let $\Ho$ be a finite-dimensional  quasi-Hopf algebra over a field $\ok$. 
The end object $\Gamma$ in $\rep \Ho$ can be chosen to be the adjoint representation $\Gamma = (\Ho,\adj)$, together with the dinatural family
$j\equiv(j_M\colon  \Ho\to M\tensor M^\ast)$ given by 
\begin{equation}
j_M:\quad   a \; \mapsto\; \sum_i(a. m_i) \tensor m^*_i ~=~  
\ipic{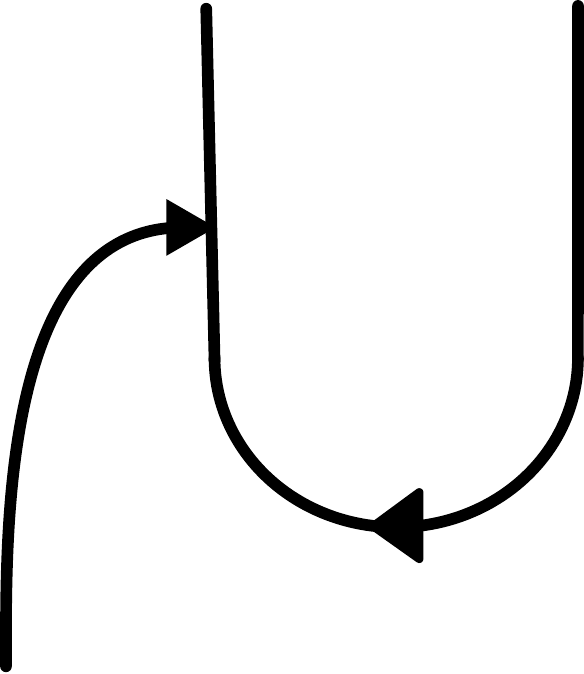}{.18} 
\put(-55,-38){\scriptsize $A$} \put(-37,31){\scriptsize $M$} \put(-5,31){\scriptsize $M^*$}
\ , 
\end{equation}
where $m_i$ is a basis in $M$ and $m^*_i$ is the dual basis.
\end{proposition}

\subsection{Hopf structure and Hopf pairing on $\coend$} \label{sec:Hopf-structure-repA}

In this section we present explicit expressions for the structure maps 
	 \eqref{eq:mult-L}--\eqref{eq:antipode-L} and \eqref{eq:omega-L}
which define a Hopf structure and a Hopf pairing on
	the universal Hopf algebra $\coend = A^*$ in $\rep A$.

\newcommand{\mm}{\hat \mu_\coend}
\newcommand{\co}{\hat \Delta_\coend}
\newcommand{\un}{\hat \eta_\coend}
\newcommand{\coun}{\hat \eps_\coend}
\newcommand{\anip}{\hat S_\coend}
\newcommand{\Hp}{\hat \omega_\coend}
\newcommand{\Co}{D}

To work with elements rather than with functionals, we dualise all structure morphisms. 
We will use the notation $\langle-,-\rangle$ for the contraction of an element in $A^*$ with an element in $A$, and of an element in $A^* \otimes A^*$ with an element in $A \otimes A$:
\begin{align}
	\langle-,-\rangle &: A^* \otimes A \to \ok \ ,
	&&\langle \varphi,a \rangle = \varphi(a) \ ,
	\\ \nonumber
	\langle-,-\rangle &: A^* \otimes A^* \otimes A\otimes A \to \ok \ ,
	&&\langle \varphi \otimes\psi,a \otimes b \rangle = \varphi(b)\psi(a)  \ .
	\hspace*{4em}
\end{align}
Note the order of arguments in the second contraction. The convention is such that $\langle-,-\rangle = \tilde\gamma_{A,A}^{\vect_\ok}$, i.e.\ \eqref{eq:gammaVU} applied to the category of vector spaces.
In terms of this bracket notation, we define $\mm : A \to A \otimes A$, etc., as follows. For all $a,b \in A$ and $f,g\in A^*$,
\begin{align} \label{eq:coend:structuremaps}
\big\langle \muc (f\tensor g) \,,\, a \big\rangle
 &~=~ \big\langle f\tensor g \,,\, \mm(a) \big\rangle
 &&, \quad \mm : A \to A \otimes A \ ,
 \\ \nonumber 
\big\langle \Delta_\coend (f) \,,\, a \otimes b \big\rangle
 &~=~ \big\langle f \,,\, \co(a \otimes b) \big\rangle 
 &&, \quad \co : A \otimes A \to A\ ,
 \\ \nonumber 
\eta_\coend (1) &~=~ \big(\, a\mapsto \un(a) \, \big) 
 &&, \quad \un : A \to \ok \ ,
 \\ \nonumber
\eps_\coend (f) &~=~ f(\coun) 
 &&, \quad \coun \in A  \ ,
 \\ \nonumber
\big\langle S_\coend(f) \,,\, a \big\rangle &~=~ 
\big\langle f \,,\, \anip(a) \big\rangle  
 &&, \quad \anip : A \to A \ ,
 \\ \nonumber
\omega_\coend (f\tensor g) &~=~  \big\langle f\tensor g \,,\, \Hp \big\rangle
 &&, \quad \Hp \in A \otimes A \ .
\end{align}

\begin{theorem}\label{thm:explicit-coend-qHopf}
Let $A$ be a finite-dimensional quasi-triangular quasi-Hopf algebra over a field~$\ok$. The Hopf algebra structure and Hopf pairing from Theorem \ref{thm:coend-Hopf+pairing} applied to the universal Hopf algebra in $\rep A$ from Proposition \ref{prop:coend-RepH} are given by the maps in \eqref{eq:coend:structuremaps} with
\begin{align}
\label{eq:qHopf-coend-dualstructuremaps}
\mm(a) &~=~ \sum_{(\as),(\Psi),(\tilde{\Psi}),(R)} 
\left[S(\as_2\Psi_1
R_2'\tilde{\Psi}_3'
)\tensor S(\as_1\tilde{\Psi}_1)\right] \cdot \Dt
\\[-.8em] \nonumber
& \hspace{8em}
 \cdot \Delta(a \as_3) \cdot \left[(\Psi_2
R_2''\tilde{\Psi}_3'')
\tensor(\Psi_3R_1\tilde{\Psi}_2)\right] \gc
 \\ \nonumber
\co(a\ot b)&~=~ \sum_{(\Co)} S(D_1)bD_2 S(D_3)aD_4 \gcg 
 \\ \nonumber
	\un(a)&~=~ \eps(\Sbeta\,a) \gc 
 \\ \nonumber
\coun&~=~\Salpha \gc 
 \\ \nonumber
\anip(a)&~=~\sum_{(R)} S(aR_1)\tilde\sqs R_2 \ ,
 \\ \nonumber
	\Hp&~=~ \sum_{(W)} S(W_3) W_4 \ot S(W_1) W_2
\ .
\end{align}
In these expressions, $\Psi=\as^{-1}$, $\tilde{\Psi}$ is another copy of $\as^{-1}$, $\Dt$ is the Drinfeld twist from \eqref{def:F}, 
$\tilde\sqs$ was given in \eqref{eq:tilde-sqs},
 the elements $D,W \in A^{\otimes 4}$ are defined as
\begin{align}
\label{eq:coend-maps-elementsDW}  
\Co&~=~(\id\tensor\id\tensor\Delta)(\as)\cdot (\one\tensor\as^{-1})
\cdot
(\one\tensor\Sbeta\tensor\one\tensor\one) \ ,
\\ \nonumber
W&~=~(\one\tensor\Salpha\tensor\one\tensor\Salpha)\cdot
(\one\tensor\as^{-1})\cdot(\one\tensor
 M
\tensor\one)\cdot(\one\tensor\as)
\cdot(\id\tensor\id\tensor\Delta)(\as^{-1}) \ ,
\end{align}
and $M$ was defined in \eqref{eq:monodromy-el}.
\end{theorem}

\begin{proof}
Let $M,N\in\rep A$ and $m\in M$, $n\in N$, $\varphi\in M^\ast$, $\psi\in N^\ast$. By \eqref{eq:mult-L} the multiplication on $\coend$ is determined by
the equality $X = Y$ with
\begin{align}\label{eq:proof-muc-1}
X &= \muc\circ (\iota_M\ot\iota_N)(\varphi\ot m \ot \psi \ot n)  \ ,
\\ \nonumber
Y &= 
	\iota_{N \otimes M}
	\circ(\gamma_{N,M}\ot\id)\Big((\id\ot\id\ot\Delta)(\Phi) 
	\cdot (\one\tensor\Phi^{-1}) 
\\ \nonumber	&
	\hspace{2em}  . \, (\id\ot \tau_{M,N^\ast N})\Big\{
	\big(\one\ot(\id\ot\Delta)(R)\big) 
	\cdot (\id\ot\id\ot\Delta)(\Phi^{-1})\, . \,\varphi\ot m \ot \psi \ot n \Big\}\Big)  \ .
\end{align}
By abbreviating $\Psi,\tilde{\Psi}=\as^{-1}$
we get 
\begin{align}\label{eq:proof-muc-2}
Y &= \hspace{-.8em}
	\sum_{(\as),(\Psi),(\tilde{\Psi}),(R)} \hspace{-.8em}
	\iota_{N \otimes M}
	\circ(\gamma_{N,M}\ot\id)\Big(\Phi_1\tilde\Psi_1\ot \Phi_2\Psi_1 
	R_2'\tilde{\Psi}_3'\ot\Phi_3'\Psi_2
	R_2''\tilde{\Psi}_3''
	\ot\Phi_3''\Psi_3R_1\tilde\Psi_2
 \\[-1em] \nonumber &	\hspace{26em}
	\, . \,\varphi\ot\psi\ot n\ot m\Big) \ .
\end{align}
Note that for $B\in A^{\ot 4}$ we have
\begin{align}
\begin{split} 
	\iota_{N \otimes M}
&\circ(\gamma_{N,M}\ot\id)(B\, . \,\varphi\ot\psi\ot n\ot m) \\
=&\Big(
a\mapsto \sum_{(B)}(\psi\ot\varphi)
\big(
\big(S(B_2)\ot S(B_1)\big)\cdot \Dt \cdot \Delta(
a)\cdot (B_3\ot B_4) \act n\ot m\big)\Big) \gp
\end{split}
\end{align}
Indeed, with \eqref{eq:gamma} we get
\be\label{iota-gamma}
   \ipic{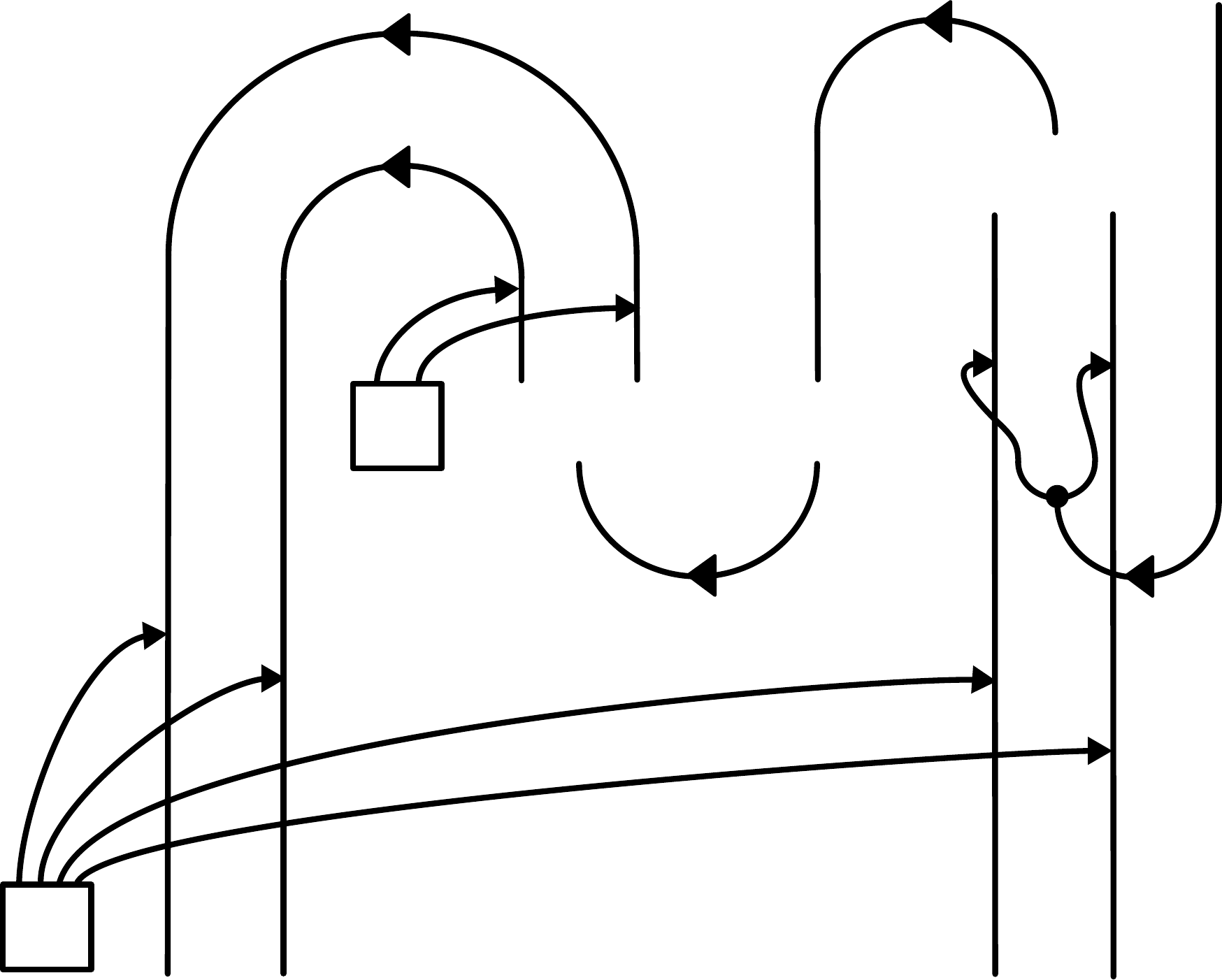}{.25}
	 \put(-188,-92){{\scriptsize $M^\ast$}}\put(-167,-92){{\scriptsize $N^\ast$}}
	 \put(-207,-79){\scriptsize$B$} \put(-146,8){\scriptsize$\Dt$} \put(-5,87){\scriptsize$A^\ast$}
	 \put(-125,9){{\scriptsize $N$}}\put(-107,9){{\scriptsize $M$}} \put(-82,9){{\scriptsize $(NM)^\ast$}}
	 \put(-38,52){{\scriptsize $NM$}} \put(-44,-92){{\scriptsize $N$}}\put(-23,-92){{\scriptsize $M$}}
	 \quad = \quad
	 \ipic{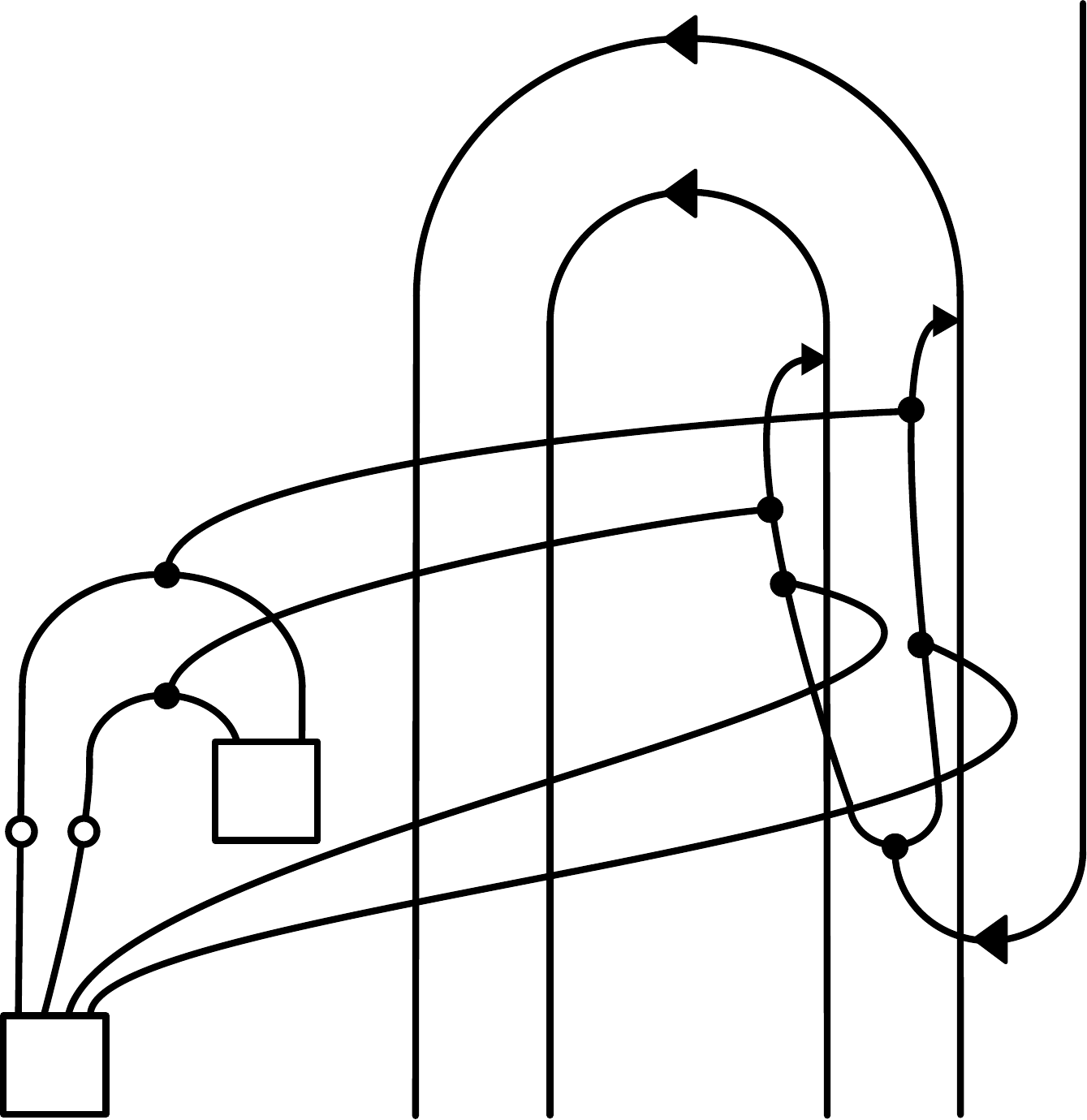}{.25}
	 \put(-157,-78){\scriptsize$B$} \put(-125,-37){\scriptsize$\Dt$} \put(-5,87){\scriptsize$A^\ast$}
	 \put(-105,-92){{\scriptsize $M^\ast$}}\put(-85,-92){{\scriptsize $N^\ast$}} \put(-43,-92){{\scriptsize $N$}}\put(-23,-92){{\scriptsize $M$}}
\ee
Thus \eqref{eq:proof-muc-2} becomes, for $a \in A$, 
\begin{align}
Y(a) &= \hspace{-.8em}
	\sum_{(\as),(\Psi),(\tilde{\Psi}),(R)} \hspace{-.8em}
	(\psi\ot\varphi)\big\{\big(S(\Phi_2\Psi_1
	R_2'\tilde{\Psi}_3')\ot S(\Phi_1\tilde\Psi_1) \big)
	\cdot \Dt \cdot \Delta(a\Phi_3)
\\[-1em]	\nonumber & \hspace{15em}
	\cdot
	\big(\Psi_2
	R_2''\tilde{\Psi}_3''
	\ot\Psi_3R_1\tilde\Psi_2\big) \act n\ot m\big\}\Big) 
\\	\nonumber
& = \big\langle f\tensor g \,,\, \mm(a) \big\rangle \ ,
\end{align}
where $f,g \in A^*$ are given by $f = \varphi((-).m)$ and $g = \psi((-).n)$. Since, with the same notation,
$(\iota_M\ot\iota_N)(\varphi\ot m \ot \psi \ot n) = f \otimes g$, from \eqref{eq:proof-muc-1} we get $X(a) = \big\langle \muc (f\tensor g) \,,\, a \big\rangle$. Altogether,
\be
\big\langle \muc (f\tensor g) \,,\, a \big\rangle
=
X(a)
=
Y(a)
=
\big\langle f\tensor g \,,\, \mm(a) \big\rangle
\ee
which proves the statement for $\muc$. 

The statements for the other maps can be shown similarly.
\end{proof}

\begin{remark}\label{rem:coend-str-Hopf}
We note that in the case $A$ is a quasi-triangular Hopf algebra, the universal Hopf algebra $\coend$ in $\rep A$ has the following structural maps on $\coend=A^*$,
using~\eqref{eq:coend:structuremaps} and applying Theorem~\ref{thm:explicit-coend-qHopf}:
\begin{gather}
\label{eq:qHopf-coend-dualstructuremaps_hopf}
\mm(a) ~=~ \sum_{(a),(R)} 
S(R_2')a' R_2''\tensor a'' R_1 \gc \qquad
\co(a\ot b)~=~ b \cdot a \gcg \qquad
 \\ \nonumber
	\un(a)~=~ \eps(a) \gc  \qquad
\coun ~=~ \one \gc 
 \\ \nonumber
\anip(a)~=~ \sum_{(R)} S(\sqs^{-1} \, a\, R_1) R_2 \  , \qquad
	\Hp ~=~ \sum_{(M)} S(M_2) \ot M_1 \ ,
\end{gather}
where $\sqs = \sum_{(R)} S(R_2)R_1$, see \eqref{sqs}.
The structure maps for the universal Hopf algebra have also been explicitly computed in \cite[Sec.\,4]{Lyubashenko:1994ma}, \cite[Sec.\,3.4]{Lyubashenko:1994tm} (in different conventions) and in \cite[Lem.\,4.4]{Vi}.
The structure maps for $H$ in \cite[Lem.\,4.4]{Vi} are precisely those for $A$ above with the opposite $\mm$ and $\co$.

\end{remark}

\subsection{Non-degeneracy of the Hopf pairing}\label{sec:non-deg-hopf-pair}

Recall from Definitions \ref{def:fact-cat} and \ref{def:fact-qHopf} that 
	a finite-dimensional quasi-triangular quasi-Hopf algebra $A$ over a field $\ok$
is called factorisable if $\omega_\coend$ is non-degenerate. Non-degeneracy of $\omega_\coend$ in turn by definition means that the morphism $D_\coend : \coend \to \coend^*$ from \eqref{eq:Omega} is invertible.

In the quasi-Hopf case, 
$\Dm$ 
is an $A$-module map $A^* \to A^{**}$. 
We will describe 
$\Dm$ in terms of an element 
$\hDm \in A \otimes A$ as, for $\varphi \in A^*$,
\be
	\Dm(\varphi)
	=
	\delta^{\vect}_A \circ \big( (\id \otimes \varphi)
	(\hDm) \big) \ .
\ee
A short computation shows that in terms of $\hat\omega_\coend$ as given in Theorem \ref{thm:explicit-coend-qHopf} we have
\be
	\hDm	=	\sum_{(X), (\hat\omega_\coend)}	 S(X_2') \, (\hat\omega_\coend)_1 X''_2	\,\otimes\,
	S(X'_1) \, (\hat\omega_\coend)_2 \, X''_1\ 
\ee
with 
$X$ as in~\eqref{eq:Rem-X}.
Substituting the explicit expression for $\hat \omega_\coend$ from Theorem~\ref{thm:explicit-coend-qHopf} gives the expression announced in Remark \ref{rem:fact-qHopf-def}\,(1).

\newcommand{\BTmap}{\mathcal{Q}^{\mathrm{BT}}}
\newcommand{\BTM}{\mathcal{M}^{\mathrm{BT}}}

\newcommand{\Acat}{\underline{A}}
\subsection{Equivalent factorisability condition}\label{sec:equiv-fact-cond}

In \cite{[BT]} the authors also introduce a notion of factorisability for 
quasi-triangular quasi-Hopf algebras. In this section we recall the definition in \cite{[BT]} and show that it is equivalent to Definition \ref{def:fact-qHopf}. 

Let $A$ be a finite-dimensional quasi-triangular quasi-Hopf algebra and consider the linear map
$\BTmap: A^*\to A$ defined by 
	(see \cite[Prop.\,2.2\,(i)]{[BT]} but note that their $\as$ is our $\as^{-1}$)
\begin{equation}\label{eq:BTmap}
\BTmap:\; \phi \mapsto (\phi\tensor \id)(\BTM) \ .
\end{equation}
The element $\BTM\in A\tensor A$ is defined as 
\begin{equation}\label{eq:BTM}
\BTM  ~= \hspace{-1em}\sum_{(R), (\tilde{R}), (\as), (\pRel),  (\qLel)}
\hspace{-1.5em}
 \qLel_1 (\as^{-1})_1 R_2 \tilde{R}_1\, \pRel_1 \tensor \qLel_2' (\as^{-1})_2 R_1 \tilde{R}_2\, \pRel_2 S\bigl(\qLel_2'' (\as^{-1})_3\bigr)\ ,
\end{equation}
where $\tilde{R}$ is a copy of $R$,	and 
\be
\pRel = \sum_{(\as)}\as_1\tensor \as_2 \Sbeta S(\as_3)\ , \qquad
\qLel = \sum_{(\as)} S(\as_1)\Salpha \as_2\tensor \as_3\ .
\ee

Recall the end $\Gamma=(A, j)$ from Proposition~\ref{prop:end-RepH} and the map $\DD$ defined by~\eqref{eq:T-DD}.

\begin{lemma}\label{lem:BTmap-DD}
$\BTmap = \DD$.
\end{lemma}
\begin{proof}
	It is enough to show that for all $X,Y \in \rep A$ we have $j_Y \circ \BTmap\circ\iota_X = j_Y\circ \DD\circ \iota_X$. By definition, $j_Y\circ \DD\circ \iota_X = \T_{X,Y}$, so that it remains to show $j_Y \circ \BTmap\circ\iota_X = \T_{X,Y}$. We will verify this by computing
$\T_{X,Y}$ in~\eqref{eq:T-tangle} explicitly. Let 
\be\label{eq:Q-el}
Q = (\one\tensor \Salpha\tensor \one\tensor \one) \cdot \left[ (\id\tensor \id \tensor \Delta)\as\right] \cdot \bigl(\id\tensor (\as^{-1}\cdot(M\tensor\one)\cdot\as)\bigr) \cdot
 (\one\tensor \one\tensor \Sbeta\tensor \one)
\ee
then we obtain
\be\label{proof-QBT-1}
   \T_{X,Y} \quad = \quad
   \ipic{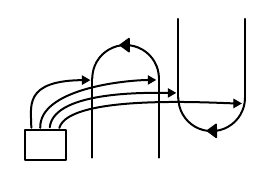}{1.1}
	 \put(-44,44){\scriptsize$Y$}  \put(-8,44){\scriptsize$Y^\ast$} 
	 \put(-116,-28){\scriptsize$Q$}
	 \put(-92,-42){\scriptsize$X^\ast$} \put(-56,-42){\scriptsize$X$} 
	\quad = \quad
  \ipic{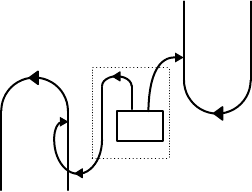}{1.1}
  \put(-39,53){\scriptsize$Y$}  \put(-3,53){\scriptsize$Y^\ast$} 
	 \put(-70,-19){\scriptsize$\BTM$}
  \put(-137,-59){\scriptsize$X^\ast$} \put(-101,-59){\scriptsize$X$} 	
  \quad ,
\ee
where $\BTM =  (\mu\tensor\mu) \circ [S\tensor \id\tensor \id \tensor S] (Q)$
and a direct calculation shows that it equals to~\eqref{eq:BTM}, as the notation suggests. The morphism within the dotted frame in  RHS of~\eqref{proof-QBT-1} is then obviously the map $\BTmap$ and therefore we finally have that 
	$\T_{X,Y} = j_Y \circ \BTmap\circ\iota_X$.
\end{proof}
 
In~\cite{[BT]}, a quasi-triangular quasi-Hopf algebra $A$ is called factorisable if the map $\BTmap$ defined in~\eqref{eq:BTmap} is an isomorphism. As an immediate consequence of Lemma~\ref{lem:BTmap-DD} and Proposition~\ref{prop:C-DD} we get:
 
\begin{corollary}\label{prop:fact-qHopf=fact-coend}
For a finite-dimensional quasi-triangular  quasi-Hopf algebra $A$,
the map $\BTmap$ from~\eqref{eq:BTmap} is an isomorphism if and only if $A$
 is factorisable in the sense of Definition~\ref{def:fact-qHopf}. 
\end{corollary}

\begin{remark}\label{rem:DD-is-Drinfeld-for-Hopf}
Consider the case that $A$ is a quasi-triangular Hopf algebra and use the end $\Gamma=(A, j)$ from Proposition~\ref{prop:end-RepH}. In this case, the map $\DD$ defined by~\eqref{eq:T-DD} is precisely the Drinfeld map. Indeed, for
$\as=\one^{\otimes3}$ and $\Salpha=\Sbeta=\one$, the expression for
$\BTM$ in \eqref{eq:BTM} reduces to $M = R_{21} R$ and $\BTmap$
from \eqref{eq:BTmap} (which is equal to $\DD$ by Lemma \ref{lem:BTmap-DD}) 
becomes the Drinfeld map, cf.\ Remark \ref{rem:fact-qHopf-def}\,(2).
We also note that we have defined so far two maps from $A^*$ to $A$ 
-- the map 
$(\delta^{\vect}_A)^{-1} \circ\Dm$ in Remark \ref{rem:fact-qHopf-def} and the Drinfeld map $\DD$. The difference is that the first intertwines the coadjoint and its dual actions, while the second intertwines the coadjoint and adjoint actions. Both can be used to test factorisability of $A$.
\end{remark}

\subsection{Integrals and cointegrals} \label{sec:int-coint}

Let $A$ be a factorisable quasi-Hopf algebra over an algebraically closed field $\ok$ (actually, we only need $\ok$ to contain square roots).

{}From
Proposition \ref{prop:integrals-exist} we know that $\coend$ has a unique-up-to-scalar two-sided integral $\Lambda_\coend : \one \to \coend$.
We will impose the normalisation
$\omega_\coend \circ (\Lambda_\coend \otimes \Lambda_\coend) \circ \lambda_{\one}^{-1}
= \id_{\tid}$, which makes $\Lambda_\coend$ unique up to sign.
Given a two-sided integral, we obtain a two-sided cointegral $\coint_\coend : \coend \to \one$ by
\be\label{eq:two-sided-cointegral}
	\coint_\coend := \Omega(\Lambda_\coend) \ ,
\ee
where $\Omega$ was given in \eqref{eq:Omega-def} (cf.\ Lemma \ref{lem:coint-from-int}). By construction, we have
\be\label{eq:int-coint-relative-norm}
	\coint_\coend \circ \Lambda_\coend = \id_{\one} \ .
\ee

There are no closed formulas for the integral or the cointegral. Instead one needs to compute the spaces of solutions to the linear conditions \eqref{eq:integral-def-general}. We will spell out these conditions for $A$ in terms of the structure maps given in Theorem \ref{thm:explicit-coend-qHopf}.

Let $\hat\Lambda_\coend \in A^*$ and $\hat\Lambda^{\mathrm{co}}_\coend \in A$ be defined as, for $a \in A$, $f \in A^*$,
\be
	\Lambda_\coend(1) = \big( a \mapsto \hat\Lambda_\coend(a) \big)
	\quad , \quad
	\coint_\coend(f) = f(\hat\Lambda^{\mathrm{co}}_\coend) \ .	
\ee
Conditions \eqref{eq:integral-def-general} turn into the following linear relations on $\hat\Lambda_\coend$ and $\hat\Lambda^{\mathrm{co}}_\coend$, for all $a \in A$,
\begin{eqnarray}
\label{eq:coend-int-coint-hatted}
(\id 
\otimes \hat\Lambda_\coend) \circ \mm(a)
~=
&\Salpha \cdot \hat\Lambda_\coend(a)&
=~
(\hat\Lambda_\coend \otimes \id) \circ \mm(a)\ ,
\\ \nonumber 
\co(\hat\Lambda^{\mathrm{co}}_\coend \otimes a)
~= 
	&\hat\Lambda^{\mathrm{co}}_\coend \cdot \eps(\Sbeta a)&
=~
\co(a \otimes \hat\Lambda^{\mathrm{co}}_\coend) \ ,
\end{eqnarray}
where $\mm$ and $\co$ are introduced in~\eqref{eq:qHopf-coend-dualstructuremaps}.
The normalisation condition is quadratic and reads
\be
	(\hat\Lambda_\coend \otimes \hat\Lambda_\coend)(\Hp) = 1 \ ,
\ee
and the relative normalisation of integral and cointegral is fixed by \eqref{eq:int-coint-relative-norm} to be $\hat\Lambda_\coend(\hat\Lambda^{\mathrm{co}}_\coend)=1$.

A {\em left (resp.\ right) integral for a quasi-Hopf algebra $A$} is an element $\mathbf{c} \in A$ satisfying $a \cdot \mathbf{c} = \eps(a) \mathbf{c}$ (resp.\ $\mathbf{c} \cdot a = \eps(a) \mathbf{c}$) for all $a \in A$,
see e.g.\ \cite{Bulacu:2011}.
Finite-dimensional quasi-Hopf algebras possess a one-dimensional space of 
left and of right integrals \cite{Hausser:1999} (see also \cite{Panaite:2000} for existence).

If one knows a non-zero left or right integral for $A$, then the following proposition provides a shortcut for computing
$\coint_\coend$.

\begin{proposition}\label{prop:coint-coend-via-int-qHopf}
Let $A$ be a finite-dimensional quasi-Hopf algebra over some field, and let $\mathbf{c} \in A$ be left (resp.\ right) integral for $A$. Then $\langle - , \mathbf{c} \rangle \in \coend^* = A^{**}$ is a left (resp.\ right) cointegral for $\coend$,
i.e.\ $\hat\Lambda^{\mathrm{co}}_\coend = \mathbf{c}$.
\end{proposition}

\begin{proof}
This is an immediate consequence of the explicit form of the structure maps of $\coend$ given in Theorem \ref{thm:explicit-coend-qHopf}. 
	The two conditions in the second line of \eqref{eq:coend-int-coint-hatted} become
\be\label{eq:L-coint-from-A-int_aux}
\text{left:} \quad
\co(\mathbf{c} \otimes a)
= 
\mathbf{c} \cdot \eps(\Sbeta a)
\qquad , \qquad
\text{right:} \quad
\co(a \otimes \mathbf{c}) 
=
\mathbf{c} \cdot \eps(\Sbeta a) \ .
\ee
Now substitute the expression for $\co$ in terms of $D$ as given in \eqref{eq:qHopf-coend-dualstructuremaps}. The left hand side in each case in  \eqref{eq:L-coint-from-A-int_aux} then becomes
\begin{align}\label{eq:left-int-Hopf-1}
\text{left:} \quad 
&\co(\mathbf{c} \otimes a)
=
\sum_{(D)} \eps(D_1) \eps(a)\eps(D_2) \eps(D_3) \mathbf{c} D_4 \ ,
\\ 
\text{right:} \quad
&\co(a \otimes \mathbf{c}) 
=
\sum_{(D)} S(D_1)\mathbf{c}\eps(D_2)\eps(D_3)\eps(a) \eps(D_4) \ , \label{eq:left-int-Hopf-2}
\end{align}
where we used $\eps \circ S = \eps$, which holds for quasi-Hopf algebras by \cite[part 7 of remark on p.\,1425]{Dr-quasi}.
Substituting the explicit form of $D$ from \eqref{eq:coend-maps-elementsDW} 
and using the counitality condition~\eqref{eq:counital} on  $\Phi^{-1}$ we get for RHS of~\eqref{eq:left-int-Hopf-1}
\be
\sum_{(\Phi)} \eps(\Sbeta a) \mathbf{c}\,  \eps(\Phi_1) \eps(\Phi_2) \eps(\Phi'_3)  \Phi''_3 \\
=  \eps(\Sbeta a) \mathbf{c} \bigl\{(\eps\ot \id) \circ \Delta \bigr\} \bigl((\eps\ot\eps\ot\id)(\Phi)\bigr) 
=   \eps(\Sbeta a) \mathbf{c} \ ,
\ee
where the counitality condition was used once more, and similarly for RHS of~\eqref{eq:left-int-Hopf-2}.
This shows that both the identities in \eqref{eq:L-coint-from-A-int_aux} hold.
\end{proof}

In particular, for factorisable $A$ we know by Proposition \ref{prop:integrals-exist} that the coend $\coend$ has a one-dimensional space of two-sided integrals (and hence of cointegrals).
The above proposition hence shows that $A$ has two-sided integrals, and a non-zero such two-sided integral $\mathbf{c}$ spans the space of cointegrals of $\coend$.

\subsection{Internal characters}\label{sec:intchar-qHopf}

Let $A$ be a factorisable ribbon
quasi-Hopf algebra over a field $\ok$.
	Then $\rep A$ is in particular pivotal, and we can use the results of Section \ref{sec:intchar-natendo}.
In this section we will give explicit expressions for the natural endomorphisms $\phi_V$ and $\modS_{\cat}(\phi_V)$ in terms of elements of the centre of $A$. Recall from \eqref{eq:tildechi-def} that $\phi_V$, and hence also
$\modS_{\cat}(\phi_V)$, is an image of the internal characters $\chi_V$
as defined in \eqref{eq:chiV-def}.

Let us identify $\chi_V$ with their images $\chi_V(1)\in A^{*}$. A short calculation using~\eqref{eq:qHopf-tilde-evcoev} and the definition of $\iota_V$ in Figure~\ref{fig:coend-qHopf} gives the linear forms  $\chi_V$
which we call \textit{$q$-characters}:
\be\label{eq:chi-Tr}
\chi_V(-) = \tr_V(\bal \cdot - )\colon\; A \longrightarrow \ok
\quad ,  \quad \text{ where }
~~ \bal = \sqs^{-1} \ribbon S(\Sbeta)\ .
\ee

Denote by $\xi\colon Z(A) \to \End(\id_{\rep A})$ the $\ok$-algebra isomorphism
between the centre of $A$ and the natural endomorphisms of the identity
functor on $\rep A$. Explicitly, for all $V \in \rep A$, $v \in V$ and $z \in Z(A)$,
\be\label{eq:xi-ZA-EndId-def}
    (\xi(z)_V)(v) = z.v \ .
\ee
We can use $\xi$ to represent the $S$-transformation of $\phi_V$ 
in terms of a central element $\bchi_V \in Z(A)$ as
$\xi(\bchi_V) = S_{\rep A}(\phi_V)$.
A short calculation starting from 
\eqref{eq:S(phi)-generalcat} gives
\be\label{eq:chiV-central-def}
\bchi_V
~=
\sum_{(\Phi),(\Psi),(M)} \chi_V\Big( S\big(\Psi_2 M_2  \Phi_2  \big)  \Salpha \, \Psi_3  \Phi_3  \Big) \, \Psi_1 \, M_1\,  \Phi_1 \ ,
\ee
where $\Psi=\Phi^{-1}$
and used~\eqref{eq:monodromy-el}.
We obtain the following corollary to Theorem \ref{thm:S_C(phi_M)-algebramap}.

\begin{corollary}
Let the field $\ok$ be algebraically closed and of characteristic zero. Then
the map
 $[V] \mapsto \bchi_V$ is an injective ring-homomorphism $\Gr(\rep A) \to Z(A)$.
\end{corollary}

The Verlinde-type formula \eqref{eq:verlinde-general} takes the form
\be
    \bchi_U \, \bchi_V
    = \sum_{W \in \Irr(\rep A)} N_{UV}^{~W} \, \bchi_W \ .
\ee
By linear independence of the $\bchi_W$, this determines the structure constants $N_{UV}^{~W}$ of $\Gr(\rep A)$ uniquely (for $\ok$ as in the corollary). We stress that it is not necessary to compute the centre of $A$ to evaluate this formula, but one does need to know all simple $A$-modules $U$ and be able to compute their characters $\tr_U(-)$.

\newcommand{\cointL}{\hat\Lambda^{\mathrm{co}}_\coend}
For completeness we also give the central element  $\bphi_V \in Z(A)$ representing
$\phi_V$ from \eqref{eq:phiV-rib-def} via $\xi(\bphi_V) = \phi_V$:
\be \label{bphi}
\bphi_V ~=~ \sum_{(F)}  F_1 \,
\chi_V\big(F_2\big)\ ,
\ee
where
\be
F= \eps(\Sbeta) 	\sum_{(\Psi),(\Phi),(\intQ)}
 \Psi_1 \intQ' \Phi_1 \tensor S(\Psi_2 \intQ''\Phi_2)\Salpha\Psi_3\Phi_3
\quad \text{and}\quad \Psi=\Phi^{-1}\ .
\ee
During the calculation we used that $\cointL = \mathbf{c}$ is a two-sided integral (Proposition~\ref{prop:coint-coend-via-int-qHopf} and the text above \eqref{eq:two-sided-cointegral}) and so $a \cdot \mathbf{c} = \eps(a) \mathbf{c}=\mathbf{c} \cdot a$ for all $a \in A$,  
together with counitality of $\Phi^{\pm1}$.

\begin{remark}
Let  $A$ be a finite-dimensional factorisable
ribbon Hopf algebra over $\ok$.
Recall \cite{[Drinfeld]}
that the space $\mathrm{qCh}$ of  $q$-characters of $A$ is defined as the space of invariants in $A^*$ under the coadjoint action.
	For example, the linear forms \eqref{eq:chi-Tr}, which now read $\chi_V = \tr_V\bigl(\sqs^{-1}\,\ribbon\, (-)\bigr)$, are $q$-characters in this sense, justifying their name.
 The two families of elements $\bchi_V$ and $\bphi_V$ simplify to
\be\label{eq:bchi-Hopf}
\bchi_V~= \sum_{(M)} M_1 \, \tr_V\big( \sqs^{-1}  \ribbon \, S(M_2) \big)  
\ \stackrel{(*)}{=}\ 
S^{-1} \circ \bigl(
	\chi_V(-) 
\ot \id\bigr) (M)
\ee
and (recall that we found $\cointL=\mathbf{c}$ and here	$F=(\id \otimes S) \circ \Delta(\mathbf{c})$)
\be\label{eq:bphi-Hopf}
\bphi_V ~=~ \sum_{(c)}  \mathbf{c}' \,
\tr_V\big( \sqs^{-1}\,\ribbon\, S(\mathbf{c}'')\big)\ \stackrel{(**)}{=}\ 
S^{-1} \circ \bigl(
	\chi_V(-) 
\ot \id\bigr) \circ
\Delta(\mathbf{c})\ ,
\ee
where for $(*)$ we used the identity\footnote{
This results in $(\id\ot S)(M) = \tau\circ (\id\ot S^{-1})(M)$ and then $(\id \tensor \phi) \circ (\id \otimes S)(M) 
	= (\phi\tensor \id)\circ (\id \otimes S^{-1})( M)
    = S^{-1} \circ \big(	(\phi\tensor \id)(M) \big)$, for $\phi\in A^*$.}
$(S\ot S)(R) = R$.
For $(**)$ we used that
$S(\mathbf{c}) =\mathbf{c}$  as follows from~\cite[Eqn.\,(2)]{R1}, which reads in our case as $S(a) = \sum_{(\mathbf{c})} \hat\Lambda_\coend(\mathbf{c}' a) \mathbf{c}''$,
and therefore $\Delta^{\mathrm{op}}(\mathbf{c}) = (S\ot S) \circ \Delta (\mathbf{c})$. \\
	Another way to relate the central elements $\bchi_V$ and $\bphi_V$ to $q$-characters is as follows.
The (algebra) map 
$S\circ \xi^{-1}\circ\,\psi^{-1}\circ\,\Omega$  from the space of $q$-characters to $Z(A)$
is the well-known \textit{Drinfeld mapping} given by
 $\mathbb{D}_{A^*,A}\colon \phi(\cdot)\mapsto (\phi\otimes\id) M$, see~\cite{[Drinfeld]}, while
the map  
$S\circ \xi^{-1}\circ\,\psi^{-1}\circ\rho^{-1}\colon \mathrm{qCh} \to Z(A)$ 
is the \textit{Radford mapping}
  given by  $\phi(\cdot)\mapsto (\phi\otimes\id)\Delta(\intQ)$,
  see~\cite{R1} for its definition and properties.
The central elements~\eqref{eq:bchi-Hopf} and~\eqref{eq:bphi-Hopf} are then images of the $q$-characters 
	$\chi_V(-)$ 
under the Drinfeld and Radford mappings
(composed with~$S^{-1}$), correspondingly.
\end{remark}

\section{$SL(2,\oZ)$-action for ribbon quasi-Hopf algebras} 
\label{sec:SL2Z-quasiHopf}

In this section, we assume that $A$ is a factorisable quasi-Hopf algebra and
we express the $S$- and $T$-transformations from~\eqref{eq:cat-ST}  in $\rep A$, and compute
the resulting action on
the centre~$Z(A)$ of $A$. 
To start with, we evaluate the map $\mathcal{Q}$ from \eqref{eq:Q-map} in $\rep A$. One finds, for $a,b \in A$, $f,g \in A^*$,
\be\label{eq:Q-RepA}
\big\langle \mathcal Q (f\tensor g) \,,\, a \otimes b \big\rangle
~=~
\big\langle f\tensor g \,,\, \hat{\mathcal Q}(a \otimes b) \big\rangle
\ee
with
\begin{align}\label{eq:Q-hat}
&\hat{\mathcal Q}(a \otimes b)
= \big(S(X_3) a X_4\big)  \ot \big(S(X_1)bX_2 \big) \ ,
\\ \nonumber
&X=
(\id\tensor\id\tensor\Delta)(\as)
\cdot
(\one\tensor\as^{-1})
\cdot
(\one\tensor M \tensor\one)
\cdot
(\one\tensor\as)
\cdot
(\id\tensor\id\tensor\Delta)(\as^{-1})
\ .
\end{align} 

Then the $S$- and 
$T$-transformations 
from~\eqref{eq:cat-ST} take the form, for $a \in A$, $f \in A^*$,
\begin{align} \label{SL2Z-action-A-star}
\big\langle \modS(f) \,,\, a \big\rangle &~=~ 
\big\langle \mathcal Q (f\tensor \Lambda_\coend) 
\,,\, a \otimes \Salpha \big\rangle 
\ , 
 \\ \nonumber
\big\langle \modT(f) \,,\, a \big\rangle &~=~ 
\big\langle f \,,\, \ribbon^{-1}a \big\rangle  \ .
\end{align}

Our next aim is to 
evaluate the action of the $S$- and $T$-generators on $\End(\id_{\cat})$ as given in~\eqref{eq:SCTC} in the case $\cat = \rep A$.
To do so, we use the isomorphism 
	$\xi$ from \eqref{eq:xi-ZA-EndId-def}
and will give the corresponding action on elements of $Z(A)$ instead.
The result is:

\begin{theorem}\label{thm:SL2Z-on-centre}
The $S$- and $T$-transformations 
 on $Z(A)$ are given by the following linear maps $Z(A)\to Z(A)$: for $z \in Z(A)$,
\begin{align} \label{eq:STact-ZA}
	\modS_Z(z)&~= \sum_{\substack{\Psi,(\Hp)}} 
	\Psi_1 \, \Sbeta \, S(\Psi_2)
	\, ({\Hp})_1 \, \Psi_3 \, 
		\hat\Lambda_\coend
	\Big(\co\big((\Hp)_2\ot \Salpha z\big)\Big) \ , \\
	\modT_Z(z)&~= \ribbon^{-1} z \gc \nonumber
\end{align}
where $\Psi=\Phi^{-1}$, and $\co$ and $\Hp$ are defined in \eqref{eq:qHopf-coend-dualstructuremaps}. 
\end{theorem}

\begin{proof}
We compute the ingredients of \eqref{eq:SCTC}.
Let $\varphi\in\End(\id_{\Cc})$ be the natural transformation which acts by a central element $z\in Z(A)$, i.e.\ $\xi(z)=\varphi$. Then
\be\label{proof-Sact-ZA-din-trafo}
   \ipic{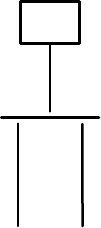}{.9}
	 \put(-31,37){\scriptsize$\psi(\varphi)$}
	 \put(-40,-59){\scriptsize$X^\ast$} \put(-13,-59)  {\scriptsize$X$}
	 \ \raisebox{4em}{\framebox{\scriptsize{$\rep A$}}} \quad \coloneqq \quad
   \ipic{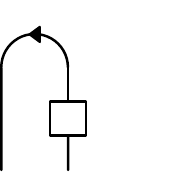}{.9}
	 \put(-46,-16){\scriptsize$\varphi$} 
	 \put(-77,-46){\scriptsize$X^\ast$} \put(-48,-46)  {\scriptsize$X$}
\ee
or equivalently as a diagram in $\vect_\ok$
\be\label{proof-Sact-ZA-din-trafo-repA}
   \ipic{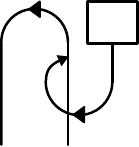}{.9}
	 \put(-20.5,20){\scriptsize$\psi(\varphi)$}
	 \put(-65,-41){\scriptsize$X^\ast$} \put(-36,-41)  {\scriptsize$X$}
	 \ \raisebox{4em}{\framebox{\scriptsize{$\vect_\ok$}}} \quad \coloneqq \quad
   \ipic{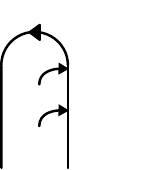}{.9}
	 \put(-56,-7){\scriptsize$\Salpha$} \put(-56,-25){\scriptsize$z$} 
	 \put(-77,-46){\scriptsize$X^\ast$} \put(-48,-46)  {\scriptsize$X$}
\ee
that is,
 $\psi(\varphi)=\langle -,\Salpha z\rangle \in A^{**}$.
Let $r\in \Hom_A(\coend,\one)$ and $s\in\Hom_A(\one,\coend)$, and write $r = \langle -,\tilde r \rangle$ for some $\tilde r \in A$. The other maps in \eqref{eq:SCTC} are given by
\be\label{proof-Sact-ZA-rho}
   \Rad(r)\ = \ 
   \ipic{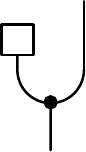}{.9}
	 \put(-32,13){\scriptsize$r$}
	 \put(-19,-41){\scriptsize$\intL$} \put(-10,-20){\scriptsize$\Delta_\coend$}
	 \ \raisebox{4em}{\framebox{\scriptsize{$\vect_\ok$}}} \quad = \quad
   \ipic{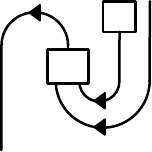}{.9}
	 \put(-42,1){\scriptsize$\co$} \put(-16,23){\scriptsize$r$}
	 \put(-69,-41){\scriptsize$\intL$}
	 \quad = \quad \hat{\Lambda}_\coend \big( \co (- \ot \tilde r) \big) \gc
\ee
\be\label{proof-Sact-ZA-Omg}
   \Omega(s)\ = \
   \ipic{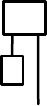}{.9}
	 \put(-15,14){\scriptsize$\omega_\coend$}
	 \put(-7,-32){\scriptsize$\coend$} \put(-17,-10){\scriptsize$s$}
	 \ \raisebox{3em}{\framebox{\scriptsize{$\vect_\ok$}}}\quad = \quad
   \ipic{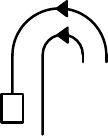}{.9}
	 \put(-10,-7){\scriptsize$\Hp$} \put(-31,-39){\scriptsize$\coend$} \put(-44,-19){\scriptsize$s$}
	 \quad = \quad \langle - , (\id\ot s)(\Hp) \rangle \gc 
\ee	
\be\label{proof-Sact-ZA-psi-inv}
   \psi^{-1}(r)_X\ = \ 
   \ipic{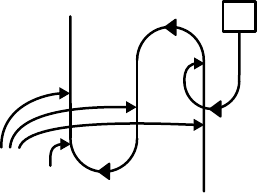}{.9}
	 \put(-113,-33){\scriptsize$\Psi$} \put(-93,-39){\scriptsize$\Sbeta$} \put(-9,33){\scriptsize$r$} 
	 \put(-28,-50){\scriptsize$X$} \put(-84,38){\scriptsize$X$} 
	 \ \raisebox{4em}{\framebox{\scriptsize{$\vect_\ok$}}}\quad = \quad
   \ipic{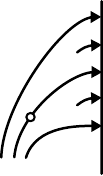}{.9}
	 \put(-46,-40){\scriptsize$\Psi$} 
	 \put(-17,-14){\scriptsize$\tilde r$} 
	 \put(-19,8){\scriptsize$\Sbeta$} 
	 \put(-5,-47){\scriptsize$X$} \put(-5,41){\scriptsize$X$} 
	 \quad = \quad \sum\limits_{(\Psi)}\Psi_1\,\Sbeta\, S(\Psi_2)\,\tilde r\, \Psi_3 \act (-) \; \gc
\ee
where $\Psi=\Phi^{-1}$. Now it is easy to see that \eqref{eq:SCTC} reduces to  \eqref{eq:STact-ZA}. 
\end{proof}

Recall that in Section~\ref{sec:intchar-qHopf} we introduced the special central elements $\bchi_V$ and $\bphi_V$ that are related to the internal characters $\chi_V$.
As a corollary of Theorem~\ref{eq:STact-ZA} and by definition $\xi(\bchi_V) = S_{\rep A}(\phi_V)$, we have the following $S$-transformation of these elements. 
\begin{corollary}\label{cor:S(PHI)=CHI}
$\bchi_V = \modS_Z(\bphi_V)$ and $\bphi_{V^*} = \modS^2_Z(\bphi_V)$.
\end{corollary}

Next we give the $S$-transformation on $A$ as $\hat \modS: A\to A$, using 
\eqref{SL2Z-action-A-star},
\be\label{eq:hat-modS}
\langle \modS(f), a  \rangle = \langle f, \hat\modS(a)  \rangle \ , 
\ee
for any $f\in A^*$ and $a\in A$. We easily get
\be\label{eq:modS-Z-1}
\hat\modS(a) =  (\hat\Lambda_\coend \ot \id) \Bigl[ \hat{\mathcal{Q}} (a\ot \Salpha)\Bigr]
= \sum_{(X)}\hat\Lambda_\coend \Bigl(S(X_3)a X_4\Bigl) S(X_1) \Salpha X_2
\ee
or equivalently, using the relation~\eqref{eq:ST-on-EndId_aux1} we have
\be\label{eq:modS-Z-2}
\hat\modS(a) = 
 \sum_{(\Phi),(\hat\omega_\coend)}\hat\Lambda_\coend \Big(\co\big(S(\Phi'_3)\,a\,\Phi''_3\ot S(\Phi'_2)(\Hp)_1\Phi''_2\big)\Big)
S(\Phi'_1)(\Hp)_2 \Phi''_1\ .
\ee
For $\hat \modT$, we obviously have $\hat \modT(a) = \ribbon^{-1} a$. Following~\eqref{eq:ST-endo-coend} and  the definition in~\eqref{eq:hat-modS}, $\hat\modS$ and $\hat\modT$ satisfy
\begin{align}
(\hat\modS\hat\modT)^3=\lambda \, \hat\modS^2 \ , \qquad \hat\modS^2 = \hat{S}_\coend^{-1} \ , \qquad \ \lambda\in\ok^{\times}\ ,
\end{align} 
with the antipode $ \hat{S}_\coend$ as in Theorem~\ref{thm:explicit-coend-qHopf}.
We note then the 
$S$-transformation~\eqref{eq:hat-modS} 
with~\eqref{eq:modS-Z-2} simplifies on linear forms $f\in\Cc(\one,\coend)$, or on the invariants $\hat f := f(1)$ of the coadjoint action
(recall that $\modS\colon \coend \to \coend$ is an intertwiner of the coadjoint action of $A$ on $\coend=A^*$ and thus $\modS$ acts on the space invariants):
\be\label{eq:modS-f}
 \modS(\hat f)(a) = \hat f\big( \hat\modS(a) \big) =  \sum_{(\hat\omega_\coend)}\hat\Lambda_\coend \Big(\co\big( a\ot (\Hp)_1\big)\Big)
\hat f\big((\Hp)_2\big)\ , \qquad a\in A\ ,
\ee
where  we used the counitality of $\Phi$. The two projective $SL(2,\mathbb{Z})$ actions, 
on $Z(A)$ in~\eqref{eq:STact-ZA} and on $\cat(\one,\coend)$ in~\eqref{eq:modS-f}, 
are related by  conjugation with the isomorphism $\rho\circ\psi\circ\xi$ between the centre $Z(A)$ and the space of coadjoint invariants.
(Note that the action~\eqref{eq:modS-f} agrees with the one from the categorical formula~\eqref{eq:Rad-Omega-S}).

Similarly, we have a projective $SL(2,\mathbb{Z})$ representation on the space $\cat(\coend,\one)$ of \textsl{co}invariants of the coadjoint action
that is identified by $\delta^{\vect}_A$ with the subspace $\Salpha\cdot Z(A) \subset A$.
	Indeed, by \eqref{proof-Sact-ZA-din-trafo-repA} the image of the isomorphism  $\psi\circ\xi : Z(A) \to \cat(\coend,\one)$ is $\delta^{\vect}_A(\Salpha\cdot Z(A))$. 
	In general, $\Salpha\cdot Z(A)$ is different from the centre $Z(A)$ and from the space of invariants of the adjoint action, 
which is $\Sbeta\cdot Z(A)$ (in the Hopf algebra case,  the three spaces are identical). 
	The $S$-transformation acts on $\Salpha\cdot Z(A)$ by restricting $\hat\modS$ to it. When evaluated on $\Salpha z$, the result~\eqref{eq:modS-Z-2} simplifies to
 \be\label{eq:modS-Z-2-alpha}
\hat\modS(\Salpha z) = 
 \sum_{(\hat\omega_\coend)}\hat\Lambda_\coend \Big(\co\big(\Salpha z\ot (\Hp)_1\big)\Big)
(\Hp)_2 \ ,\qquad z\in Z(A) \ ,
\ee
where we again used the counitality of $\Phi$. 
By construction, $\hat\modS$ commutes with the $A_{\mathrm{op}}$-action 
on $A$ given by $\rho(b)\colon a \mapsto \sum_{(b)}S(b') a b''$ \footnote{ Note the different position of the antipode $S$ here with respect to the adjoint action.} and thus $\hat\modS$ acts on the spaces of invariants of this action, which is indeed $\Salpha\cdot Z(A)$.

\begin{remark}
If $A$ is a factorisable Hopf algebra, recalling Remark~\ref{rem:coend-str-Hopf}
we have the $S$-transformation~\eqref{eq:modS-Z-1}
(or equivalently~\eqref{eq:modS-Z-2}) on $A$ as
\be
\hat\modS(a) = \sum_{(M)} \hat\Lambda_\coend\big(S(M_2)a\big) M_1
= \sum_{(M)} \hat\Lambda_\coend\big(M_1 a\big) S^{-1}(M_2)\ , \quad a\in A\ ,
\ee
where we used $(S\ot S)(M) = \tau_{A,A} (M)$,
while $\modS_Z$ from~\eqref{eq:STact-ZA} becomes
\be\label{eq:SZ-Hopf}
\modS_Z(z) = \sum_{(M)} \hat\Lambda_\coend\big(z M_1\big) S(M_2) \ , \quad z\in Z(A) \ .
\ee 
When $\hat\modS$ 	is
restricted to $Z(A)$ the two formulas become equal: $\modS_Z(z) = S^2\big(\hat\modS(z)\big)$, but $S^2$ is identity on the centre as $S^2(a) = \sqs a \sqs^{-1}$.
We note that~\eqref{eq:SZ-Hopf} agrees with the $S$-transformation obtained in~\cite[Sec.\,2]{[Kerler]} 
\footnote{Note that $\hat\Lambda_\coend\colon A \to \ok$ is an intertwiner  
for the $A_{\mathrm{op}}$-action
on A by $\rho(b)\colon a \mapsto \sum_{(b)}S(b') a b''$. Therefore, the right integral condition in RHS of~\eqref{eq:coend-int-coint-hatted} (the one in the context of the coends) simplifies to the standard right-cointegral condition for a Hopf algebra: $(\hat\Lambda_\coend\ot\id)\circ \Delta(a) = \hat\Lambda_\coend(a) \one$. 
(But the LHS of~\eqref{eq:coend-int-coint-hatted} does not simplify to the standard left-cointegral condition for a Hopf algebra.)
The right (co)integral $\mu_D$ used in~\cite[Section 2]{[Kerler]} thus coincides with our $\hat\Lambda_\coend$ 
(possibly up to a sign),
as the normalisation is also the same.}, 
which is a slight rewriting of the ``quantum Fourier'' $S$-transformation originally obtained in~\cite{Lyubashenko:1994ma}.
\end{remark}

\appendix

\section{Dinatural transformations}\label{app:dinat-tr}
We recall here the concept of dinatural transformations between two functors.
Let $\cat$ and~$\catD$ be any categories and let $F\colon \cat^{\operatorname{op}} \times \cat \to \catD$ and
$G\colon \cat \times  \cat^{\operatorname{op}}  \to \catD$  be two functors.
 The next definition is a slight modification of the one in 
\cite[Ch.\,IX.4]{MacLane-book} -- the order of categories for the second functor is different.

\begin{definition}\label{def:dinat-tr}
A \emph{dinatural transformation} from the functor $F$ to the functor $G$
is a family of morphisms $\phi\equiv\left(\phi_U \colon F(U,U)\to G(U,U)\right)_{U\in\cat}$ in~$\catD$,
written $\phi\colon F \stackrel{..}{\longrightarrow} G$, that makes the diagram 
\begin{center}
\makebox[0pt]
{
\begin{xy}
  \xymatrix@C+=1.5cm{
  	  		&  F(V,V) \ar[r]^{\phi_{V}} & G(V,V) \ar[dr]^{G(\id, f)}  & \\
      F(V,U)  \ar[dr]_{F(f,\id)} \ar[ur]^{F(\id ,f)} &  & & G(V,U) \\
     			&  F(U,U) \ar[r]^{\phi_{U}}   & G(U,U)\ar[ur]_{G(f, \id)}  &
  }
\end{xy}
}
\end{center}
commute  for all $U,V\in\cat$ and $f\in{\cat}(U,V)$. 
\end{definition}

For coends and ends, we  consider the cases where one of the  functors, $F$ or $G$, is a ``constant'' functor: e.g. $G\colon U\times V \mapsto B$  for all $U,V\in\cat$ and  an object $B\in\catD$. Definition~\ref{def:dinat-tr} then reduces to the following one.

\begin{definition}\label{def:dinat-const}
\mbox{}
\begin{enumerate}\setlength{\leftskip}{-1.5em}
\item
A \emph{dinatural transformation} from the functor $F$ to an object $B\in\catD$ 
is a family of morphisms $\phi\equiv\left(\phi_U \colon F(U,U)\to B\right)_{U\in\cat}$ in~$\catD$,
written $\phi\colon F \stackrel{..}{\longrightarrow} B$, that makes the diagram 
\begin{equation*}
\begin{xy}
  \xymatrix@C+=1.5cm{
      F(V,U)  \ar[d]_{F(f,\id)} \ar[r]^{F(\id ,f)} &  F(V,V) \ar[d]^{\phi_{V}}   \\
      F(U,U) \ar[r]^{\phi_{U}}   &  B 
  }
\end{xy}
\end{equation*}
commutative  for all $U,V\in\cat$ and $f\in{\cat}(U,V)$. 

\item
A \emph{dinatural transformation} from an object $B\in\catD$  to the functor $G$  is a family of morphisms $\phi\equiv\left(\phi_U \colon B\to G(U,U)\right)_{U\in\cat}$ in~$\catD$
that makes the diagram
\begin{equation*}
\begin{xy}
  \xymatrix@C+=1.5cm{
      B  \ar[d]_{\phi_{U}} \ar[r]^{\phi_{V}} &  G(V,V) \ar[d]^{G(\id ,f)}   \\
      G(U,U) \ar[r]^{G(f,\id)}   &   G(V,U)
  }
\end{xy}
\end{equation*}
 commutative  for all $U,V\in\cat$ and $f\in{\cat}(U,V)$. 

\end{enumerate}
\end{definition}

\section{Proof of Proposition \ref{prop:univ-vs-coend}}\label{app:univ-vs-coend}

We first prove that the two coalgebra structures are equal.
Recall that the coproduct for the universal Hopf algebra $(\coend,\varphi)$ is determined 
by the defining relation~\eqref{eq:rA-cop} involving the element from $\Nat(\id,\id\tensor (\coend\otimes \coend))$ -- the right-hand side of~\eqref{eq:rA-cop} --
while the coproduct for  the coend $(\coend,\iota)$ is given by the defining relation~\eqref{eq:cop-L} involving the element from $\Din((-)^* \otimes (-),\coend\otimes \coend)$.
By Corollary~\ref{cor:coend-Nat-rep} and Lemma~\ref{lem:I-Din},
and the universality property of the coend $\coend$ we have the commutative diagram (where all arrows are bijections)
\be\label{eq:comm-diag-3}
\xymatrix@C=40pt@R=30pt{
      \cat(\coend,V)  \ar[r]^{\varphi_{V}\quad }\ar[dr]_{(-) \circ \iota}  & \Nat(\id,\id\tensor V)    \\
         &   \Din((-)^* \otimes (-),V) \ar[u]_{\zeta_{V}}
}
\ee
We use this diagram for $V=\coend\otimes \coend$ to compare the coproducts on  $(\coend,\iota)$ and on $(\coend,\varphi)$ -- 
	this
is equivalent to comparing the dinatural transformation on RHS of~\eqref{eq:cop-L} with  the image of RHS of~\eqref{eq:rA-cop} under the bijection $\zeta^{-1}_{\coend\otimes \coend}$. We begin with rewriting the map $X \xrightarrow{\nattr_X} X \otimes \coend$ in terms of $\iota$, recall the definition~\eqref{eq:nattr}. Applying the diagram~\eqref{eq:comm-diag-3} for $V=\coend$ we get  $\nattr = \natiso_{\coend}(\id) = \zeta_{\coend}(\iota)$ or graphically
\be\label{eq:iota-tilde-pic}
   \nattr_X\quad\coloneqq\quad
   \ipic{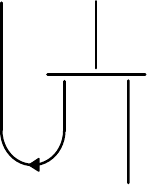}{.9}
	 \put(-66,42){\scriptsize$X$}   \put(-25,42){\scriptsize$\coend$}
	 \put(-5,12)  {\scriptsize$\iota_X$}
	\put(-12,-49)  {\scriptsize$X$}
	\ee
 Then,  RHS of~\eqref{eq:rA-cop} is
 \be\label{eq:iota-tilde-cop-pic}
   \natiso_{\coend\tensor\coend}(\Delta_\coend)_X \quad = \quad
   \ipic{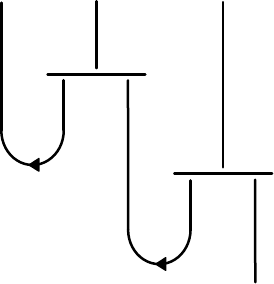}{.9}
		 \put(-121,63){\scriptsize$X$}   \put(-80,63){\scriptsize$\coend$}  \put(-24,63){\scriptsize$\coend$}
		 \put(-70,35){\scriptsize$\iota_X$}
	 	 \put(-13,-8)  {\scriptsize$\iota_X$}
	 \put(-13,-70)  {\scriptsize$X$}
	\ee
Applying $\zeta^{-1}_{\coend\otimes \coend}$ on it, recall~\eqref{eq:tilde-inv}, we get then indeed 
RHS of~\eqref{eq:cop-L} or the  dinatural transformation defining the coproduct on the coend $(\coend,\iota)$ in 	
	Figure~\ref{fig:Hopf-coend}. 
For the counit maps, we apply $\zeta_{\one}^{-1}$ on RHS of~\eqref{eq:rA-counit} and get indeed RHS of~\eqref{eq:eps-L}. Similarly for the antipode $S$ map (though, the algebra structure is discussed below), we apply 
 $\zeta_{\coend}^{-1}$ on RHS of~\eqref{eq:rA-S} and get indeed RHS of~\eqref{eq:antipode-L}, or the corresponding diagram in 
	Figure~\ref{fig:Hopf-coend}, 
after an elementary calculation using naturality of the braiding.

To compare the algebra structures, we use a direct calculation instead of the ``double'' analogue of the diagram~\eqref{eq:comm-diag-3}.
Recall that the multiplication for $(\coend,\natiso)$ is defined in~\eqref{eq:rA-mult} by the equality  $\natiso_\coend^2(\muc)_{X,Y} = \nattr_{X\otimes Y}$, where the map $\natiso^2_V$ is defined  in~\eqref{eq:natiso2}.
In order to show the equality of the multiplications on $(\coend,\iota)$ and $(\coend,\natiso)$, we compute the image of $\muc$ from~\eqref{eq:mult-L} (or graphically in 
	Figure~\ref{fig:Hopf-coend}) 
under the map $\natiso_\coend^2$ and show that it is equal to $\nattr_{X\otimes Y}$.
We have
 \be\label{eq:mult-L-A}
   \natiso_\coend^2(\muc)_{X,Y} \quad = \quad
   \ipic{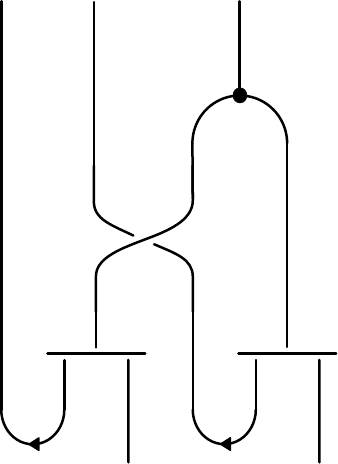}{.8}
		 \put(-133,92){\scriptsize$X$}   \put(-97,92){\scriptsize$Y$} \put(-40,92){\scriptsize$\coend$}
		  \put(-35,60){\scriptsize$\muc$}
		 \put(-88,-41)  {\scriptsize$\iota_X$}
	 	 \put(-13,-41)  {\scriptsize$\iota_Y$}
		 \put(-86,-98)  {\scriptsize$X$}  \put(-11,-98)  {\scriptsize$Y$}
   \quad = \quad
     \ipic{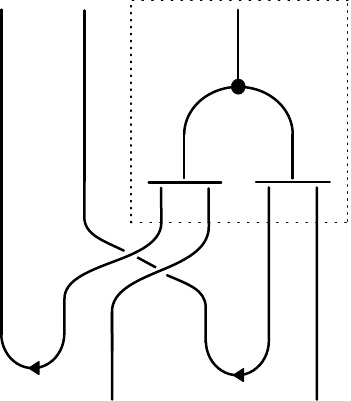}{.86}
		 \put(-150,82){\scriptsize$X$}   \put(-114,82){\scriptsize$Y$} \put(-48,84){\scriptsize$\coend$}
		  \put(-39,53){\scriptsize$\muc$}
		 \put(-88,11)  {\scriptsize$\iota_X$}
	 	 \put(-13,11)  {\scriptsize$\iota_Y$}
		 \put(-103,-92)  {\scriptsize$X$}  \put(-18,-92)  {\scriptsize$Y$}
	\ee
where we used naturality of the braiding. Using then the defining equality for $\muc$ in 
	Figure~\ref{fig:Hopf-coend}, 
we can replace the part of the diagram inside the dashed square by RHS of~\eqref{eq:mult-L} that gives (after an elementary graphical calculus)
\be\label{eq:mult-L-A-2}
   \natiso_\coend^2(\muc)_{X,Y} \ = \ 
   \ipic{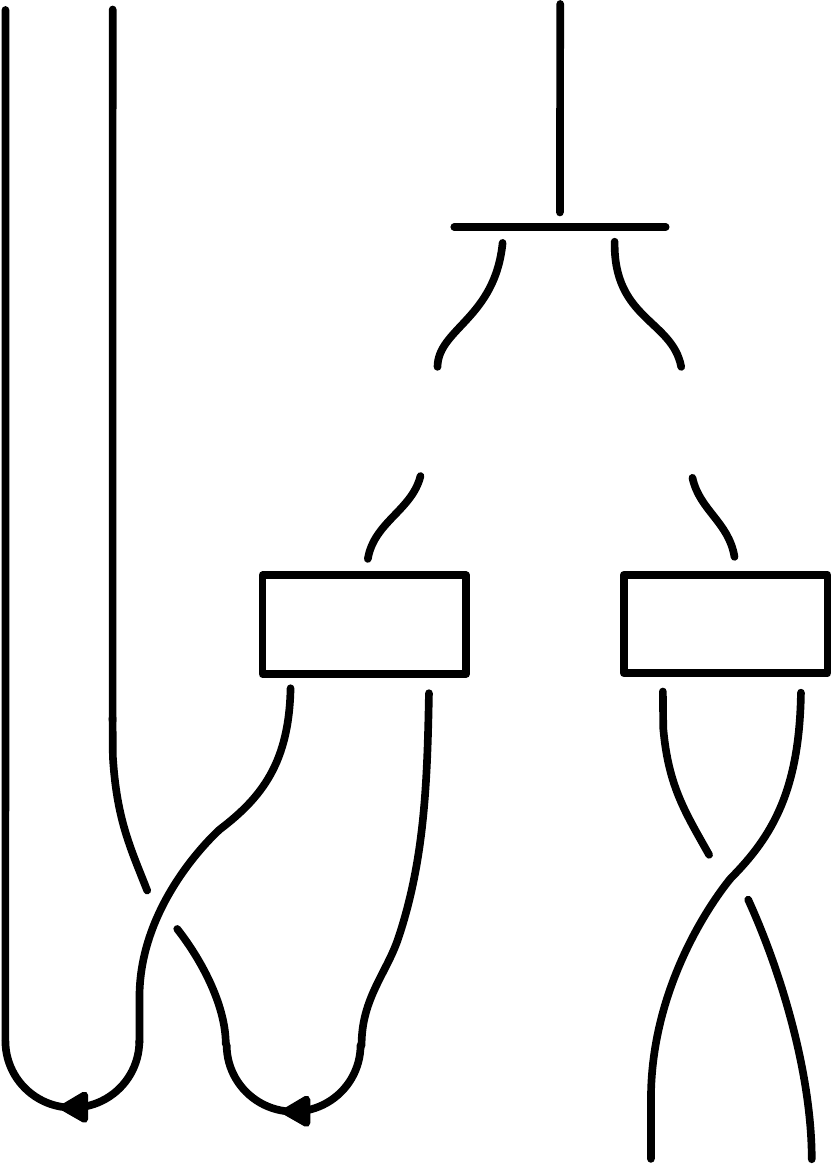}{.3}
	 \put(-123,86){\scriptsize{$X$}} 	 \put(-106,86){\scriptsize{$Y$}}  \put(-41,88){\scriptsize{$\coend$}} 
   \put(-31,57){\scriptsize$\iota_{Y\tensor X}$}
	 \put(-78,21){\tiny$(Y\tensor X)^\ast$} \put(-30,21){\tiny$Y\tensor X$}
	 \put(-77,-8){\scriptsize$\gamma_{Y,X}$} \put(-18,-9){\scriptsize$\id$}
	 \put(-31,-93){\scriptsize$X$} \put(-6,-93){\scriptsize$Y$}
\quad = \quad
	 \ipic{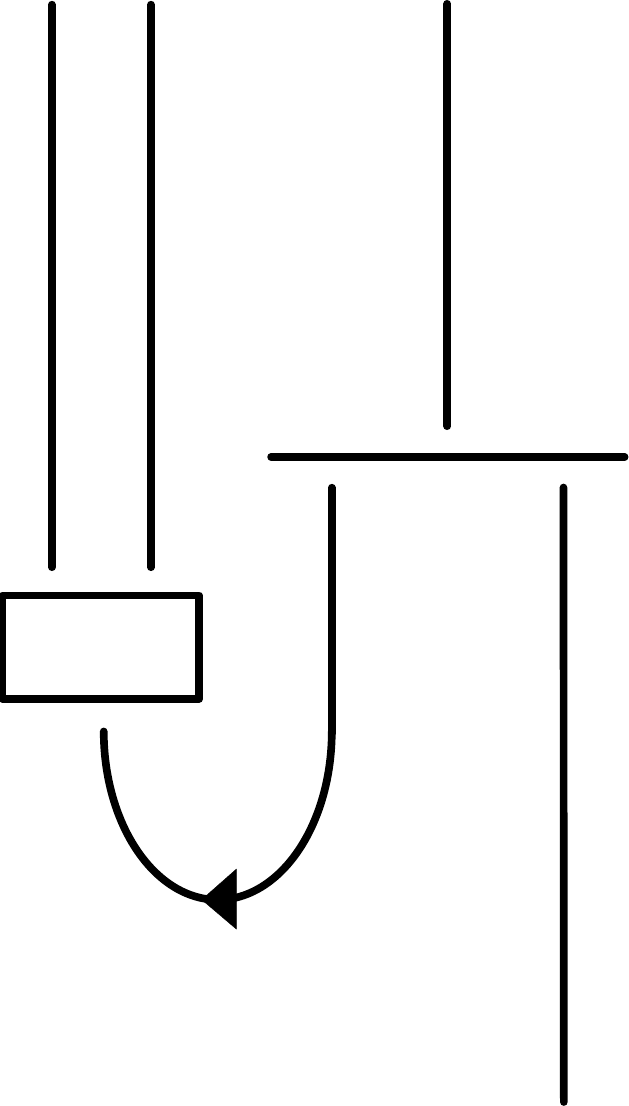}{.3}
	  \put(-88,84){\scriptsize{$X$}} 	 \put(-71,84){\scriptsize{$Y$}}  \put(-30,84){\scriptsize{$\coend$}} 
		\put(-20,20){\scriptsize$\iota_{X\tensor Y}$} \put(-89,-16){\scriptsize$\id_{X\otimes Y}$}	\put(-20,-90){\scriptsize$X\otimes Y$} \gc
\ee
where in the last equality we used the dinaturality property~\eqref{eq:P-iota-pic} of $\iota_{Y\otimes X}$ in order to move the braiding from right to left, and
then we also used the explicit diagram~\eqref{eq:gammaVU} for the isomorphism $\gamma_{Y,X}$, and the zig-zag identity to simplify the diagram. 
We get thus indeed $\natiso_\coend^2(\muc)_{X,Y} = \nattr_{X\otimes Y}$, recall~\eqref{eq:iota-tilde-pic},
and therefore the two multiplications are equal. The unit maps $\eta_\coend$ are compared in a similar way. Finally, for the Hopf pairing $\omega_\coend$ we calculate 
$\natiso_{\one}^2(\omega_{\coend})_{X,Y}$
along the lines in~\eqref{eq:mult-L-A} using RHS of~\eqref{hopf-pair} for the dashed region and simplify it up to RHS of~\eqref{eq:rA-omega}, as claimed.

\newcommand\arxiv[2]      {\href{http://arXiv.org/abs/#1}{#2}}
\newcommand\doi[2]        {\href{http://dx.doi.org/#1}{#2}}
\newcommand\httpurl[2]    {\href{http://#1}{#2}}

\end{document}